\documentclass[a4paper,11pt,leqno]{amsart}

\usepackage{amsmath, amssymb, amsthm, stmaryrd}
\usepackage{euscript}
\usepackage{verbatim}
\usepackage{graphicx}
\usepackage{enumerate}
\numberwithin{equation}{section}

\newtheorem{theorem}{Theorem}[section]
\newtheorem{prop}[theorem]{Proposition}

\newtheorem{lemma}[theorem]{Lemma}

\theoremstyle{definition}
\newtheorem{definition}[theorem]{Definition}

\theoremstyle{remark}
\newtheorem{rk}{Remark}

\newcommand{\Mbar}{\overline{M}}

\newcommand{\Mtil}{\widetilde{M}}

\newcommand{\Omegatil}{\widetilde{\Omega}}

\newcommand{\bR}{\mathbb{R}}
\newcommand{\bB}{\mathbb{B}}
\newcommand{\bS}{\mathbb{S}}

\newcommand{\phat}{\widehat{p}}
\newcommand{\qhat}{\widehat{q}}

\newcommand{\gbar}{\overline{g}}
\newcommand{\ghat}{\widehat{g}}
\newcommand{\gcheck}{\check{g}}

\renewcommand{\hbar}{\overline{h}}
\newcommand{\Gambar}{\overline{\Gamma}}

\def\into{\hookrightarrow}

\newcommand{\phitil}{\widetilde{\phi}}
\newcommand{\xtil}{\tilde{x}}

\newcommand{\vtil}{\widetilde{v}}
\newcommand{\wtil}{\widetilde{w}}
\newcommand{\Deltabar}{\overline{\Delta}}
\newcommand{\nablabar}{\overline{\nabla}}
\newcommand{\Ftil}{\widetilde{F}}

\newcommand{\kulk}{\owedge} 

\DeclareMathOperator{\vol}{Vol}

\newcommand{\diff}  [2]{\frac{d #1}{d #2}}

\newcommand{\pdiff} [2]{\frac{\partial #1}{\partial #2}}

\newcommand{\riem}{\mathcal{R}}

\newcommand{\riemdddd}[4]{\riem_{#1 #2 #3 #4}}
\newcommand{\riemuddd}[4]{\riem^{#1}_{\phantom{#1} #2 #3 #4}}

\newcommand{\riemcons}{\mathcal{K}}

\newcommand{\riembar}{\overline{\riem}}
\newcommand{\riembardddd}[4]{\riembar_{#1 #2 #3 #4}}
\newcommand{\riembaruddd}[4]{\riembar^{#1}_{\phantom{#1} #2 #3 #4}}

\newcommand{\ric}{\mathrm{Ric}}

\newcommand{\ricud}[2]{\ric^{#1}_{\phantom{#1} #2}}

\newcommand{\tdddd}[5]{\mathcal{#1}_{#2 #3 #4 #5}}
\newcommand{\tuddd}[5]{\mathcal{#1}^{#2}_{\phantom{#2} #3 #4 #5}}

\newcommand{\grad}[1]{\nabla_{#1}}

\newcommand{\hess}{\mathrm{Hess}}

\newcommand{\hessdd}[2]{\grad{#1, #2}}

\newcommand{\hessbar}{\overline{\hess}}
\newcommand{\tlhessbar}{\mathring{\overline{\hess}}}
\newcommand{\hessbardd}[2]{{\overline{\nabla}}_{#1, #2}}
\newcommand{\tlhessbardd}[2]{\mathring{\overline{\nabla}}_{#1, #2}}

\newcommand{\sff}{S}

\newcommand{\sffud}[2]{\sff^{#1}_{\phantom{#1} #2}}

\newcommand{\kronecker}[2]{\delta^{#1}_{\phantom{#1}#2}}

\begin{document}
\title[Conformal compactification of ALH metrics]{Conformal compactification of asymptotically locally hyperbolic metrics II: Weakly ALH metrics}
\author[Romain Gicquaud]{Romain Gicquaud}
\date{\today}
\address{Current Address: \newline
Laboratoire de Math\'ematiques et de Physique Th\'eorique \newline
UFR Sciences et Technologie\newline
Facult\'e Fran\c cois Rabelais\newline
Parc de Grandmont \newline
37300 Tours, \newline \newline}
\email{romain.gicquaud@lmpt.univ-tours.fr}

\begin{abstract}
In this paper we pursue the work initiated in \cite{Bahuaud, BahuaudGicquaud}: study the extent to which conformally compact asymptotically hyperbolic metrics can be characterized intrinsically. We show how the decay rate of the sectional curvature to $-1$ controls the H\"older regularity of the compactified metric. To this end, we construct harmonic coordinates that satisfy some Neumann-type condition at infinity. Combined with a new integration argument, this permits us to recover to a large extent our previous result without any decay assumption on the covariant derivatives of the Riemann tensor. We believe that our result is optimal.
\end{abstract}

\maketitle
\tableofcontents

\section{Introduction}\label{secIntro}
% Introduction

Conformally compact Riemannian manifolds have been much studied during the last decades. On the one hand, conformally compact Einstein manifolds turned out to be a very powerful tool in the study of conformal geometry. This is the so called AdS/CFT correspondence, see e.g. \cite{AdS-CFT}, \cite{AndersonReview} or \cite{DjadliGuillarmouHerzlich}. On the other hand, these manifolds are natural Cauchy surfaces in the theory of general relativity. We refer the reader to \cite{GicquaudSakovich} and references therein for further details.\\

The usual definition of these spaces is an extension of the ball model of the hyperbolic space: Let $\Mbar$ be an $(n+1)-$dimensional (smooth) compact manifold with boundary $\partial M$. A smooth non-negative function $\rho : \Mbar \to \bR$ is called a defining function for $\partial M$ if $\rho^{-1}(0) = \partial M$ and $d\rho \neq 0$ everywhere on $\partial M$. A metric $g$ on $M$ is called \emph{conformally compact} if the metric \[\gbar = \rho^2 g\] extends to a regular (at least $C^2$) metric on $\Mbar$. A straightforward calculation using the conformal transformation law of the curvature, see \cite[Theorem 1.159]{Besse}, shows that if the metric $g$ also satisfies \[\left|d\rho\right|_{\gbar}^2 = 1 \text{ on } \partial M,\] the sectional curvature of the metric $g$ tends to $-1$ at infinity: \[\sec_g = -1 + O(\rho).\] Such a metric is called \emph{asymptotically hyperbolic} (AH). See \cite{LeeFredholm} for more details. It can be proved that such manifolds are complete and that for any given point $p \in M$ (or non-empty compact subset $K \subset M$) there exists a constant $C > 0$ such that \[C^{-1} \rho \leq e^{-s} \leq C \rho,\] where $s = d_g(p, .)$ (resp. $s = d_g(K, .)$).\\

Analysis on these spaces, in particular elliptic theory, is now well-developed. Let us mention the works of Mazzeo \cite{MazzeoEllipticTheory}, Graham and Lee \cite{GrahamLee, LeeFredholm} whose underlying ideas play a central role in this work.\\

This definition however involves extrinsic data (the manifold $\Mbar$ and the defining function $\rho$). An important issue is then to give an intrinsic definition, i.e. involving only $M$ and $g$, of asymptotically hyperbolic manifolds and compare it to the usual definition. A similar program has been accomplished in the asymptotically Euclidean (ALE) case by Bando, Kazue and Nakajima in \cite{BandoKazueNakajima} and by Herzlich in \cite{HerzlichCompactification}. Mimicking the definition of an asymptotically locally Euclidean manifold, we define \emph{an asymptotically locally hyperbolic (ALH) manifold} to be a complete non-compact Riemannian manifold whose curvature satisfies \[\sec_g = -1 + O(e^{-a s}),\] for some constant $a > 0$ which will be called the \emph{order} of $(M, g)$. Here $s$ is the distance for the metric $g$ with respect to a given point or compact subset (see also Definition \ref{defALHManifold}). As in the ALE case, one wants to keep only ``nice'' geometries at infinity. While in the ALE case, this can be achieved by requiring some maximal volume growth condition, see \cite{BandoKazueNakajima}, we will assume that the manifold $(M, g)$ admits an \emph{essential subset} $K$. This notion has been introduced in \cite{BahuaudMarsh}. Roughly speaking, an essential subset is a non-empty compact convex subset such that $\sec_g(M \setminus \mathring{K}) < 0$ and $M \setminus \mathring{K}$ is diffeomorphic to $\partial K \times [0; \infty)$ via the outgoing exponential map, see Definition \ref{defEssentialSubset}. Among hyperbolic manifolds, only the convex-cocompact ones satisfy this assumption and a simple convexity argument shows that Cartan-Hadamard and conformally compact manifolds admit essential subsets (see e.g. \cite{GicquaudThesis}).\\

Prior to the study of asymptotically locally hyperbolic manifolds is the study of negatively curved manifolds. A first result on the compactification of these manifolds is due to Anderson and Schoen in \cite{AndersonSchoen}. Let $(M^{n+1}, g)$ be a complete simply connected manifold whose curvature is negatively pinched: \[-b^2 \leq \sec_g \leq -a^2\] everywhere for some constants $0 < a < b$. Then the Cartan-Hadamard theorem \cite[Chapter 6, Theorem 3.3]{Petersen} asserts that the exponential map $\exp_p$ from any given point $p \in M$ is a diffeomorphism from $T_p M \simeq \bR^{n+1}$ onto $M$. Defining some well chosen rescaled coordinates, one can make $M$ diffeomorphic to the interior of the open unit ball $\mathring{B} \subset \bR^{n+1}$. Using some arguments of comparison geometry, Anderson and Schoen were able to prove that changing the point $p$ under consideration leads to $C^\alpha$-compatible charts for the sphere at infinity $\bS^n = \partial \overline{B}$, where $\alpha = \frac{a}{b}$. They used this construction to give a classification of positive harmonic functions on $M$. Their work was extended and generalized to manifolds whose curvature is negatively pinched only outside some essential subset by Bahuaud and Marsh in \cite{BahuaudMarsh}: changing the essential subset leads to $C^\alpha$-coordinate transition functions for the manifold $\Mbar = M \bigcup M(\infty)$, where $M(\infty)$ is the boundary at infinity, see Definition \ref{defBoundaryInfinity}, which can be attached in a similar manner to $M$ by using some rescaled coordinate functions. Generalizing the work of Anderson and Schoen, it can be shown that $M(\infty)$ coincide with the Martin boundary of $M$.\\

With these results at hand, Bahuaud and the author initiated the study of the conformal compactification of ALH manifolds in \cite{Bahuaud, BahuaudGicquaud}. We proved that if $(M, g)$ is ALH of order $a \in (0; 2)$, $a \neq 1$, then, setting $\rho = e^{- s}$ where $s$ denotes the distance function from the essential subset, the metric $\gbar = \rho^2 g$ extends to a continuous metric on $\Mbar$. If further the covariant derivatives of the Riemann tensor satisfy $\left|\nabla \riem\right|_g = O(e^{-a s}), \left|\nabla^{(2)} \riem\right|_g = O(e^{-a s})$, then the metric $\gbar$ is $C^{0, a}$ if $a \in (0; 1)$ and $C^{1, a-1}$ if $a \in (1; 2)$. The strategy of the proof is to study the following system of ODE:
\begin{equation}\label{eqRiccatiSystem}
\left\lbrace
\begin{aligned}
\partial_0 \sffud{i}{j} + \sffud{i}{k} \sffud{k}{j} & = - \riemuddd{i}{0}{j}{0}\\
\partial_0 g_{ij} & = 2 g_{ik} \sffud{k}{j},
\end{aligned}
\right.
\end{equation}
satisfied by the metric in Fermi coordinates $\{s, x^1, \ldots, x^n\}$, where Latin indices go from $1$ to $n$, $0$ denotes the $s$-component and $S$ is the second fundamental form (more exactly the Weingarten map) of the hypersurfaces of constant $s$, see \cite[Chapter 2, Proposition 4.1]{Petersen}. Regularity is proved by taking tangential derivatives of this system, this is where the assumptions on the covariant derivatives of the Riemann tensor arise.\\

As is well known, the Riemann tensor locally controls the metric in $W^{2, p}$-norm for any $p \in (1; \infty)$ in some well chosen coordinates, e.g. harmonic coordinates, see \cite{DTK}. This remark leads to the belief that the assumptions on the covariant derivatives of the Riemann tensor are superfluous to some extent. Results in this direction were first proved by Hu, Qing and Shi in their nice article \cite{HuQingShi}. Their result are divided in two parts:

\begin{enumerate}
\item If $(M, g)$ is ALH of order $a \in (0; 1)$ the metric $\gbar = \rho^2 g$ extends to a $C^{0, \mu}$-metric on $\Mbar$ where $\mu = \frac{2}{3} a$. The main idea of the proof is to make the metric $g$ evolve under some modified Ricci flow thus getting a family of metrics $g(t)$, then to recover by some Shi-type estimates (see e.g. \cite[Theorem 6.9]{HamiltonRicci}) the assumption on the covariant derivatives of the Riemann tensor of the metrics $g(t)$ and to apply the results of \cite{Bahuaud, BahuaudGicquaud}. The H\"older regularity of the metric $\gbar(t)$ becomes worse and worse as $t$ tends to zero but at a controlled pace. Hence by an argument similar to Lemma \ref{lmBridge}, this implies regularity for the metric $\gbar$ itself.
\item If $a \in (1; 2)$, the authors construct harmonic coordinates for the metric $\gbar$ and get that if $a > 2 - \frac{1}{n+1}$ then one can define harmonic coordinates for the metric $\gbar$, so the metric $\gbar$ belongs to some $W^{2, p}$-space in this coordinate system. The assumption $a > 2 - \frac{1}{n+1}$ comes from the necessity that the Riemann tensor of $\gbar$ belongs to some $L^p$-space, $p \in (n+1; \infty)$.
\end{enumerate}

However, in view of \cite{Bahuaud, BahuaudGicquaud}, these results are not expected to be optimal. An intuitive reason why the above method does not lead to optimal results in the case $a \in (1; 2)$ is that the blow-up of the Riemann tensor occurs all along the hypersurface at infinity while tools in elliptic PDE (elliptic regularity, Sobolev injections,...) are limited in some sense by point singularities. The natural way to solve this apparent conflict is not to work with the metric $\gbar$ but instead to apply elliptic theory to the metric $g$ itself! This is the method we will follow in this article. We will concentrate on the case $0 < a < 2$, the optimal result in the case $a > 2$ is obtained in \cite{HuQingShi}: When $a > 2$ the Riemann tensor of the metric $\gbar$ is bounded, its derivatives can still blow up but this converts into H\"older regularity for the Riemann tensor of $\gbar$ up to the boundary. The main result we prove is the following:

\begin{theorem}\label{thmMain} Suppose $(M,g)$ is an asymptotically locally hyperbolic manifold of order $a > 0$ and $K \subset M$ is an essential subset. Then there exists a unique function $t$ such that $t - e^s = O(e^{(1-a)s})$, where $s = d_g(K, .)$, and $\Delta t = (n+1) t$. Furthermore,
\begin{itemize}
\item if $0 < a < 1$, there exists an atlas on $\Mbar = M \bigcup M(\infty)$ such that the metric $\gbar = t^{-2} g$ extends to a $C^{0, a}$-smooth metric up to the boundary,
\item if $1 < a < 2$, there exists an atlas on $\Mbar$ such that the metric $\gbar = t^{-2} g$ extends to a $C^{1, \mu}$-smooth metric up to the boundary, for any $\mu \in (0; a-1)$.
\end{itemize}
\end{theorem}

The proof of this theorem contains several new ideas. We shall first show the existence of the function $t$ and construct harmonic charts in a neighborhood of each point at infinity for the metric $g$. From the work of Anderson and Schoen \cite{AndersonDirichlet, AndersonSchoen} it is known that bounded harmonic functions are in bijection with $L^\infty$ functions on $M(\infty)$. An issue that immediately appears is to choose the functions on $M(\infty)$ to use as Dirichlet data at infinity for the coordinate functions. This issue could in principle be bypassed by constructing functions that satisfy a Neumann condition at infinity, i.e. such that their derivative along a certain set of unbounded geodesics decays fast. In the classical context of conformally compact AH manifolds, it can be seen that such functions induce harmonic coordinates on the boundary at infinity. This is the method we will pursue in this article: we first prove a first regularity result for the metric $\gbar = t^{-2} g$ sufficient to construct harmonic charts on $M(\infty)$ and show how to extend them to harmonic coordinates on $M$ satisfying a Neumann condition at infinity. Then we use the Neumann condition at infinity to show that these functions are actually ``nice'' functions on $\Mbar$ (meaning e.g. that for such a function $\phi$, $|d\phi|_g = O(t^{-1})$ so $|d\phi|_{\gbar} = O(1)$). This will be carried out by two integration arguments (Lemmas \ref{lmMovingWindow1} and \ref{lmMovingWindow2}). An important tool in this study will be the local Sobolev spaces introduced by Sakovich and the author in \cite{GicquaudSakovich}.\\

The outline of this paper is as follows. Section \ref{secPrelim} contains all the basic tools we will need in this paper: basic definitions, construction of the function $t$ which will serve both as a conformal factor and a coordinate (Proposition \ref{lmConstructionT}), two integration lemmas (Lemmas \ref{lmMovingWindow1} and \ref{lmMovingWindow2}) and a lemma to convert estimates with respect to the metric $g$ to regularity for the metric $\gbar$ (Lemma \ref{lmBridge}). Section \ref{secAlt1} is devoted to the proof of the first part of Theorem \ref{thmMain} which will also be used in Section \ref{secHarm} to construct harmonic coordinates charts satisfying a Neumann condition at infinity. We also prove Proposition \ref{propBijection} showing how to glue $M(\infty)$ to $M$. Section \ref{secHarm} is divided in two parts. First we construct harmonic coordinates on $M(\infty)$, then we extend them to harmonic charts on $M$. In Section \ref{secEstimates}, we prove estimates for the derivatives of the coordinate functions. Finally in Section \ref{secBoundaryReg}, we prove Theorem \ref{thmMain} in the case $1 < a < 2$.\\

\noindent\textbf{Acknowledgments:} The author is grateful to Eric Bahuaud and Julien Cortier for useful comments on preliminary versions of this article. The author benefited from the experience of Gilles Carron, Raffe Mazzeo and Laurent V\'eron. The author also thanks Mattias Dahl, Erwann Delay, Marc Herzlich, Xue Hu, Jie Qing and Yuguang Shi for their interest in this work. Last but not least, the author is indebted to Anna Sakovich for her careful proofreading of a preliminary version of this article.

\section{Preliminaries}\label{secPrelim}
% Preliminary results

In this section we present some preliminary definitions and results. Let us first begin by some notations:

\subsection{Notation} In what follows we will always assume given a complete non-compact Riemannian manifold $(M, g)$ of dimension $n+1$. We use Latin letters to denote indices going from $0$ to $n$ and Greek letters for indices from $1$ to $n$. Unless otherwise stated, we use the Einstein summation convention. Let $T$ be an arbitrary tensor, then $T$ is dominated by some function $f : M \to \bR$ ($T = O(f)$) if \[\left|T\right|_g = O(f).\] However if $T$ is for example a covariant 2-tensor, the notation $T_{\mu\nu} = O(f)$ means that the component $T_{\mu\nu}$ of the tensor $T$ is dominated by $f$. Our convention for the curvature tensor is the following: For any vector field $X$ \[\nabla_i \nabla_j X^k - \nabla_j \nabla_i X^k = \riemuddd{k}{l}{i}{j} X^l.\] This convention is the opposite of \cite{Bahuaud, BahuaudGicquaud}. We denote by $\riemcons$ the constant $+1$ curvature tensor \[\tdddd{K}{i}{j}{k}{l} = g_{ik} g_{jl} - g_{ik} g_{jl},\] and by $\mathcal{E}$ the difference between the curvature tensor and the constant $-1$ curvature tensor: \[\mathcal{E} = \riem + \riemcons.\] For an arbitrary symmetric 2-tensor $T$ we denote by $\mathring{T}$ the traceless part of $T$: \[\mathring{T}_{ij} = T_{ij} - \frac{g^{kl}T_{kl}}{n+1} g_{ij}.\] Remark that the traceless part of $T$ does not depend on the metric $g$ in a given conformal class.\\

We will be dealing with two different metrics: $g$ and $\gbar$ and with their associated function spaces. To distinguish if we are using the function spaces associated to $g$ or $\gbar$, we will sometime indicate the metric. When the metric is not indicated, the rule is the following: weighted function spaces always refer to spaces defined using $g$ while unweighted function spaces defined on a subset with non-empty intersection with the boundary at infinity $\partial M = M(\infty)$ will always denote function spaces defined with respect to the metric $\gbar$.

\subsection{Basic definitions}\label{secMainDef}

Before entering the details of the proof of Theorem \ref{thmMain}, we start by giving a precise definition of certain terms already appearing in the introduction. We introduce the following three definitions:

\begin{definition}[\cite{ShiTian}]\label{defALHManifold}
Let $(M, g)$ be a complete non-compact manifold, $(M, g)$ is called \emph{asymptotically locally hyperbolic} (ALH) of order $a > 0$ if the sectional curvature of the metric $g$ satisfies 
\begin{equation} \label{AH0} \tag{ALH}
|\riem + \riemcons|_g = O(e^{-as}),
\end{equation}
or equivalently if \[\sec_g + 1 = O(e^{-a s}),\] where $s$ is the distance with respect to the metric $g$ from any given point or non-empty compact subset of $M$.
\end{definition}

\begin{definition}[\cite{BahuaudMarsh}]\label{defEssentialSubset}
A non-empty subset $K \subset M$ is called \emph{essential} if it satisfies the following assumptions:
\begin{enumerate}
\item $K$ is a compact submanifold of codimension 0 of $M$ with smooth boundary $Y,$
\item \begin{equation} \label{NSC} \tag{NSC} \sec(  M \setminus \mathring{K} ) < 0, \end{equation}
\item $K$ is totally convex, i.e. if $\gamma : [a; b] \to M$ is a geodesic with $\gamma(a), \gamma(b) \in K$, then $\gamma([a; b]) \subset K$.
\end{enumerate}
\end{definition}

\begin{definition}[\cite{EberleinONeill}]\label{defBoundaryInfinity}
Let $(M, g)$ be a complete non-compact Riemannian manifold. The \emph{boundary at infinity} $M(\infty)$ of $(M, g)$ is the set of equivalence classes of unbounded geodesic rays $\gamma: [0; \infty) \to M$ having unit speed, $|\dot{\gamma}|_g(t) = 1$, under the relation \[\gamma_1 \sim \gamma_2 \text{ iff } d_g(\gamma_1(t), \gamma_2(t)) \text{ is bounded when } t \to \infty.\]
\end{definition}

In \cite{BahuaudMarsh}, Bahuaud and Marsh proved the following result:

\begin{prop}[\cite{BahuaudMarsh}]\label{propBijection} Assume that $(M, g)$ admits an essential subset $K$. Denote $Y = \partial K$ and $N_+ Y$ the outgoing normal bundle of $Y$. Then the exponential map induces a diffeomorphism \[N_+ Y \simeq M \setminus \mathring{K}.\] Furthermore, the map $y \mapsto \sigma_y$, where $\sigma_y$ denotes the outgoing unitary geodesic normal to $Y$ starting at $y$, defines a bijection $Y \simeq M(\infty)$.
\end{prop}

We now introduce a new class of function spaces which are the ALH analogs of the one introduced in \cite{GicquaudSakovich}:

\begin{definition}[Local Sobolev spaces]\label{defLocalSobolev}
Let $E$ be a geometric tensor bundle over $M$, $\Omega \subset M$ an open subset, $p \in (1; \infty)$, $k$ a non negative integer and $\delta \in \bR$. We define the $X^{k, p}_\delta(\Omega, E)$ function space as the set of sections $u \in W^{k, p}_{loc}(\Omega, E)$ such that the norm \[\left\|u \right\|_{X^{k, p}_\delta(\Omega, E)} = \sup_{x \in M} e^{\delta s} \left\|u \right\|_{W^{k, p}(B_1(x) \cap \Omega, E)}\] is finite.
\end{definition}

\begin{rk}
If we choose a different radius for the ball appearing in the definition of the $X^{k, p}_\delta$-norm, then we get an equivalent norm.
\end{rk}

\subsection{Codazzi-type equations}\label{secCodazzi}

In the rest of the article we shall say that an equation of the form \[\nabla_i T_{jk} - \nabla_j T_{ik} = f_{ijk}\] for a symmetric 2-tensor $T$ is a Codazzi (or of Codazzi type) equation. This kind of equations usually appears in the study of hypersurfaces, see e.g. \cite[Chapter 2, Theorem 3.8]{Petersen}, but also shows up in different contexts in Riemannian geometry. The one that will interest us most in this article is the following: if $T_{ij} = \hessdd{i}{j} \phi$ is the Hessian of a certain function $\phi$, then $T$ satisfies the following equation: \[\nabla_i T_{jk} - \nabla_j T_{ik} = - \riemuddd{l}{k}{i}{j} \nabla_l \phi.\] This equation is not elliptic in general. But in the case where $T$ is traceless, in particular if $T$ is the Hessian of a harmonic function, it is easily shown that the principal symbol of this equation is injective. As a consequence, the following proposition holds:

\begin{prop}[Elliptic regularity for the Codazzi equation]\label{propEllipticRegCodazzi}
Let $\Omega \subset \bR^{n+1}$ be a non-empty bounded open subset. Let $g$ be a $W^{1, p}$-metric on $\Omega$ for some $p \in (n+1; \infty)$, $g$ uniformly equivalent to the Euclidean metric on $\Omega$. For any $\Omega' \subset\subset \Omega$ and any $q \in \left(\frac{p}{p-1}; p\right]$ there exists a constant $C$ (depending only on $\Omega$, $\Omega'$, $p$, $q$, $\left\|g\right\|_{W^{1, p}(\Omega)}$ and $\left\|g^{-1}\right\|_{W^{1, p}(\Omega)}$) such that for any traceless symmetric 2-tensor $T \in L^q(\Omega)$ satisfying \[\nabla_i T_{jk} - \nabla_j T_{ik} = f_{ijk}\] where $f \in L^q(\Omega)$, then $T \in W^{1, q}(\Omega')$ and \[\left\|T\right\|_{W^{1, q}(\Omega', g)} \leq C \left(\left\|f\right\|_{L^q(\Omega, g)} + \left\|T\right\|_{L^q(\Omega, g)}\right).\]
\end{prop}

\begin{proof} In this proof, we denote $(x^0, x^1, \ldots, x^n)$ the canonical coordinates on $\bR^{n+1}$. We concentrate first on the case $g = \delta$ the Euclidean metric. The Codazzi equation reads \[\partial_i T_{jk} - \partial_j T_{ik} = f_{ijk}.\] Tracing this equation over $i$ and $k$ and using the fact that $T$ is traceless, we get \[\partial^i T_{ij} = f^i_{\phantom{i}ji}.\] Taking the divergence of the Codazzi equation with respect to the $i$-index, we obtain \[\Delta T_{jk} - \partial^i \partial_j T_{ik} = \partial^i f_{ijk}.\] Combining the previous two formulas, we get
\[\Delta T_{jk} = \partial^i f_{ijk} + \partial_j f^i_{\phantom{i}ki}.\] The proposition follows then from standard elliptic theory.\\

We now consider the general case and apply a zoom process analogous to \cite[Theorem 9.11]{GilbargTrudinger}. We assume first that $T \in W^{1, q}_{loc}(\Omega)$. Select $P \in \overline{\Omega'}$. Without loss of generality, we can assume $P=0$ and, by a linear change of variables, that $g_{ij}(P) = \delta_{ij}$. Let $\lambda \in \bR^*_+$ be large enough so that $B_{\frac{2}{\lambda}}(0) \subset \Omega$ and set $y = \lambda x$. In what follows, indices with a prime correspond to components of tensors with respect to the $y$-coordinate system:
\[\partial_{i'} = \pdiff{~}{y^i} = \frac{1}{\lambda} \pdiff{~}{x^i}\] and set $g' = \lambda^2 g$ so that $g'_{i'j'}(0) = \delta_{i'j'}$ where $\delta_{i'j'} = 1$ if $i'=j'$ and $0$ otherwise. In this coordinate system, the Codazzi equation reads
\[\partial_{i'} T_{j'k'} - \partial_{j'} T_{i'k'} = f_{i'j'k'} + \Gamma^{l'}_{i'k'} T_{j'l'} - \Gamma^{l'}_{j'k'} T_{i'l'}.\]
Let $\chi \in C^\infty_c(B_2(0))$ be a cut-off function such that $0 \leq \chi \leq 1$, $\chi = 1$ on $B_1(0)$. Set $\phi = \frac{1}{n+1}\delta^{i'j'} T_{i'j'} = \frac{1}{n+1} \left(\delta^{i'j'} - g'^{i'j'}\right) T_{i'j'}$ and $U_{i'j'} = \chi(y) \left(T_{i'j'} - \phi \delta_{i'j'}\right)$. $U$ is a symmetric traceless 2-tensor for the metric $\delta_{i'j'}$. $U$ satisfies the following equation:
\begin{eqnarray*}
\partial_{i'} U_{j'k'} - \partial_{j'} U_{i'k'}
  & = & \chi \left(f_{i'j'k'} + \Gamma^{l'}_{i'k'} T_{j'l'} - \Gamma^{l'}_{j'k'} T_{i'l'}\right)\\
  & & - \delta_{j'k'} \partial_{i'}(\chi \phi) + \delta_{i'k'} \partial_{j'} (\chi \phi) + T_{j'k'} \partial_{i'} \chi - T_{i'k'} \partial_{j'} \chi.
\end{eqnarray*}

Hence, in the $y$-coordinate system, there exists a constant $C > 0$ such that
\[\left\|U\right\|_{W^{1, q}(B_2)} \leq C \left(\left\|\chi f\right\|_{L^q(B_2)} + \left\|\Gamma \chi T\right\|_{L^q(B_2)} + \left\|\chi \phi\right\|_{W^{1, q}(B_2)} + \left\|T\right\|_{L^q(B_2)}\right),\] (we used the fact that $\|U\|_{L^p} \leq \|T\|_{L^p}$). In this estimate, all the norms are with respect to the $y$-coordinate system and the $y$-components. We claim that $\left\|\Gamma \chi T\right\|_{L^q(B_2)} \leq C \|\Gamma\|_{L^p(B_2)} \|\chi T\|_{W^{1, q}(B_2)}$. Indeed if $q < n+1$, from the Sobolev injection, we have that $\Gamma \chi T \in L^r(B_2)$ with $\frac{1}{r} = \frac{1}{p} + \frac{1}{q} - \frac{1}{n+1} < \frac{1}{q}$ while if $q > n+1$, $\Gamma \chi T \in L^p \subset L^q$. Since $\|\chi T\|_{W^{1, q}} \leq C \left(\|U\|_{W^{1, q}} + \|\chi \phi\|_{W^{1, q}}\right)$, we get
\begin{multline*}
\left\|\chi T\right\|_{W^{1, q}(B_2)} \leq C \left(\left\|\chi f\right\|_{L^q(B_2)} + \left\|\Gamma\right\|_{L^p(B_2)} \left\|\chi T\right\|_{W^{1, q}(B_2)}\right.\\ \left.+ \left\|\chi \phi\right\|_{W^{1, q}(B_2)} + \left\|T\right\|_{L^q(B_2)}\right), 
\end{multline*}
where we also used the Sobolev injection.\\

We now claim that \[\|g' - \delta\|_{W^{1, p}(B_2)} = O\left(\frac{1}{\lambda^{1-\frac{n+1}{p}}}\right).\] Indeed
\[\begin{aligned}
\left\|\partial' g'\right\|_{L^p(B_2)}
  & = \left(\int_{B_2} \sum_{i',j',k'} \left|\partial_{i'} g'_{j'k'}\right|^p dy\right)^\frac{1}{p}\\
  & = \left(\lambda^{n+1} \int_{B_{\frac{2}{\lambda}}} \sum_{i,j,k} \left|\lambda^{-1}\partial_i g_{jk}\right|^p dx\right)^\frac{1}{p}\\
  & \leq \frac{1}{\lambda^{1-\frac{n+1}{p}}} \left(\int_{B_{2}} \sum_{i,j,k} \left|\partial_i g_{jk}\right|^p dx\right)^\frac{1}{p}.
\end{aligned}\]

From \cite[Theorem 7.17]{GilbargTrudinger}, for any $y \in B_2$,
\[\left|g'(y) - \delta\right| = \left|g'(y) - g'(0)\right| \leq C \left\|\partial' g'\right\|_{L^p} = O\left(\frac{1}{\lambda^{1-\frac{n+1}{p}}}\right).\]

From the expression of the Christoffel symbols, we get that for any triple $(i', j', k')$, \[\left\|\Gamma_{i'j'}^{k'}\right\|_{L^p(B_2)} = O\left(\frac{1}{\lambda^{1-\frac{n+1}{p}}}\right).\] Similarly if $\lambda$ is large enough so that $\frac{1}{2} \delta \leq g' \leq 2 \delta$ on $B_2$, \[\left\|\chi \phi\right\|_{W^{1, q}(B_2)} \leq C \left\|g - \delta\right\|_{W^{1, p}(B_2)} \left\|\chi T\right\|_{W^{1, q}(B_2)} \leq \frac{C}{\lambda^{1-\frac{n+1}{p}}} \left\|\chi T\right\|_{W^{1, q}(B_2)}.\]

As a consequence, we proved that, if $\lambda$ is large enough,
\[\left\|\chi T\right\|_{W^{1, q}(B_2)} \leq C \left[\left\|\chi f\right\|_{L^q(B_2)} + \left\|T\right\|_{L^q(B_2)} + \frac{1}{\lambda^{1-\frac{n+1}{p}}} \left\|\chi T\right\|_{W^{1, q}(B_2)}\right].\] Thus if $\lambda$ is such that $\lambda^{1-\frac{n+1}{p}} > 2 C$, \[\left\|T\right\|_{W^{1, q}(B_1)} \leq \left\|\chi T\right\|_{W^{1, q}(B_2)} \leq 2 C \left(\left\|\chi f\right\|_{L^q(B_2)} + \left\|T\right\|_{L^q(B_2)}\right).\] Returning to the $x$-coordinate system we have proved that for any $P \in \Omega$, there exists a $\lambda > 0$ and a constant $C$ such that \[\left\|T\right\|_{W^{1, q}(B_{\frac{1}{\lambda}})} \leq C \left(\left\|f\right\|_{L^q(B_{\frac{2}{\lambda}})} + \left\|T\right\|_{L^p(B_{\frac{2}{\lambda}})}\right).\] The proof of the estimate now follows by covering $\overline{\Omega'}$ by a finite number of such ellipsoids (since we used a linear change of variable to assume that $g_{ij}(P) = \delta_{ij}$).\\

We now remove the assumption $T \in W^{1, q}_{loc}$ and assume only that $T \in L^q$. The proof is based on a mollification argument, see e.g. \cite[Section 7.2]{GilbargTrudinger}. Let $\rho : \bR^{n+1} \to \bR$ be a smooth non-negative function supported in $B_1(0)$ such that $\int_{\bR^{n+1}} \rho(x)dx = 1$. For any function $u \in L^q$ vanishing outside some compact subset $K \subset \Omega$ and $h > 0$, we set \[u^h(x) = \frac{1}{h^{n+1}} \int_{\bR^{n+1}} \rho\left(\frac{x-y}{h}\right) u(y) dy.\]

We denote $r = \frac{1}{2} d(\Omega', \bR^{n+1}\setminus \Omega)$, where $d$ is to be understood as the Euclidean distance. Let $h > 0$ be such that $h < r$. Then the functions $T_{ij}^h, g_{ij}^h, \ldots$ make sense on $\Omega'' = \{x \in \bR^{n+1} | d(x, \Omega') < r\}$. Upon diminishing $h$ if necessary, we can assume that all the metrics $g^h$ are uniformly equivalent to the Euclidean metric on $\Omega''$. We can assume without loss of generality that $T$ is compactly supported in $\Omega''$. Indeed, let $\psi \in C^\infty_c(\Omega'')$ be a smooth function such that $0 \leq \psi \leq 1$ and $\psi = 1$ on $\Omega'$, then $T' = \psi T$ satisfies \[\nabla_i T'_{jk} - \nabla_j T'_{ik} = f_{ijk} + (\partial_i\psi) T_{jk} - (\partial_j\psi) T_{ik} \in L^q(\Omega),\] thus if the proposition is proved for compactly supported functions, since $T' = \psi T$ on $\Omega'$ we conclude that $T \in W^{1, q}(\Omega')$.\\

We let $\phi^h = \frac{1}{n+1} \left(g^h\right)^{ij} T^h_{ij}$, where $\left(g^h\right)^{ij}$ is the inverse matrix of $g^h_{ij}$ and define $\mathring{T}_{ij}^h = T^h_{ij} - \phi^h g^h_{ij}$. $\mathring{T}_{ij}^h$ satisfies the following equation:
\[\partial_i \mathring{T}_{jk}^h - \partial_j \mathring{T}_{ik}^h = f^h_{ijk} + \left(\Gamma^{l}_{ik} T_{jl} - \Gamma^{l}_{jk} T_{il} \right)^h - \left[\partial_i(\phi^h g^h_{jk}) - \partial_j(\phi^h g^h_{ik})\right].\]

Let $r$ be such that $\frac{1}{r} = \frac{1}{p} + \frac{1}{q} < 1$. We find that
\[\begin{aligned}
\left\|\left(\Gamma^{l}_{ik} T_{jl} - \Gamma^{l}_{jk} T_{il} \right)^h\right\|_{L^r(\Omega'')}
  & \leq \left\|\Gamma^{l}_{ik} T_{jl} - \Gamma^{l}_{jk} T_{il}\right\|_{L^r(\Omega)}\\
  & \leq C \left\|\Gamma\right\|_{L^p(\Omega)} \left\|T\right\|_{L^q(\Omega)},\\
\left\|f^h\right\|_{L^r(\Omega'')}
  & \leq \left\|f\right\|_{L^r(\Omega)} \leq C \left\|f\right\|_{L^q(\Omega)}.
\end{aligned}\]

By a proof entirely similar to the previous one, we find that this implies that
\[\left\|T^h\right\|_{W^{1, r}(\Omega')} \leq C \left(\left\|\Gamma\right\|_{L^p(\Omega)} \left\|T\right\|_{L^q(\Omega)} + \left\|f\right\|_{L^q(\Omega)} + \left\|T\right\|_{L^q(\Omega)}\right)\] for some constant $C$ depending only on $\Omega$, $\Omega'$, $p$, $q$, $\left\|g\right\|_{W^{1, p}(\Omega)}$ and $\left\|g^{-1}\right\|_{W^{1, p}(\Omega)}$. We now let $h$ tend to zero. Then $T^h \to T$ strongly in $L^q$ and, since $W^{1, r}(\Omega')$ is reflexive, there exists a sequence $h_i \to 0$ such that $(T^{h_i})_i$ converges weakly in $W^{1, r}(\Omega')$. Hence, $T = \mathrm{weak-lim}~T^{h_i} \in W^{1, r}(\Omega')$. From the Sobolev embedding theorem, $T \in L^{r^*}$ where $r^* = \infty$ if $r > n+1$, $\frac{1}{r^*} = \frac{1}{r} - \frac{1}{n+1} = \frac{1}{q} + \left(\frac{1}{p} - \frac{1}{n+1}\right) < \frac{1}{q}$ if $r < n+1$. Therefore, by a bootstrap argument, we can prove that $T \in W^{1, q}(\Omega')$.
\end{proof}

\subsection{Construction of a new conformal factor}\label{secConstrT}

In earlier work, Bahuaud and the author employed the function $e^{-s}$ as a defining function for $M(\infty)$. This idea was also employed in \cite{HuQingShi}. The main new idea we shall exploit here is to remark that the function $e^s$ approximately satisfies a certain PDE: \[\Delta e^s = (H+1) e^s \simeq (n+1) e^s,\] where $H = n + O(e^{-a s})$ is the mean curvature of the level sets of $s$ (see Proposition \ref{propComparison}). We replace it here by another function $t$ that satisfies exactly this equation. This permits to get more precise informations on the metric $g$. Note that similar eigenfunctions for the Laplacian have been previously considered in other contexts. See \cite{LeeSpectrum}, \cite{Qing, BoniniMiaoQing} or \cite{MerzlichMassFormulae}. In this subsection, we prove the existence of $t$:

\begin{lemma}[Radial coordinate]\label{lmConstructionT}
There exists a unique function $t : M \to \bR$ such that
\[\left\lbrace\begin{aligned}
\Delta t & = (n+1) t\\
       t & = e^s + O(e^{(1-a)s}).
\end{aligned}\right.\]
Furthermore $t - e^s \in X^{2, p}_{a-1}$, $T := \hess(t) - t g \in X^{1, p}_{a-1}$ for any $p \in (n+1; \infty)$.
\end{lemma}

The proof consists in several steps contained in the next subsections. We denote $Y_s$ the hypersurfaces of constant $s$.

\subsubsection{Estimates in Fermi charts}
In a Fermi coordinate chart, the metric and shape operator satisfy the system of differential equations \eqref{eqRiccatiSystem} that we refer to as the Riccati system (see e.g. \cite[Chapter 2, Proposition 4.1]{Petersen}). We prove estimates for the shape operator by analyzing the Riccati equation. We begin by considering the following scalar differential equation:
\begin{equation*} 
\left\{\begin{aligned}
\partial_s \lambda + \lambda^2 &= 1 + O(e^{-as}),\\
\lambda(0) & > 0.
\end{aligned} \right.
\end{equation*}

The following lemma has already been proven in \cite[Lemma 2.1]{BahuaudGicquaud} and \cite[Lemma 2.1]{HuQingShi} (see also \cite{ShiTian} and \cite{Bahuaud}):

\begin{lemma} \label{lmRiccatiEstimate}  Suppose that $f \in L^\infty( [0;\infty) )$ such that there exist constants $\epsilon > 0$ and $J > 0$ with
\[ \left\{\begin{aligned}
 f &> \epsilon~\text{a.e.},\\
|f(s) - 1| &\leq J e^{-as}~\text{a.e.},
\end{aligned}\right. \]
where $a \in (0; 2)$. Suppose further that $\lambda$ is a solution of the Riccati equation

\begin{equation}
\begin{aligned}
\lambda' + \lambda^2 &= f(s), \text{and } \nonumber \\
\lambda(0) & > 0.
\end{aligned}
\end{equation}

Then $\lambda$ is a positive Lipschitz function such that, for a positive constant $C = C(a, J, \lambda(0))$, \[| \lambda - 1 | \leq C e^{-as},\] for all $s > 0$.
\end{lemma}

The proof of the following proposition follows the same method as its analogues in \cite{Bahuaud, BahuaudGicquaud} using Lemma \ref{lmRiccatiEstimate} for the basic scalar estimate. See \cite{BahuaudThesis} or \cite[Chapter 6]{Petersen} for more details.

\begin{prop}[Asymptotic for the second fundamental form, \cite{BahuaudGicquaud}]\label{propComparison}  Given curvature assumptions \eqref{NSC} and \eqref{AH0} for some $a \in (0; 2)$, let $(y^{\beta}, s)$ be Fermi coordinates for $Y$ on $W \times [0,\infty)$ for an open set $W \subset Y$. There exists a positive constant $C$ such that we have
\[(1 - C e^{-as})  \; \kronecker{\beta}{\alpha} \; \leq \; \sffud{\beta}{\alpha}(y, s) \; \leq \; (1 + C e^{-as}) \; \kronecker{\beta}{\alpha}.\]
\end{prop}

\subsubsection{Harmonic coordinates, harmonic radius and applications}
We recall further results from \cite{BahuaudGicquaud}. Given an arbitrary smooth $(n+1)$-manifold $M$, $x \in M$, $Q > 1$, $k \in \mathbb{N}$ and $\alpha \in (0, 1)$, the $C^{k,\alpha}_Q$-\textit{harmonic radius} is the largest radius $r_H=r_H(Q,k,\alpha)(x)$ such that on the geodesic ball $B_x(r_H)$ centered at $x$ with radius $r_H$, there exist harmonic coordinates in which the metric is $C^{k,\alpha}_Q$-controlled:

\begin{enumerate}
\item $Q^{-1} \delta_{ij} \leq g_{ij} \leq Q \delta_{ij}$
\item $\sum_{1 \leq |\beta| \leq k} r_H^{|\beta|} \sup_y \left| \partial^\beta g_{ij}(y) \right|
      +\sum_{|\beta|=k} r_H^{k+\alpha}\sup_{y \neq z} \frac{\left|\partial^\beta g_{ij}(y) - \partial^\beta g_{ij}(z)\right|}{d_g(y, z)^\alpha}\leq Q-1$
\end{enumerate}

The following theorem is taken from \cite{HH}:

\begin{theorem}[\cite{HH}]\label{thmHarmRadCtrl} Given $\alpha \in (0, 1)$, $Q > 1$ and $\delta > 0$. Let $(M, g)$ be a smooth complete (n+1)-manifold and let $\Omega$ be an open subset of $M$. Set $\Omega_\delta = \left\{ x \in M~\text{such that}~d_g(x, \Omega) < \delta\right\}$. Assume that there exists a constant $C > 0$ such that:
\begin{equation}
\label{eqHarmRicciControl}
\left| \ric (x) \right|_g \leq C \quad \text{for all~}x \in \Omega_\delta.
\end{equation}

Assume also that the injectivity radius is bounded from below on $\Omega_\delta$:

\begin{equation}
\label{eqHarmIngControl}
\exists~i > 0~\text{such that}~\mathrm{inj}_{(M, g)} (x) > i\quad \forall x \in \Omega_\delta.
\end{equation}

There exists a positive constant $r_0 = r_0\left(n, Q, \alpha, \delta, C\right)$ such that

\begin{equation}
\label{eqHarmRadControl}
r_H\left(Q, 1, \alpha\right)(x) \geq r_0\quad\forall x \in \Omega.
\end{equation}
\end{theorem}

As pointed out in \cite{BahuaudGicquaud}, $(M, g)$ satisfies the assumptions of the theorem. Indeed, it is obvious that there exists a constant $C > 0$ such that $\left| \ric (x) \right|_g \leq C$. So we need only to prove that the injectivity radius is bounded from below. Since $K$ is compact it is sufficient to prove it on $\Omega = M \setminus K$. Set $\Omega_\delta = \{s > - \delta\}$ where $s$ is to be understood as the signed distance to $Y = \partial K$. If $\delta$ is small enough there is a diffeomorphism $\Omega_{2\delta} \simeq (-2\delta, \infty) \times Y$ given by the normal exponential map, such that $\sec_g < 0$ on $\Omega_{2\delta}$ and the second fundamental form of the slices $Y_s$ is positive definite. The exponential map with base point in $\Omega_\delta$ has no critical point at radius smaller than $\delta$ because of the negative curvature assumption so the injectivity radius on $\Omega_\delta$ is bounded from below if there is no closed geodesic with arbitrary small length (see \cite[Lemma 4.8.1]{Jost}). Even more is true:

\begin{lemma}[\cite{BahuaudGicquaud}]\label{lmClosedGeod}
There is no closed geodesic lying entirely in $\Omega_\delta$.
\end{lemma}

\begin{proof}
Let $\gamma: \mathbb{S}^1 \to \Omega_\delta$ be such a geodesic parametrized with constant speed. The function $s$ has a non-negative Hessian on $\Omega_\delta$ because its Hessian is the second fundamental form $S$ of the slices $Y_s$ so the image of $\gamma$ must lie in a slice $Y_s$ because otherwise $s$ would reach a maximum on the image of $\gamma$. Now $\gamma$ satisfies the geodesic equation, $\nabla_{\dot{\gamma}} \dot{\gamma} = 0$, and in particular \[0=\left\langle N,\nabla_{\dot{\gamma}}\dot{\gamma}\right\rangle=-\sff\left(\dot{\gamma},\dot{\gamma}\right) \neq 0,\] where $N$ is the unit normal vector to $Y_s$. A contradiction.
\end{proof}

As a first consequence of this result we can prove the following analog of the Cheng-Yau maximum principle (see e.g. \cite{GrahamLee} or \cite{GicquaudSakovich} for other variants)

\begin{lemma}[Cheng-Yau maximum principle]\label{lmChengYau}
Let $(M, g)$ be an ALH manifold, $K \subset M$ an essential subset and $f \in C^2_{loc}(M)$ be a function bounded from above. There exists a sequence of points $p_i \in M$ such that:
\begin{itemize}
\item $\lim_{k \to \infty} f\left(p_k\right) = \sup_M f$
\item $\lim_{k \to \infty} \left|\nabla f\left(p_k\right)\right|_g = 0$
\item $\limsup_{k \to \infty} \Delta f \left(p_k\right) \leq 0$.
\end{itemize}
\end{lemma}

\begin{proof} 
Set $F = (\sup_M f) - f$ so $F \geq 0$. Without loss of generality, we can assume that $F > 0$ everywhere on $M$. Select a sequence of points $q_k \in M$ such that $\lim_k F(q_k) = 0$. Choose arbitrary values for $Q > 1$ and $\alpha \in (0; 1)$. From Theorem \ref{thmHarmRadCtrl} and Lemma \ref{lmClosedGeod}, there exist harmonic coordinates $(x^i_k)$ on each of the balls $B_r (q_k)$, where $r > 0$ is the the infimum of the $(Q, 1, \alpha)$-harmonic radii of points of $M$, in particular, in these coordinates $Q^{-1} \delta_{ij} \leq g_{ij} \leq Q \delta_{ij}$ and $\left|\partial_l g_{ij}\right| \leq Q$. We will also assume that $x_k^i(q_k) = 0$. We set $\phi_k = r^2 - \sum_i (x^i)^2$. It is easy to see that $|\nabla \phi_k|_g$ and $|\Delta_g \phi_k|$ are bounded on $B_r(q_k)$ by some constant $C$ which depends only on $n$ and $Q$. Choose $p_k$ such that \[\frac{F(p_k)}{\phi_k(p_k)} = \min_{B_r(x_k)} \frac{F}{\phi_k}.\] We prove now that the sequence of points $p_k$ satisfies the conclusions of the lemma. Indeed, since $\phi_k(p_k) \leq r^2 = \phi_k(q_k)$, one has \[F(p_k) = \phi_k(p_k) \frac{F(p_k)}{\phi_k(p_k)} \leq r^2 \frac{F(p_k)}{\phi_k(p_k)} \leq r^2 \frac{F(q_k)}{\phi_k(q_k)} = F(q_k),\] so $F(p_k) \to 0$ i.e. $f(p_k) \to \sup_M f$. Also at $p_k$, $\nabla \log \frac{F}{\phi_k} = 0$, so $\frac{\nabla F}{F} = \frac{\nabla \phi_k}{\phi_k}$. From this we get that $|\nabla f|_g (p_k) = |\nabla F|_g (p_k) = \frac{F}{\phi_k}(p_k) |\nabla \phi_k|_g \leq \frac{C}{r^2} F(q_k) \to 0$. This proves the second point. Furthermore $\Delta_g \log \frac{F}{\phi_k}(p_k) \geq 0$. From this we get:

\[\begin{aligned}
\Delta_g \log F & \geq \Delta_g \log \phi_k\\
\frac{\Delta_g F}{F} - \frac{|\nabla F|^2_g}{F^2} & \geq \frac{\Delta_g \phi_k}{\phi_k} - \frac{|\nabla \phi_k|^2_g}{\phi_k^2}\\
\frac{\Delta_g F}{F}  & \geq \frac{\Delta_g \phi_k}{\phi_k}\\
\Delta_g F & \geq \frac{F}{\phi_k}(p_k) \Delta_g \phi_k.
\end{aligned}\]

Arguing as for the gradient, we get that $\Delta_g f(p_k) = - \Delta_g F(p_k)$ is smaller than some quantity that tends to $0$ when $k$ tends to infinity. This proves the third point of the lemma.
\end{proof}

\subsubsection{End of the proof of Lemma \ref{lmConstructionT}} We first extend the function $e^s : M \setminus K \to \bR$ to a smooth function $t_0 : M \to \bR$ and write $t = t_0 + t_1$. The function $t_1$ has to satisfy
\begin{eqnarray*}
- \Delta t_1 + (n+1) t_1
  & = & \Delta t_0 - (n+1) t_0\\
  & = & (H-n) e^s \quad\text{outside $K$}\\
\end{eqnarray*}

From the estimate $H = n + O(e^{-a s})$ (Proposition \ref{propComparison}), we deduce that $f := (H - n) e^s = O(e^{(1-a) s})$. For any integer $i > 0$ set $K_i = \{ s < i\} \cup K$. Remark that $\left|d s\right| = 1$ so $\partial K_i$ is smooth. On each $K_i$, there exists a unique solution $t_1^i \in W^{2, p} (K_i)$ of the equation
\[\left\lbrace\begin{aligned}
- \Delta t^i_1 + (n+1) t^i_1 & = f \quad\text{on $K_i$}\\
t^i_1 & = 0 \quad\text{on $\partial K_i$}.
\end{aligned}\right.\]

We now proceed in two different manners according to the value of $a$:

\noindent $\bullet$ $0 < a < 1$: Set $\phi_+ = A + t_0^{(1-a)}$. If $A > 0$ is large enough, it can be easily checked that $\phi_+$ is a super-solution: \[-\Delta \phi_+ + (n+1) \phi_+ \geq \delta \phi_+\] for some constant $\delta > 0$. Computing as in \cite{GrahamLee}, we get

\begin{eqnarray*}
-\Delta t_1^i
  & = & -\Delta \left( \phi_+ \frac{t_1^i}{\phi_+}\right)\\
  & = & -\frac{t_1^i}{\phi_+} \Delta \phi_+ - 2 \left\langle \nabla \phi_+, \nabla \frac{t_1^i}{\phi_+}\right\rangle_g - \phi_+ \Delta \frac{t_1^i}{\phi_+},\\
f - (n+1) t_1^i
  & = & -\frac{t_1^i}{\phi_+} \Delta \phi_+ - 2 \left\langle \nabla \phi_+, \nabla \frac{t_1^i}{\phi_+}\right\rangle_g - \phi_+ \Delta \frac{t_1^i}{\phi_+},\\
\frac{f}{\phi_+} - (n+1) \frac{t_1^i}{\phi_+}
  & = & -\frac{t_1^i}{\phi_+} \frac{\Delta \phi_+}{\phi_+} - 2 \left\langle \frac{\nabla \phi_+}{\phi_+}, \nabla \frac{t_1^i}{\phi_+}\right\rangle_g - \Delta \frac{t_1^i}{\phi_+},\\
\frac{f}{\phi_+}
  & = & \frac{t_1^i}{\phi_+} \left(-\frac{\Delta \phi_+}{\phi_+} + n+1 \right) - 2 \left\langle \frac{\nabla \phi_+}{\phi_+}, \nabla \frac{t_1^i}{\phi_+}\right\rangle_g - \Delta \frac{t_1^i}{\phi_+}.
\end{eqnarray*}

By the maximum principle (Lemma \ref{lmChengYau}), we get that \[\sup_{K_i} \left|\frac{t_1^i}{\phi_+}\right| \leq \frac{1}{\delta} \sup_{K_i} \frac{f}{\phi_+} \leq C \left\| f\right\|_{L^\infty_{a-1}(M, \bR)} < \infty\] (remark that $\left|\frac{\nabla \phi_+}{\phi_+}\right|_g$ is bounded on $M$). A classical argument involving the compactness of the embeddings $W^{2, p}(K_i, \bR) \into C^0(K_i, \bR)$ for each $i$ together with elliptic regularity and a diagonal extraction process (see \cite{GrahamLee}) yield the existence of a solution $t_1 \in W^{2, p}_{loc}$ of our initial equation. From the fact that $\left| t_1 \right| \leq C \left\| f\right\|_{L^\infty_{a-1}(M, \bR)} \phi_+$ and standard elliptic regularity applied in harmonic charts, we get that $t_1 \in X^{2, p}_{a-1}(M, \bR)$.\\

\noindent $\bullet$ $1 \leq a < 2$: The classical maximum principle leads to the fact that all the functions $t_1^i$ are uniformly bounded so the limit function $t_1$ is bounded. Now remark that if $i > 0$ is large enough, then the function $\phi_+ = A e^{(1-a)s}$ is a super-solution on $M \setminus K_i$: \[-\Delta \phi_+ + (n+1) \phi_+ \geq \delta \phi_+\] for some constant $\delta > 0$. Select $A$ large enough so that $\delta \phi_+ \geq f$ on $M \setminus K_i$, $\phi_+ \geq |t_1| + 1$ on $\partial K_i$. Then Lemma \ref{lmChengYau} shows that $-\phi_+ \leq t_1 \leq \phi_+$. Hence, by elliptic regularity in harmonic charts (given by Theorem \ref{thmHarmRadCtrl}), we get that $t_1 \in X^{2, p}_{a-1}(M, \bR)$.\\

Uniqueness is an easy exercise in both cases. The Hessian of $t$ satisfies the following equation:

\[\nabla_i \hessdd{j}{k} t - \nabla_j \hessdd{i}{k} t = - \riemuddd{l}{k}{i}{j} \nabla_l t.\]

Let $T = \hess (t) - t g$, by assumption this tensor is the traceless part of the Hessian of $t$. A simple calculation leads to the following equality:

\[\nabla_i T_{jk} - \nabla_j T_{ik} = - \tuddd{E}{l}{k}{i}{j} \nabla_l t.\]

Let $x$ be an arbitrary point on $M$. Let $r = r_H(Q, 1, \alpha)$. Applying Proposition \ref{propEllipticRegCodazzi} in harmonic charts centered at $x$, we get that \[\left\| T \right\|_{W^{1, p}(B_{\frac{r}{2}} (x))} \leq C \left( \left\| \tuddd{E}{l}{k}{i}{j} \nabla_l t \right\|_{L^p(B_r(x))} + \left\| T \right\|_{L^p(B_r(x))}\right),\] where $C > 0$ is a constant which depends only on $Q$. From the first part of the proof, we know that $t - e^s \in X^{2, p}_{a-1}$ so 

\begin{eqnarray*}
\hess (t) - t g  &  =  & \hess (t - e^s) - (t - e^s)g + \hess(e^s) - e^s g\\
		 &  =  & \hess (t_1) - t_1 g + e^s \left(S + ds \otimes ds - g\right) \quad\text{on $M \setminus K$}\\
		 & \in & X^{0, p}_{a-1},\\
\end{eqnarray*}
since from Proposition \ref{propComparison}, $S - g_s = O(e^{-a s})$ where $g_s$ is the metric induced on $Y_s$. Similarly,
\[\begin{aligned}
\left\| \tuddd{E}{l}{k}{i}{j} \nabla_l t \right\|_{L^p(B_1(x))}
  & \leq \left(\sup_{{B_r}(x)} \left| \mathcal{E} \right|_g\right) \left\| \nabla t \right\|_{L^p(B_r(x))}\\
  & \leq C e^{-a s} e^s = C e^{(1-a)s}.
\end{aligned}\]
Hence, \[\left\| T \right\|_{W^{1, p}(B_{\frac{r}{2}} (x))} \leq C e^{(1-a)s},\] i.e., following Remark \ref{rkDefX}, $T \in X^{1, p}_{a-1}$.\\

\begin{rk}\label{rkDefX}
We can now replace $e^s$ in the definition of the $X$ spaces by $t$.
\end{rk}

\subsection{Two integration lemmas}\label{secWindow}

Estimates for the $L^p$-norm of tensors in balls of constant radius can be obtained by a integration argument. This is the content of the next lemma. Before stating it, we need to introduce some notation. Let $U$ be an open subset of $\bR^n$ and let $h_0$ be the Euclidean metric on $U$. Define on $N = (0; \infty) \times U$ the metric \[h = dw^2 + h_w,\] where $h_w = e^{2 w} h_0$. Let $T$ be a covariant tensor of rank $k$ defined on $N$. For any point $x = (w, x^1, \ldots, x^n)$, we introduce the tangential (pseudo-)norm of $T$ with respect to the metric $h$ as follows:
\[ \left| T\right|_{h,t}^2 = \sum_{\mu_1,\ldots, \mu_{2k} = 1}^k h_w^{\mu_1 \mu_{k+1}} \cdots h_w^{\mu_k \mu_{2k}} T_{\mu_1 \ldots \mu_k} T_{\mu_{k+1} \ldots \mu_{2k}}.\]
For any $p \in [1; \infty)$ and any open subset $\Omega \subset N$, we define the tangential $L^p$-norm by

\begin{equation}\label{eqLpTangentialNorm}
\left\| T\right\|_{L^p_t(\Omega, h)} = \left(\int_\Omega \left| T\right|^p_{h,t} d\mu_h\right)^{\frac{1}{p}}.
\end{equation}

Define, for any $w \in (0; \infty)$, $x \in U$ and $r > 0$, \[\Omega_{w, x}(r) = \{(w', y), |w'-w|<r \text{ and } |y - x| < r e^{-w}\}.\] We also denote \[B_{w, x}(r) = \{(z, y) \vert -r < z < r, |y - x| < r e^{-w}\}.\]

\begin{lemma}\label{lmMovingWindow1} Let $x \in U$ and $r > 0$ be such that $B_r(x) \subset U$ and $p \in [1; \infty)$. Let $T$ be a covariant tensor of rank $k$ on $N$ such that $T \in L^p_{loc}$, $\partial_0 T \in L^p_{loc}$. The following estimate holds for any $w_0, w_1 \in \bR$ such that $0 \leq w_0 \leq w_1$:
\begin{eqnarray*}
\left\|T\right\|_{L^p_t(\Omega_{w_1, x}(r), h)}
  & \leq & \int_{w_0}^{w_1} e^{\left(\frac{n}{p}-k\right)(w_1-w)} \left\|\partial_0 T\right\|_{L^p_t(\Omega_{w, x}(r), h)} dw\\
  &    & + e^{\left(\frac{n}{p}-k\right)(w_1-w_0)} \left\|T\right\|_{L^p_t(\Omega_{w_0, x}(r), h)}.
\end{eqnarray*}
\end{lemma}

\begin{proof}
From the fundamental formula of calculus, for all $y \in B_r(x)$ and $z \in (-r; r)$, we have

\[T(w_1 + z, y) = \int_{w_0}^{w_1} \partial_0 T (w + z, y) dw + T(w_0 + z, y).\]

Taking the norm and keeping track of the point under consideration, we get

\begin{eqnarray*}
\left|T(w_1 + z, y)\right|_{h_{w_1+z}, t}
  &   =  & \left| \int_{w_0}^{w_1} \partial_0 T (w + z, y) dw + T(w_0 + z, y)\right|_{h_{w_1+z}, t}\\
  & \leq & \int_{w_0}^{w_1} \left|\partial_0 T (w + z, y) \right|_{h_{w_1+z}, t} dw + \left| T(w_0 + z, y) \right|_{h_{w_1+z}, t}\\
  & \leq & \int_{w_0}^{w_1} e^{-k(w_1-w)} \left|\partial_0 T (w + z, y) \right|_{h_{w+z}, t} dw\\
  &    & + e^{-k(w_1-w_0)} \left| T(w_0 + z, y) \right|_{h_{w_0+z}, t}.
\end{eqnarray*}

By the triangle inequality for the norm \[\left\|f\right\|_{L^p(B_{w, x}(r), h)} = \left(\int_{B_{w,x}(r)} |f(z,y)|^p e^{n z} dz dy^n\right)^{\frac{1}{p}},\] we get

\begin{eqnarray*}
\lefteqn{\left\|T\right\|_{L^p_t(\Omega_{w_1, x}(r), h)}}\\
  &   =  & \left(\int_{B_{w_1, x}(r)} \left|T(w_1 + z, y)\right|^p_{h_{w_1+z}, t} e^{n(w_1+z)} dz dy^n\right)^{\frac{1}{p}}\\
  & \leq & \int_{w_0}^{w_1} e^{-k(w_1-w)} \left(\int_{B_{w_1, x}(r)} \left|\partial_0 T(w + z, y)\right|^p_{h_{w+z}, t} e^{n(w_1+z)} dz dy^n\right)^{\frac{1}{p}} dw\\
  & &    + e^{-k(w_1-w_0)} \left(\int_{B_{w_1, x}(r)} \left|T(w_0 + z, y)\right|^p_{h_{w_0+z}, t} e^{n(w_1+z)} dz dy^n\right)^{\frac{1}{p}}\\
  & \leq & \int_{w_0}^{w_1} e^{\left(\frac{n}{p}-k\right)(w_1-w)} \left(\int_{B_{w, x}(r)} \left|\partial_0 T(w + z, y)\right|^p_{h_{w+z}, t} e^{n(w+z)} dz dy^n\right)^{\frac{1}{p}} dw\\
  & &    + e^{\left(\frac{n}{p}-k\right)(w_1-w_0)} \left(\int_{B_{w_0, x}(r)} \left|T(w_0 + z, y)\right|^p_{h_{w_0+z}, t} e^{n(w_0+z)} dz dy^n\right)^{\frac{1}{p}}.
\end{eqnarray*}

Consequently,
\begin{eqnarray}
\label{eqMovingWindow1}
\left\|T\right\|_{L^p_t(\Omega_{w_1, x}(r), h)}
  & \leq & \int_{w_0}^{w_1} e^{\left(\frac{n}{p}-k\right)(w_1-w)} \left\|\partial_0 T\right\|_{L^p_t(w+B_{w_1, x}(r), h)} dw\\
  &		 & + e^{\left(\frac{n}{p}-k\right)(w_1-w_0)} \left\|T\right\|_{L^p_t(w_0+B_{w_1, x}(r), h)},\nonumber
\end{eqnarray}
where we used the notation $w + B_{w_1, x}(r) = \{(w', y) \vert w-r < w' < w+r, |y - x| < r e^{-w}\}$. Finally remark that for any $w \leq w_0$, $w + B_{w_1, x} \subset \Omega_{w, x}(1)$. This conclude the proof of the lemma.
\end{proof}

Intuitively, the exponential term appearing in Lemma \ref{lmMovingWindow1} should be $e^{-k w}$ instead of $e^{\left(\frac{n}{p}-k\right)w}$. This is due to the fact that translating roughly a covariant tensor of order k in the $w$-direction makes its tangential norm decrease by a factor $e^{-k \delta w}$ where $\delta w$ is the displacement. Remark however that this factor is the pth-root of ratio of the volumes of $w+B_{w_1, x}(r)$ and $w+B_{w, x}(r) = \Omega_{w, x}(r)$. We will show that this undesirable fact can be avoided in certain circumstances. See Section \ref{secEstimates}.\\

Estimates based on this lemma will be obtained by the Gronwall lemma which we state here for future references. The proof of this lemma is standard.

\begin{lemma}[A Gronwall lemma]\label{lmGronwall}
Let $a, b : [w_0; w_1] \to \bR$ be two continuous functions defined on the interval $[w_0; w_1]$, with $b \geq 0$. Let $f : [w_0; w_1] \to \bR$ be a continuous function such that \[\forall w \in [w_0; w_1], f(w) \leq a(w) + \int_{w_0}^w b(v) f(v) dv,\] then $f$ satisfies the following inequality: \[\forall w \in [w_0; w_1], f(w) \leq a(w) + \int_{w_0}^w a(v) b(v) e^{\int_s^w b(u) du} dv.\]
\end{lemma}

We also present a development of Lemma \ref{lmMovingWindow1} which will be useful in Subsection \ref{secW1pHessbar} where we will see that the naive estimate for the tangential derivatives of certain functions (the components of the Hessian of the coordinate functions) is weaker than the estimate for the derivative in the direction of $w$.

\begin{lemma}\label{lmMovingWindow2} Let $x \in U$ be such that $B_r(x) \subset U$ and $p \in (n+1; \infty)$. Let $F \in W^{1, p}_{loc}(N, \bR)$ be a function. There exist constants $c_1, c_2$ depending only on $p$ and $r$ such that for any $w_0, w_1 \in (0; \infty)$, $w_1 > w_0$, the following estimate holds:
\begin{equation}\label{eqMovingWindow2}
\begin{aligned}
\left\|F - \Ftil(w_1, x)\right\|_{L^p(\Omega_{w_1,x}(r), h)} \leq
  & c_1 e^{\frac{n}{p}(w_1 - w_0)} \left\|\partial_0 F \right\|_{L^p(\Omega_{w_0,x}(r), h)}\\
  & + c_2 e^{-\left(1-\frac{n}{p}\right)(w_1 - w_0)}\left\|d F \right\|_{L^p(\Omega_{w_0,x}(r), h)}\\
  & + 2 \int_{w_0}^{w_1} e^{\frac{n}{p}(w_1-w)} \left\|\partial_0 F\right\|_{L^p(\Omega_{w, x}(r), h)} dw,
\end{aligned}
\end{equation}
where \[\Ftil(w_1, x) = \frac{1}{\left|\Omega_{w_1, x}(r)\right|} \int_{z=-r}^r \int_{|y-x| < r e^{-w_1}} F(w_1+z, y) e^{n(w_1+z)} dy dz\] is the average of $F$ on $\Omega_{w_1, x}(r)$ with respect to the measure associated to $h$.
\end{lemma}

\begin{proof}
As in the proof of Lemma \ref{lmMovingWindow1}, we start from the fundamental formula of calculus: \[F(w_1+z, y) = F(w_0+z, y) + \int_{w_0}^{w_1} \partial_0 F(w+z, y) dw.\] Integrating with respect to $y$ and $z$, we get
\begin{eqnarray*}
\Ftil(w_1, x)
  & = & \frac{1}{\left|\Omega_{w_1, x}(r)\right|} \int_{(z', y') \in B_{w_1, x}(r)} F(w_0+z', y') e^{n(w_1+z')} dy' dz'\\
  & & + \frac{1}{\left|\Omega_{w_1, x}(r)\right|} \int_{w_0}^{w_1} \int_{(z', y') \in B_{w_1, x}(r)} \partial_0 F(w+z', y') e^{n(w_1+z')} dy' dz' dw.
\end{eqnarray*}
Hence,
\begin{multline}\label{eqProofMovingWindow2}
  F(w_1+z, y) - \Ftil(w_1, x) =\\
  F(w_0+z, y) - \frac{1}{\left|\Omega_{w_1, x}(r)\right|} \int_{(z', y') \in B_{w_1, x}(r)} F(w_0+z', y') e^{n(w_1+z')} dy' dz'\\
  + \int_{w_0}^{w_1} \left(\partial_0 F(w+z, y) - \frac{1}{\left|\Omega_{w_1, x}(r)\right|} \int_{(z', y') \in B_{w_1, x}(r)} \partial_0 F(w+z', y') e^{n(w_1+z')} dy' dz' \right) dw.
\end{multline}
We now estimate $\left\|F - \Ftil(w_1, x)\right\|_{L^p(\Omega_{w_1, x}(r), h)}$, the $L^p$-norm of this previous expression on $\Omega_{w_1, x}(r)$. We estimate first the second term of \eqref{eqProofMovingWindow2}. Remark that the term integrated with respect to $z'$ and $y'$ is the average of \mbox{$\partial_0 F(w+., .)$} over $B_{w_1, x}(r)$ and is independent of $z$ and $y$, so
\begin{eqnarray*}
\lefteqn{\left\|\frac{1}{\left|\Omega_{w_1, x}(r)\right|} \int_{(z', y') \in B_{w_1, x}(r)} \partial_0 F(w+z', y') e^{n(w_1+z')} dy' dz'\right\|_{L^p(\Omega_{w_1, x}(r), h)}}\\
 & =	& \frac{1}{\left|\Omega_{w_1, x}(r)\right|^{1-\frac{1}{p}}} \left|\int_{(z', y') \in B_{w_1, x}(r)} \partial_0 F(w+z', y') e^{n(w_1+z')} dy' dz'\right|\\
 & \leq & \left|\int_{(z', y') \in B_{w_1, x}(r)} \left(\partial_0 F(w+z', y')\right)^p e^{n(w_1+z')} dy' dz'\right|^{\frac{1}{p}}\\
 & \leq & e^{\frac{n}{p}(w_1-w)} \left\|\partial_0 F\right\|_{L^p(w + B_{w_1, x}(r), h)},
\end{eqnarray*}
where we used the H\"older inequality to pass from the second line to the third one. Hence,
\begin{multline*}
\lefteqn{\left\|\int_{w_0}^{w_1} \left(\partial_0 F(w+z, y) - \frac{1}{\left|\Omega_{w_1, x}(r)\right|} \int_{(z', y') \in B_{w_1, x}(r)}\!\!\!\!\!\!\!\!\!\!\!\!\!\!\!\!\!\!\!\!\!\!\!\! \partial_0 F(w+z', y') e^{n(w_1+z')} dy' dz' \right) dw\right\|_{L^p(\Omega_{w_1, x}, h)}}\\
 \leq 2 \int_{w_0}^{w_1} e^{\frac{n}{p}(w_1-w)} \left\|\partial_0 F\right\|_{L^p(w + B_{w_1, x}(r), h)} dw.
\end{multline*}
We now estimate the $L^p$-norm of the first term of Equation \eqref{eqProofMovingWindow2}. We remark that it can be rewritten as follows (see also \cite[Lemma 7.16]{GilbargTrudinger}):
\begin{eqnarray*}
\lefteqn{F(w_0+z, y) - \frac{1}{\left|\Omega_{w_1, x}(r)\right|} \int_{(z', y') \in B_{w_1, x}(r)} F(w_0+z', y') e^{n(w_1+z)} dy' dz'}\\
 & = & \frac{1}{\left|\Omega_{w_1, x}(r)\right|} \int_{(z', y') \in B_{w_1, x}(r)} \left(F(w_0+z, y) - F(w_0+z', y')\right) e^{n(w_1+z)} dy' dz'\\
 & = & \frac{1}{\left|\Omega_{w_1, x}(r)\right|} \int_0^1 \int_{(z', y') \in B_{w_1, x}(r)}\!\!\!\!\!\!\!\!\!\!\!\!\!\!\!\!\!\!\!\!\!\!\!\! \left((z-z')\partial_0 F + (y-y')^\mu \partial_\mu F\right)(w_0+z'+\lambda(z-z'), y' + \lambda(y-y'))\\
 & & \qquad e^{n(w_1+z)} dy' dz' d\lambda.
\end{eqnarray*}
Hence, using the convention that, unless explicitely written, $F$ and its derivatives have to be evaluated at the point $(w_0+z'+\lambda(z-z'), y' + \lambda(y-y'))$, we get
\begin{eqnarray*}
\lefteqn{\left(\int_{(z, y) \in B_{w_1, x}(r)} \left|F(w_0+z, y) - \frac{1}{\left|\Omega_{w_1, x}(r)\right|} \int_{(z', y') \in B_{w_1, x}(r)}\!\!\!\!\!\!\!\!\!\!\!\!\!\!\!\!\!\!\!\!\!\!\!\! F(w_0+z', y') e^{n(w_1+z)} dy' dz'\right|^p e^{n(w_1+z)} dz dy\right)^{\frac{1}{p}}}\\
 & \leq & \frac{1}{\left|\Omega_{w_1, x}(r)\right|} \left(\int_{(z, y), (z', y') \in B_{w_1, x}(r)} \left|\int_0^1 \left((z-z')\partial_0 F + (y-y')^\mu \partial_\mu F\right) d\lambda\right|^p\right.\\
 & & \qquad\left. e^{n(w_1+z')} dz' dy' e^{n(w_1+z)} dz dy\right)^{\frac{1}{p}}\\
 & \leq & \frac{1}{\left|\Omega_{w_1, x}(r)\right|} \left(\int_{(z, y), (z', y') \in B_{w_1, x}(r)} \left|\int_0^1 (z-z')\partial_0 F d\lambda\right|^p e^{n(2 w_1+ z' + z)} dz' dy' dz dy\right)^{\frac{1}{p}}\\
 & & + \frac{1}{\left|\Omega_{w_1, x}(r)\right|} \left(\int_{(z, y), (z', y') \in B_{w_1, x}(r)} \left|\int_0^1 (y-y')^\mu \partial_\mu F d\lambda\right|^p e^{n(2 w_1+ z' + z)} dz' dy' dz dy\right)^{\frac{1}{p}}\\
 & \leq & \frac{1}{\left|\Omega_{w_1, x}(r)\right|} \int_0^1\left(\int_{(z, y), (z', y') \in B_{w_1, x}(r)} \left|(z-z')\partial_0 F\right|^p e^{n(2 w_1+ z' + z)} dz' dy' dz dy\right)^{\frac{1}{p}} d\lambda\\
 & & + \frac{1}{\left|\Omega_{w_1, x}(r)\right|} \int_0^1\left(\int_{(z, y), (z', y') \in B_{w_1, x}(r)} \left|(y-y')^i \partial_i F\right|^p e^{n(2 w_1+ z' + z)} dz' dy' dz dy\right)^{\frac{1}{p}}d\lambda
\end{eqnarray*}

Remark that if $(z, y), (z', y') \in B_{w_1, x}(r)$, $|z-z'| < 2r$ and $|y-y'|_{h_0} < 2r e^{r-w_1}$. Thus, $|y-y'|_{h_{w_0}} < 2r e^r e^{w_0 - w_1}$. As a consequence,

\begin{eqnarray*}
 \lefteqn{\left(\int_{(z, y) \in B_{w_1, x}(r)} \left(F(w_0+z, y) - \frac{1}{\left|\Omega_{w_1, x}(r)\right|} \int_{(z', y') \in B_{w_1, x}(r)}\!\!\!\!\!\!\!\!\!\!\!\!\!\!\!\!\!\!\!\!\!\!\!\! F(w_0+z', y') e^{n(w_1+z')} dy' dz'\right)^p e^{n(w_1+z)} dz dy\right)^{\frac{1}{p}}}\\
 & \leq & \frac{2r}{\left|\Omega_{w_1, x}(r)\right|} \int_0^1\left(\int_{(z, y), (z', y') \in B_{w_1, x}(r)} \left|\partial_0 F(w_0+z'+\lambda(z-z'), y' + \lambda(y-y')) \right|^p\right.\\
 & & \left.\qquad e^{n(2w_1+z'+z)} dz' dy' dz dy\right)^{\frac{1}{p}} d\lambda\\
 & & + \frac{2r e^r e^{w_0 - w_1}}{\left|\Omega_{w_1, x}(r)\right|} \int_0^1\left(\int_{(z, y), (z', y') \in B_{w_1, x}(r)} \left|\left|dF\right|_h (w_0+z'+\lambda(z-z'), y' + \lambda(y-y'))\right|^p\right.\\
 & & \qquad \left.\phantom{\int} e^{n(2w_1+z'+z)} dz' dy' dz dy\right)^{\frac{1}{p}}d\lambda.
\end{eqnarray*}

In both integrals, we set $u_\lambda = z' + \lambda (z-z')$ and $v_\lambda = y + \lambda (y'-y)$. So,

\begin{eqnarray*}
\lefteqn{\left(\int_{(z, y) \in B_{w_1, x}(r)} \left(F(w_0+z, y) - \frac{1}{\left|\Omega_{w_1, x}(r)\right|} \int_{(z', y') \in B_{w_1, x}(r)}\!\!\!\!\!\!\!\!\!\!\!\!\!\!\!\!\!\!\!\!\!\!\!\! F(w_0+z', y') e^{n(w_1+z')} dy' dz'\right)^p e^{n(w_1+z)} dz dy\right)^{\frac{1}{p}}}\\
 & \leq & \frac{2r}{\left|\Omega_{w_1, x}(r)\right|} \int_0^1 \lambda^{-\frac{n+1}{p}}\left(\int_{(z, y), (z', y') \in B_{w_1, x}(r)}\!\!\!\!\!\!\!\!\!\!\!\!\!\!\!\!\!\!\!\!\!\!\!\! \left|\partial_0 F(w_0+ u_\lambda, y + v_\lambda) \right|^p e^{n(2w_1+z'+z)} du_\lambda dv_\lambda dz dy\right)^{\frac{1}{p}} d\lambda\\
 & & + \frac{2r e^r e^{w_0 - w_1}}{\left|\Omega_{w_1, x}(r)\right|} \int_0^1 \lambda^{-\frac{n+1}{p}} \left(\int_{(z, y), (z', y') \in B_{w_1, x}(r)}\!\!\!\!\!\!\!\!\!\!\!\!\!\!\!\!\!\!\!\!\!\!\!\!\!\!\!\!\!\! \left|\left|dF\right|_h (w_0+ u_\lambda, y + v_\lambda)\right|^p e^{n(2w_1+z'+z)} du_\lambda dv_\lambda dz dy\right)^{\frac{1}{p}}d\lambda\\
 & \leq & \frac{2r}{\left|\Omega_{w_1, x}(r)\right|} \left|\Omega_{w_1, x}(r)\right|^{\frac{1}{p}} \int_0^1 \lambda^{-\frac{n+1}{p}} d\lambda  \left(\int_{(z, y) \in B_{w_1, x}(r)}\!\!\!\!\!\!\!\!\!\!\!\!\!\!\!\!\!\!\!\!\!\!\!\! \left|\partial_0 F(w_0+ z, y) \right|^p e^{n(w_1+z)} dz dy\right)^{\frac{1}{p}}\\
 & & + \frac{2r e^r e^{w_0 - w_1}}{\left|\Omega_{w_1, x}(r)\right|} \left|\Omega_{w_1, x}(r)\right|^{\frac{1}{p}} \int_0^1 \lambda^{-\frac{n+1}{p}} d\lambda \left(\int_{(z, y) \in B_{w_1, x}(r)}\!\!\!\!\!\!\!\!\!\!\!\!\!\!\!\!\!\!\!\!\!\!\!\! \left|\left|dF\right|_h (w_0 + z, y)\right|^p  e^{n(w_1+z)} dz dy\right)^{\frac{1}{p}}\\
 & \leq & \frac{2r p}{p - n - 1} \left|\Omega_{w_1, x}(r)\right|^{\frac{1}{p}-1} e^{\frac{n}{p}(w_1 - w_0)} \left\|\partial_0 F \right\|_{L^p(w_0 + B_{w_1,x}(r), h)}\\
 & & + \frac{2r e^r p}{p - n - 1} e^{w_0 - w_1} \left|\Omega_{w_1, x}(r)\right|^{\frac{1}{p}-1} \left\|d F \right\|_{L^p(w_0 + B_{w_1,x}(r), h)}.
\end{eqnarray*}
Combining the two inequalities, we finally get:
\[\begin{aligned}
\left\|F - \Ftil(w_1, x)\right\|_{L^p(\Omega_{w_1,x}(r), h)}
  \leq & c_1 e^{\frac{n}{p}(w_1 - w_0)} \left\|\partial_0 F \right\|_{L^p(w_0 + B_{w_1,x}(r), h)}\\
       & + c_2 e^{-\left(1-\frac{n}{p}\right)(w_1 - w_0)}\left\|d F \right\|_{L^p(w_0 + B_{w_1,x}(r), h)}\\
       & + 2 \int_{w_0}^{w_1} e^{\frac{n}{p}(w_1-w)} \left\|\partial_0 F\right\|_{L^p(\Omega_{w, x}(r), h)} dw,
\end{aligned}\]
where \[ c_1 = \frac{2r p}{p - n - 1} \left|\Omega_{w_1, x}(r)\right|^{\frac{1}{p}-1}, c_2 = \frac{2r e^r p}{p - n - 1} \left|\Omega_{w_1, x}(r)\right|^{\frac{1}{p}-1}.\] (Remark that the volume of $\Omega_{w, x}(r)$ depends only on $r$).
\end{proof}

\subsection{A regularity result}\label{secBridge}
In this section, we show how the behavior of the derivatives of a function near an hypersurface influences its regularity. Similar results have been previously proven in \cite[Lemma 3.8]{BahuaudGicquaud}, see also \cite[Lemma 2.4]{HuQingShi} and \cite[Proposition 13.8.7]{TaylorPDE3}. Due to the fact that Schauder estimates are false for $C^2$ or $W^{2, \infty}$ regularity, we will not always get $L^\infty$-control of the derivatives of the metric. So we give the following improvement. Let $U \subset \bR^n$ be an open subset and $\epsilon > 0$. We denote $h$ the hyperbolic metric $h = \frac{1}{\rho^2} (d\rho^2 + \sum_i (dx^i)^2)$ on $\Omega = U \times (0; \epsilon)$ and remark that setting $w = - \log \rho$ this metric equals the metric defined in Section \ref{secWindow}. For any $r > 0$, we set
\begin{equation}\label{eqDefOmegaR}
  \Omega(r) = \cup_{(w, x) \in \Omega} \Omega_{w, x}(r). 
\end{equation}

\begin{lemma}\label{lmBridge} Suppose that for some $r > 0$ \[\begin{array}{rccc} F: & \Omega(r) & \to & \bR\\ & (x, \rho) & \mapsto & F(x, \rho)\end{array}\] is such that $dF \in X^{0, p}_a(\Omega(r), h)$, for some $p \in (n+1; \infty)$ and some $a \geq 0$ such that $a \leq 1 - \frac{n+1}{p}$. Then $F \in C^a(\overline{\Omega})$ where the closure is taken in $\bR^n \times \bR_+$.
\end{lemma}

\begin{proof} Our proof extends \cite[Lemma 3.8]{BahuaudGicquaud}. We introduce the coordinate $w = -\log \rho$. We first remark that it is sufficient to prove the following tangential H\"older regularity:

\begin{equation}
\label{eqTangentialReg}
\frac{|F(x_1, \rho) - F(x_2, \rho)|}{|x_1 - x_2|^a} \leq \tilde{C}\quad \forall x_1, x_2 \in U, \rho \in (0; \epsilon)
\end{equation}
for some constant $\tilde{C}$ independent of $x_1, x_2, \rho$. Indeed, if $\rho_1, \rho_2 \in (0; \epsilon)$ (without loss of generality, we will assume that $\rho_1 < \rho_2$) then,

\[\begin{aligned}
\frac{|F(x_1, \rho_1) - F(x_2, \rho_2)|}{\left(|x_1-x_2|^2 + |\rho_1 - \rho_2|^2\right)^{\frac{a}{2}}}
  & \leq \frac{|F(x_1, \rho_1) - F(x_2, \rho_1)|}{\left(|x_1-x_2|^2 + |\rho_1 - \rho_2|^2\right)^{\frac{a}{2}}}
  + \frac{|F(x_2, \rho_1) - F(x_2, \rho_2)|}{\left(|x_1-x_2|^2 + |\rho_1 - \rho_2|^2\right)^{\frac{a}{2}}}\\
  & \leq \frac{|F(x_1, \rho_1) - F(x_2, \rho_1)|}{\left|x_1-x_2\right|^a}
  + \frac{|F(x_2, \rho_1) - F(x_2, \rho_2)|}{\left|\rho_1 - \rho_2\right|^a}\\
  & \leq \tilde{C} + \frac{|F(x_2, \rho_1) - F(x_2, \rho_2)|}{\left|\rho_1-\rho_2\right|^a},
\end{aligned}\]
provided that Estimate \eqref{eqTangentialReg} holds. We estimate the second term as follows. For any $\rho_0 \in [\rho_1; \rho_2]$, we introduce the coordinates $y = \rho_0^{-1} (x - x_2)$. In the coordinate system $(w, y)$, the metric $h$ takes the following form \[h = dw^2 + e^{2(w-w_0)} \delta_{\mu\nu} dy^\mu dy^\nu,\] where $w_0 = -\log \rho_0$. In particular, restricting to the cylinder $\{(w, y) \in \bR \times \bR^n | w_0-r \leq w \leq w_0 + r, |y|^2 \leq r^2 e^{-2(w-w_0)}\}$ (which corresponds to $\Omega_{w_0, x_2}(r)$ in the coordinate system $(w, x)$), the metric is uniformly equivalent to the Euclidean metric $dw^2 + \delta_{\mu\nu} dy^\mu dy^\nu$. From the Morrey theorem \cite[Theorem 7.17]{GilbargTrudinger}, we get that for any $(x, \rho), (x', \rho') \in \Omega_{w_0, x_2}(r)$,

\begin{equation}\label{eqSobolevOmega}
\frac{|F(x, \rho) - F(x', \rho')|}{\left(|w-w'|^2 + e^{2 w_0} |x-x'|^2\right)^{\frac{a}{2}}} \leq C \left\|dF\right\|_{L^p(\Omega_{w_0, x_2}(r), h)} \leq C' e^{-a w_0}\left\|dF\right\|_{X^{0, p}_a},
\end{equation}
where $C$ is a constant independent of $w_0$, $x_2$ and $F$, and $w = -\log \rho$, $w' = -\log \rho'$. In particular, if $x = x' = x_2$, $w' = w_0$, we get that \[\frac{|F(x_2, e^{-w}) - F(x_2, e^{-w_0})|}{\left|w-w_0\right|^a} \leq C' e^{-a w_0}\left\|dF\right\|_{X^{0, p}_a}\] for any $w \in (w_0-1; w_0+1)$. Set $k = \left\lfloor w_1-w_2\right\rfloor+1$ and $\delta w = \frac{w_1-w_2}{k} < 1$ where we denote $w_i = -\log \rho_i$, $i=1,2$ and by $\lfloor x \rfloor$ the greatest integer smaller than or equal to $x$. We estimate

\[\begin{aligned}
|F(x_2, \rho_1) - F(x_2, \rho_2)|
  & \leq \sum_{l=0}^{k-1} \left|F(x_2, e^{-(w_2 + l \delta w)}) - F(x_2, e^{-(w_2 + (l+1) \delta w)})\right|\\
  & \leq C' \sum_{l=0}^{k-1} e^{-a (w_2 + l \delta w)} (\delta w)^a \left\|dF\right\|_{X^{0, p}_a}.
\end{aligned}\]

We now distinguish two cases. If $w_1 - w_2 \geq 1$, $\frac{1}{2} \leq \delta w \leq 1$ then

\[\begin{aligned}
|F(x_2, \rho_1) - F(x_2, \rho_2)|
  & \leq C' \left\|dF\right\|_{X^{0, p}_a} e^{-a w_2} \sum_{l=0}^{\infty} e^{-a \frac{l}{2}}\\
  & \leq C'' \left\|dF\right\|_{X^{0, p}_a} \rho_2^a\\
  & \leq C^{(3)} \left\|dF\right\|_{X^{0, p}_a} |\rho_2 - \rho_1|^a.
\end{aligned}\]

If $w_1 - w_2 < 1$ then $k = 1$ and

\[\begin{aligned}
|F(x_2, \rho_1) - F(x_2, \rho_2)|
  & \leq C' e^{-a w_2} (w_1-w_2)^a \left\|dF\right\|_{X^{0, p}_a}\\
  & \leq C' \rho_2^a (w_1-w_2)^a \left\|dF\right\|_{X^{0, p}_a}\\
  & \leq C' \left|\rho_1 - \rho_2\right|^a \left\|dF\right\|_{X^{0, p}_a}
\end{aligned}\]
where we used the mean value theorem for the function $\log$ in the interval $[\rho_1, \rho_2]$ to pass from the second line to the third one. This shows that \[\frac{|F(x_2, \rho_1) - F(x_2, \rho_2)|}{\left|\rho_1-\rho_2\right|^a}\] is bounded independently of $x_2, \rho_1, \rho_2$.\\

We now turn our attention to the proof of the tangential H\"older regularity: \[\frac{|F(x_1, \rho) - F(x_2, \rho)|}{|x_1 - x_2|^a} \leq \tilde{C}\quad \forall x_1, x_2 \in U, \rho \in (0; \epsilon).\] We distinguish two cases. First if $\rho \geq |x_1-x_2|$, from the estimate \eqref{eqSobolevOmega}, we get \[|F(x_1, \rho) - F(x_2, \rho')| \leq C' \left\|dF\right\|_{X^{0, p}_a} |x_1-x_2|^a.\] Otherwise if $\rho \leq |x_1-x_2|$, as the proof of \cite[Lemma 3.8]{BahuaudGicquaud} suggests, there is to lift this inequality up to some height $h \geq \rho$: \[|F(x_1, \rho) - F(x_2, \rho)| \leq |F(x_1, \rho) - F(x_1, h)| + |F(x_1, h) - F(x_2, h)| + |F(x_2, h) - F(x_2, \rho)|.\] The first and the last terms in this inequality have been already estimated: \[ |F(x_1, \rho) - F(x_1, h)|, |F(x_2, \rho) - F(x_2, h)| \leq C'' \left\|dF\right\|_{X^{0, p}_a} h^a.\] To estimate the second term, we use the estimate \eqref{eqSobolevOmega} and cut the segment between $x_1$ and $x_2$ in $k$ small pieces of length at most $\rho$, where $k = \left\lfloor \frac{|x_1-x_2|}{\rho} + 1\right\rfloor$. Hence, from the estimate \eqref{eqSobolevOmega} applied with $w_0 = -\log \rho$, we get: \[|F(x_1, h) - F(x_2, h)| \leq k C' h^a \left\|dF\right\|_{X^{0, p}_a} \leq C' h^{a-1} \left\|dF\right\|_{X^{0, p}_a} (|x_1 - x_2| + h).\] Thus, we get the following estimate for the tangential H\"older regularity: \[|F(x_1, \rho) - F(x_2, \rho)| \leq 2 C'' \left\|dF\right\|_{X^{0, p}_a} h^a + C' h^{a-1} \left\|dF\right\|_{X^{0, p}_a} (|x_1 - x_2| + h).\] Selecting $h = |x_1 - x_2|$, we get: \[|F(x_1, \rho) - F(x_2, \rho)| \leq 2 (C'' + C') \left\|dF\right\|_{X^{0, p}_a} |x_1 - x_2|^a.\] This ends the proof of the lemma.
\end{proof}

\section{The case $0 < a < 1$}\label{secAlt1}
% The case 0 < a < 1: Preliminary analysis

In what follows, for any $t_0 \in \bR$, we denote $\Sigma_{t_0} = t^{-1} (t_0)$ the level set of the function $t$. Since the function $t$ satisfies $\left|\nabla t - \nabla e^s\right|_g = O(e^{(1-a)s})$, we get that the function $t$ is proper and has no critical points outside some compact set $K' \subset\subset \left\{t < t_0 \right\}$ for some $t_0$ large enough. The gradient flow of $t$ defines a diffeomorphism $t^{-1} [t_0; \infty) \simeq \Sigma_{t_0} \times [t_0; \infty)$. We can define a new coordinate system in a neighborhood of infinity by selecting coordinate functions $x^\mu$ on some $U_0 \subset \Sigma_{t_0}$ and extending them radially by the gradient flow (i.e. solve the first order ODE $\langle dt, dx^\mu\rangle_g = 0$). We also denote $w = \log t$. In these modified Fermi coordinates, the metric takes the following form:

\begin{equation}\label{eqMetricZeroShift}
g = N^2 dw^2 + g_w,
\end{equation}

where $g_w = g_{\mu\nu} dx^\mu dx^\nu$ is the metric induced on the hypersurface $\Sigma_t$, where we set for any $\tau \geq t_0$, \[\Sigma_{\tau} = \{x \in M | t(x) = \tau \}.\] By analogy with the ADM formalism in general relativity, we will call the function $N$ the \emph{lapse} and remark that the shift vector is null in this context. Recall that we take the convention that Greek indices correspond to tangential coordinates $x^\mu$ while zero index corresponds to the $w$ coordinate. We set $\rho = t^{-1} = e^{-w}$, where $t$ is the function constructed in Lemma \ref{lmConstructionT}, and $\gbar = \rho^2 g$. From Equation \eqref{eqMetricZeroShift}, we get that \[\gbar = N^2 d\rho^2 + \rho^2 g_w.\] The aim of this section is to prove Theorem \ref{thmMain} in the case $0 < a < 1$.

\subsection{Zeroth and first order estimates for the metric $g$}\label{secZerothOrder} In this section, we show how the properties of the function $t$ can be translated into estimates for the metric $\gbar$. For future reference, we give the expression of the Christoffel symbols in modified Fermi coordinates:

\begin{lemma}[Christoffel symbols of $g$]\label{lmChristoffelZeroShift} In the coordinates $(w, x^\mu)$ the Christoffel symbols of the metric $g$ read:
\[\begin{aligned}
\Gamma^0_{00} 		& = \frac{\partial_0 N}{N}\\
\Gamma^0_{0\mu}		& = \frac{\partial_\mu N}{N}\\
\Gamma^\mu_{00}		& = g^{\mu\nu} \frac{\partial_\nu N}{N}\\
\Gamma^0_{\mu\nu}	& = - \frac{1}{2} N^{-2} \partial_0 g_{\mu\nu}\\
\Gamma^\mu_{0\nu}	& = \frac{1}{2} g^{\mu\sigma} \partial_0 g_{\sigma\nu}\\
\Gamma^\mu_{\nu\sigma}	& = \frac{1}{2} g^{\mu\alpha} \left(\partial_\nu g_{\alpha\sigma}+ \partial_\sigma g_{\nu\alpha} - \partial_\alpha g_{\nu\sigma} \right) \quad\text{(Christoffel symbols of $g_w$).}
\end{aligned}\]
\end{lemma}

\begin{proof}
Straightforward calculations.
\end{proof}

\begin{lemma}[$L^\infty$-estimate for $\gbar$]\label{lmLinftyEstimateZeroShift}~
\begin{itemize}
	\item For any $a \in (0; 2)$ the lapse function satisfies the following estimate: \[N = 1 + O(e^{-aw}),\]
	\item There exists a constant $C > 0$ such that for any $w \in [w_0; \infty)$, where $w_0 = \log t_0$, the metric $g_w$ satisfies \[ C^{-1} e^{2(w-w_0)} g_{w_0} \leq g_w \leq C e^{2(w-w_0)} g_{w_0},\] where we identify $t^{-1} (t_0; \infty) \subset M$ and $\Sigma_{t_0} \times [w_0; \infty)$.
\end{itemize}
As a consequence, there exists a constant $C' > 0$ such that the metric $\gbar$ satisfies the following estimate:
\[C'^{-1} \left( d\rho^2 + g_{w_0}\right) \leq \gbar \leq C' \left( d\rho^2 + g_{w_0}\right).\]
\end{lemma}

\begin{proof}
First, from Lemma \ref{lmConstructionT}, we have that \[\hess(t) = t g + O(e^{(1-a)w}).\] However, \[\hessdd{0}{0} t = t \left(1-\Gamma^0_{00}\right) = t \left(1-\frac{\partial_0 N}{N}\right).\] The estimate for $\hess(t)$ gives: \[\hessdd{0}{0} t = t g_{00} (1 + O(e^{-aw})) = t N^2 (1 + O(e^{-aw})).\] Hence, we have the following estimate: \[1 - \frac{\partial_0 N}{N} = N^2\left(1 + O(e^{-aw})\right).\] We rewrite it as follows: \[\partial_0\frac{e^{2(w-w_0)}}{N^2} = 2 e^{2(w-w_0)} + O(e^{(2-a)w}).\] This equation can be explicitly integrated and yields: \[\frac{e^{2(w-w_0)}}{N^2} - \frac{1}{N^2(w_0)} = e^{2(w-w_0)} - 1 + O(e^{(2-a)w}).\] This proves the estimate for $N$.\\

We now turn our attention to the estimate for the tangential part of metric: \[\frac{1}{2} N^{-2} \partial_0 g_{\mu\nu} = - \Gamma^0_{\mu\nu} = \frac{1}{t} \hessdd{\mu}{\nu} t.\] From the previous estimate, we have \[\frac{1}{2} \partial_0 g_w = g_w + O(e^{-aw}).\] We now argue as in \cite{BahuaudThesis}. Let $U$ be the tensor such that \[\frac{1}{2} \partial_0 g_{\mu\nu} = \left(\kronecker{\sigma}{\mu} + U^\sigma_{\phantom{\sigma}\mu}\right) g_{\sigma\nu}.\] $U$ satisfies $\left|U\right| \leq C e^{-aw}$. Let $\mu_+(w)$ denote the maximum eigenvalue of $g_w$ with respect to the metric $g_{w_0}$. The metric equation above and the estimate for the tensor $U$ imply that $\mu_+$ is Lipschitz continuous and that wherever it is differentiable \[ \diff{\mu_+}{w} \leq 2(1 + C e^{-a w}) \mu_+.\] Hence $\mu_+(w) \leq \mu_+(w_0) e^{2(w-w_0) - \frac{C}{a} (e^{-aw} - e^{-aw_0})} \leq C e^{2w}$. Similarly, if $\mu_-(w)$ denotes the minimum eigenvalue of $g_w$ with respect to the metric $g_{w_0}$, $\mu_-(w) \geq C^{-1} e^{2w}$. This proves that $C^{-1} e^{2w} g_{w_0} \leq g_w \leq C e^{2w} g_{w_0}$.
\end{proof}

As a corollary of the proof, we get the following:

\begin{lemma}[Estimate for the normal derivative of $\gbar$]\label{lmEstimateNormalDerivativeGbar1}
If $a \in(0; 2)$, the derivative of $\gbar$ with respect to $\rho$ satisfies the following estimate:
\[\partial_\rho \gbar_{\mu\nu} = O(\rho^{a-1})\]
\end{lemma}

\begin{proof}
Starting from the estimate \[\frac{1}{2} \partial_0 g = g + O(e^{-aw}),\] we get
\begin{eqnarray*}
\frac{1}{2} \partial_0 \gbar_{\mu\nu}
	& = & e^{-2w} \left(\frac{1}{2} \partial_0 g_{\mu\nu} - g_{\mu\nu}\right)\\
	& = & e^{-2 w} \left|\partial_\mu\right|_g \left|\partial_\nu\right|_g O(e^{-aw})\\
	& = & O(e^{-aw}).
\end{eqnarray*}
Hence,
\[\partial_\rho \gbar_{\mu\nu} = -\frac{1}{\rho} \partial_0 \gbar_{\mu\nu} = O(\rho^{a-1}).\]
\end{proof}

\begin{lemma}[Estimates for the derivatives of the lapse]\label{lmEstimatesLapse1}
If $a \in (0; 2)$, the following estimates hold:
\[
\left\lbrace
\begin{aligned}
\partial_\rho N &= O(\rho^{a-1})\\
\partial_\mu  N &= O(\rho^{a-1}).\\
\end{aligned}
\right.
\]
\end{lemma}

\begin{proof}
From Lemmas \ref{lmChristoffelZeroShift} and \ref{lmLinftyEstimateZeroShift}, we have $\Gamma^0_{00} = \frac{\partial_0 N}{N} = 1 - N^2 (1 + O(e^{-aw})) = O(e^{-aw})$. Hence, we obtain
\begin{eqnarray*}
\partial_\rho N
	& = & - \frac{1}{\rho} \partial_0 N\\
	& = & O(e^{(1-a) w})
\end{eqnarray*}
We now focus on the tangential derivatives of the lapse. From Lemma \ref{lmChristoffelZeroShift}, we have:
\begin{eqnarray*}
\frac{\partial_\mu N}{N}
	& = & \Gamma^0_{0\mu}\\
	& = & -\frac{1}{t} \hessdd{0}{\mu} t\\
	& = & -\frac{1}{t} \left(\hessdd{0}{\mu} t - t g_{0\mu}\right) \qquad\text{since $\partial_0$ and $\partial_\mu$ are orthogonal}\\
	& = & \frac{1}{t} \left|\partial_0\right|_g \left|\partial_\mu\right|_g O(t^{1-a})\\
	& = & O(t^{1-a}).
\end{eqnarray*}
So we immediately get that \[\partial_\mu N = O(\rho^{a-1}).\]
\end{proof}

Let $U_0' \subset \subset U_0$ be a non-empty open set of $\Sigma_{t_0}$. In what follows, we identify $U_0$ and $U'_0$ with their image in $\bR^n$ by the coordinates $x^\mu$. By rescaling the $x^\mu$ variables by some large constant, we can assume that, for any $x \in U'_0$, $B_1(x) \subset U_0$. We now define on $N = (w_0; \infty) \times U_0$ the metric $h = dw^2 + h_w$ where $h_w = e^{2w} h_0$ and $h_0 = \sum (dx^\mu)^2$ is the Euclidean metric on $U_0$. Select $p$ such that $a \leq 1 - \frac{n+1}{p}$.

\begin{lemma}[First order estimates for the tangential metric]\label{lmFirstOrderTangentialZeroShift} The tangential derivatives of the metric $\gbar$ satisfy the following estimate:
\[\left\|\partial \gbar\right\|_{L^p_t(\Omega_{w,x}(1), h)} = O(e^{-(2+a) w})\] uniformly in $x \in U'_0$, for $w \in (w_0+1; \infty)$.
\end{lemma}

\begin{proof}
The proof is based on Lemma \ref{lmMovingWindow1}. First remark that
\begin{eqnarray*}
\nabla_\alpha \hessdd{\beta}{\gamma} t
	& = & \partial_\alpha \left( \hessdd{\beta}{\gamma} t \right)
	- \Gamma^\sigma_{\alpha\beta} \hessdd{\sigma}{\gamma} t - \Gamma^\sigma_{\alpha\gamma} \hessdd{\beta}{\sigma} t
	- \Gamma^0_{\alpha\beta} \hessdd{0}{\gamma} t - \Gamma^0_{\alpha\gamma} \hessdd{\beta}{0} t,\\
\frac{1}{t} \nabla_\alpha \hessdd{\beta}{\gamma} t
	& = & - \partial_\alpha \left( \Gamma^0_{\beta\gamma}\right)
	+ \Gamma^\sigma_{\alpha\beta} \Gamma^0_{\sigma\gamma} + \Gamma^\sigma_{\alpha\gamma} \Gamma^0_{\beta\sigma}
	+ \Gamma^0_{\alpha\beta} \Gamma^0_{0\gamma} + \Gamma^0_{\alpha\gamma} \Gamma^0_{\beta 0}
\end{eqnarray*}

We now estimate all the terms that appear in this equation. We recall first that \[\frac{1}{2} \partial_0 g_{\mu\nu} = g_{\mu\nu} + O(e^{(2-a)w}).\] Hence,

\begin{eqnarray*}
\Gamma^\sigma_{\alpha\beta} \Gamma^0_{\sigma\gamma} + \Gamma^\sigma_{\alpha\gamma} \Gamma^0_{\beta\sigma}
	& = & - \frac{1}{2} N^{-2} \left(\Gamma^\sigma_{\alpha\beta} \partial_0 g_{\sigma\gamma} + \Gamma^\sigma_{\alpha\gamma} \partial_0 g_{\beta\sigma}\right)\\
	& = & - N^{-2} \left[\Gamma^\sigma_{\alpha\beta} g_{\sigma\gamma} + \Gamma^\sigma_{\alpha\gamma} g_{\beta\sigma} + \left(\Gamma * O(e^{(2-a)w})\right)_{\alpha\beta\gamma}\right]\\
	& = & - N^{-2} \left[ \partial_\alpha g_{\beta\gamma} + \left(\partial g * O(e^{-aw})\right)_{\alpha\beta\gamma}\right]\\
\end{eqnarray*}
where to pass from the second line to the third we used the explicit expression of the Christoffel symbols (see Lemma \ref{lmChristoffelZeroShift}) and the fact that $g^{\mu\nu} = O(e^{-2w})$. Similarly,
\begin{eqnarray*}
\Gamma^0_{\alpha\beta} \Gamma^0_{0\gamma} + \Gamma^0_{\alpha\gamma} \Gamma^0_{\beta0}
	& = & -\frac{1}{2} N^{-2} \left(\partial_0 g_{\alpha\beta} \frac{\partial_\gamma N}{N} + \partial_0 g_{\alpha\gamma} \frac{\partial_\beta N}{N}\right)\\
	& = & O(e^{(3-a)w}),\\
\partial_\alpha \left( \Gamma^0_{\beta\gamma}\right)
	& = & - \frac{1}{2} \partial_\alpha \left( N^{-2} \partial_0 g_{\beta\gamma} \right)\\
	& = & -\frac{1}{2} N^{-2} \partial_0 \partial_\alpha g_{\beta\gamma} + N^{-3} \partial_0 g_{\beta\gamma} \partial_\alpha N\\
	& = & -\frac{1}{2} N^{-2} \partial_0 \partial_\alpha g_{\beta\gamma} + O(e^{(3-a)w}).
\end{eqnarray*}

Combining all these estimates, we get
\begin{multline*}
\frac{1}{t} \nabla_\alpha \left(\hessdd{\beta}{\gamma} t - t g_{\beta\gamma}\right) =\\ -\frac{1}{2} N^{-2} \left[ \partial_0 \partial_\alpha g_{\beta\gamma} + 2 \partial_\alpha g_{\beta\gamma} + \left(\partial g * O(e^{-aw})\right)_{\alpha\beta\gamma}\right] + O(e^{(3-a)w}),
\end{multline*}
where we used the fact that $\nabla_\alpha t = 0$. Multiplying this inequality by $N^2$, we get \[ \partial_0 \partial_\alpha \gbar_{\beta\gamma} + \left(\partial \gbar * O(e^{-aw})\right)_{\alpha\beta\gamma} = -\frac{2 N^2}{t} \nabla_\alpha \left(\hessdd{\beta}{\gamma} t - t g_{\beta\gamma}\right) + O(e^{(1-a)w}).\] 

We now take the tangential norm with respect to the metric $h_w$. Since the tensor $\partial_\alpha \gbar_{\beta\gamma}$ is covariant of rank 3, we get
\[ \left|\partial_0 \partial \gbar\right|_{h, t} \leq C e^{-aw} \left| \partial \gbar\right|_{h, t} + \left|\frac{2 N^2}{t} \nabla\left(\hess(t) - tg\right)\right|_{h, t} + O(e^{-(2+a)w}),\] for some constant $C > 0$. We now use this formula to estimate the $L^p_t$-norm of $\partial_0 \partial \gbar$ on $\Omega_{w, x}(1)$ and obtain
\begin{multline*}
 \left\|\partial_0 \partial \gbar\right\|_{L^p_t(\Omega_{w, x}(1), h)} \leq \\ C' e^{-aw} \left\|\partial_\alpha \gbar_{\beta\gamma}\right\|_{L^p_t(\Omega_{w, x}(1), h)} + \left\|\frac{2 N^2}{t} \left(\hess(t) - tg\right)\right\|_{L^p_t(\Omega_{w, x}(1), h)} + O(e^{-(a+2)w}) 
\end{multline*}
 for some other constant $C' > 0$. From Lemmas \ref{lmLinftyEstimateZeroShift} and \ref{lmConstructionT}, we get that \[\left\|\frac{2 N^2}{t} \nabla\left(\hess(t) - tg\right)\right\|_{L^p_t(\Omega_{w, x}(1), h)} = O(e^{-(2+a)w}).\] As a conclusion, \[ \left\|\partial_0 \partial \gbar\right\|_{L^p_t(\Omega_{w, x}(1), h)} \leq C' e^{-aw} \left\|\partial \gbar\right\|_{L^p_t(\Omega_{w, x}(1), h)} + O(e^{-(a+2)w}).\]

We are now in a position to apply Lemma \ref{lmMovingWindow1}. For any $w_1 > w_0$, we have:
\begin{eqnarray*}
\left\|\partial \gbar\right\|_{L^p_t(\Omega_{w_1, x}(1), h)}
	& \leq & \int_{w_0}^{w_1} e^{\left(\frac{n}{p} - 3\right) (w_1 - w)} \left\|\partial_0 \partial \gbar\right\|_{L^p_t(\Omega_{w, x}(1), h)} dw\\
	&    & + e^{\left(\frac{n}{p} - 3\right) (w_1 - w_0)} \left\|\partial \gbar\right\|_{L^p_t(\Omega_{w_0, x}(1), h)}\\
	& \leq & C'' \int_{w_0}^{w_1} e^{\left(\frac{n}{p} - 3\right) (w_1 - w)} \left( e^{-aw} \left\|\partial \gbar\right\|_{L^p_t(\Omega_{w, x}(1), h)} + e^{-(a+2)w}\right) dw \\
	&    & + e^{\left(\frac{n}{p} - 3\right) (w_1 - w_0)} \left\|\partial \gbar\right\|_{L^p_t(\Omega_{w_0, x}(1), h)}.
\end{eqnarray*}

Set $f(w) = e^{- \left(\frac{n}{p} - 3\right) w} \left\|\partial \gbar\right\|_{L^p_t(\Omega_{w, x}(1), h)}$, the previous inequality becomes \[f(w_1) \leq C'' \int_{w_0}^{w_1} \left[e^{-a w} f(w) + e^{-\left(\frac{n}{p} - a + 1\right)w}\right] dw + f(w_0).\] Lemma \ref{lmGronwall} then implies that $f(w) = O(e^{-\left(\frac{n}{p} - a + 1\right)w})$. This ends the proof of the lemma.
\end{proof}

\subsection{End of the proof of Theorem \ref{thmMain} when $0 < a < 1$}

Set $L = t^{-1} (0; t_0]$. The gradient flow of $t$ defines a diffeomorphism from $M \setminus L$ to $\Sigma_{t_0} \times (t_0; \infty)$. Changing the second component of this product by its inverse ($t \mapsto \frac{1}{t}$), we get a diffeomorphism $\phi$ from $M \setminus \mathring{L}$ to $\Sigma_{t_0} \times \left(0; \frac{1}{t_0}\right)$. Therefore it makes sense to add a boundary to $M$ and define $\Mtil = M \cup \Sigma_\infty$ by prolonging the diffeomorphism $\phi$ to $\phitil : \Mtil \setminus L \to \Sigma_{t_0} \times \left[0; \frac{1}{t_0}\right)$. Recall from \cite{BahuaudMarsh} that the exponential map $\exp$ from the outward normal bundle $N_+ Y$ of $Y = \partial K$ to $M \setminus \mathring{K}$ also yields a diffeomorphism $\psi : M \setminus \mathring{K} \to Y \times (0; \infty)$ and can be extended by a similar procedure to a map $\Mbar \setminus K \to Y \times [0; 1)$ where the second coordinate is $\rho'(x) = e^{-s(x)}$ for any $x \in M \setminus K$. Remark that $Y \times \{0\}$ can be identified with $M(\infty)$, see \cite{BahuaudMarsh, Bahuaud, BahuaudThesis}. We prove the following proposition:

\begin{prop}
The manifolds $\Mbar$ and $\Mtil$ are bi-Lipschitz equivalent.
\end{prop}

\begin{proof}
Let $S$ be a compact subset containing both $K$ and $L$. We denote $\Omega_1 = \phi(M \setminus S) \subset \Sigma_{t_0} \times \left(0; \frac{1}{t_0}\right)$ and $\Omega_2 = \psi(M \setminus S) \subset Y \times (0; 1)$. Then the map $\psi \circ \phi^{-1}$ is a diffeomorphism from $\Omega_1$ to $\Omega_2$. We set $\Omegatil_1 = \Omega_1 \cup \Sigma_{t_0} \times \{0\}$ and $\Omegatil_2 = \Omega_2 \cup Y \times \{0\}$. Recall that the metric $\gbar = \rho^2 g$ extends to an $L^\infty$-metric on $\Mtil$ from Lemma \ref{lmLinftyEstimateZeroShift} and that similarly the metric $\ghat = (\rho')^2 g$ extends to an $L^\infty$-metric on $\Mbar$. The metric $\gbar$ (or more precisely $(\phi^{-1})^*\gbar$ on $\Omega_1$ satisfies \[ \gbar = \left(\frac{\rho'}{\rho}\right)^2 \left(\psi \circ \phi^{-1}\right)^* \ghat.\] From the facts that $\ghat$ and $\gbar$ are $L^\infty$-metrics and that $\frac{\rho'}{\rho}$ is uniformly bounded from Lemma \ref{lmConstructionT}, the norm of the differential of $\psi \circ \phi^{-1}$ is uniformly bounded on $\Omega_1$, hence extends uniquely to a Lipschitz map from $\Omegatil_1$ to $\Omegatil_2$. A similar proof shows that $\phi \circ \psi^{-1}$ is also a Lipschitz map from $\Omegatil_2$ to $\Omegatil_1$. By continuity, these two maps are inverse each of the other. This proves that $\Mbar$ and $\Mtil$ are bi-Lipschitz equivalent in a neighborhood of infinity. They are also clearly $C^\infty$-equivalent in the interior.
\end{proof}

In view of Lemmas \ref{lmLinftyEstimateZeroShift} and \ref{lmEstimateNormalDerivativeGbar1}, Lemma \ref{lmBridge} immediately gives that the function $N \in C^{0, a}(\Omega_0 \times [0;e^{-w_0}))$. We shall now turn our attention to the first order estimates for the tangential metric in the coordinate system we constructed. Combining Lemma \ref{lmEstimatesLapse1} and the fact that $\left|\partial_0 \gbar_w\right|_{g} = O(e^{-(2+a)w})$ (Lemma \ref{lmFirstOrderTangentialZeroShift}), we get that \[\left\|\partial \gbar_w\right\|_{L^p(\Omega_{w, x}, h)} = O(e^{-(a+2)w}).\] This means that for each component $\gbar_{\mu\nu}$ of $\gbar_w$,

\begin{eqnarray*}
\left\|\partial \gbar_{\mu\nu}\right\|_{L^p(\Omega_{w, x}(1), h)}
	&  =   & \left\|\partial \gbar_w(\partial_\mu, \partial_\nu)\right\|_{L^p(\Omega_{w, x}(1), h)}\\
	& \leq & \left\|\partial \gbar_w\right\|_{L^p(\Omega_{w, x}, h)} \left\|\partial_\mu\right\|_{L^\infty(\Omega_{w, x}(1), h)} \left\|\partial_\nu\right\|_{L^\infty(\Omega_{w, x}(1), h)}\\
	& \leq & C e^{2w} \left\|\partial \gbar_w\right\|_{L^p(\Omega_{w, x}, h)}\\
	&  =   & O(e^{-aw}).
\end{eqnarray*}

By Lemma \ref{lmBridge}, we get that each component $\gbar_{\mu\nu}$ belongs to $C^{0, a}(\Omega_0 \times [0;e^{-w_0}))$. This ends the proof of Theorem \ref{thmMain} in the case $0 < a < 1$.

\section{Construction of harmonic charts}\label{secHarm}
% Construction of harmonic charts

When constructing harmonic coordinates that complement the function $\rho$ which we defined previously, we are faced with the problem of choosing their value on $M(\infty)$. We shall first construct harmonic charts on $M(\infty)$ in Subsection \ref{secHarmonicInfinity} and extend them to harmonic charts on the inside of $M$ in Subsection \ref{secHarmonicInterior}. These coordinates enjoy a nice property we shall exploit in Section \ref{secEstimates}: we remark that they satisfy a certain Neumann condition at infinity. A similar idea was already present in an implicit form in \cite[Section 3]{HuQingShi}, where they construct even harmonic charts (with respect to the metric $\gbar$) on the double of $\Mbar$. In all that follows we select a point $\phat_0 \in M(\infty)$ and modified Fermi coordinate charts $x^\mu$ in a neighborhood $U'$ of $\phat_0$ as we constructed in the previous section such that $x^\mu(\phat_0) = 0$ so that $(\rho = t^{-1}, x^\mu)$ form a coordinate system in a neighborhood of $\phat_0$ in $\Mbar$. Up to a linear redefinition of the coordinates $x^\mu$, we can assume that $\gbar_{\mu\nu}(\phat_0) = \delta_{\mu\nu}$.\\

For any $\lambda > 0$, we set \[B_{\frac{1}{\lambda}} = \left\{\qhat \in M(\infty) | \sum_\mu (x^\mu(\qhat))^2 < \frac{1}{\lambda^2}\right\} \subset M(\infty)\] and \[D_{\frac{1}{\lambda}} = \left\{ p \in \Omega' | \rho^2(p) + \sum_\mu (x^\mu(p))^2 < \frac{1}{\lambda^2}\right\}  \subset M.\] If $0 < a < 1$, set $\alpha = a$ otherwise, if $1 < a < 2$, choose $\alpha \in (0; 1)$ arbitrarly. We prove the following proposition:

\begin{prop}\label{propCoordsHarm}
If $\lambda > 0$ is large enough, there exist harmonic functions $y^1, \ldots, y^n$ with respect to the metric $g$ on $D_{\frac{1}{\lambda}}$ such that \[\langle dw, dy^\mu\rangle = O(e^{-(1+a)w}),\] $y^\mu \in C^{1, \alpha}(D_{\frac{1}{\lambda}}, \gbar)$ and such that $(\rho, y^1, \ldots, y^n)$ form a coordinate system in a neighborhood of $\phat_0$.
\end{prop}

The proof is carried out in the next two subsections.

\subsection{Harmonic charts on $M(\infty)$}\label{secHarmonicInfinity}
We first replace the coordinate functions $x^\mu$ on $M(\infty)$ by new coordinates $y^\mu_\infty$ that are harmonic in a small ball around $\phat_0$. Recall that the Laplace operator for the metric $\gbar$ on $M(\infty)$ can be written \[\Deltabar \phi = \frac{1}{\sqrt{\det \gbar}} \partial_\mu \left(\gbar^{\mu\nu} \sqrt{\det \gbar}~ \partial_\nu \phi\right),\] see e.g. \cite[Exercise 1.35]{HamiltonRicci}. Since the metric $\gbar$ is only H\"older continuous, the equation $\Deltabar y^\mu_\infty = 0$ has to be interpreted in the weak sense: \[\forall \xi \in C^1_c\left(B_{\frac{1}{\lambda}}\right)~\int_{B_{\frac{1}{\lambda}}} \left\langle d \xi, d y^\mu_\infty \right\rangle_{\gbar} d\mu_{\gbar} = 0.\] To construct harmonic charts, we use low regularity results \cite[Theorem 8.33 and 8.34]{GilbargTrudinger} and a zoom process (see also \cite{HuQingShi}). For any (large) constant $\lambda$ we set $z^\mu = \lambda x^\mu$ and $\gbar' = \lambda^2 \gbar$. We shall also denote by $\mu', \nu',\ldots$ the components corresponding to the coordinate vectors $\pdiff{}{z^\mu}, \pdiff{}{z^\nu}, \ldots$: \[\gbar'_{\mu'\nu'} = \lambda^2 \gbar\left(\pdiff{}{z^\mu}, \pdiff{}{z^\nu}\right) = \gbar\left(\pdiff{}{x^\mu}, \pdiff{}{x^\nu}\right) = \gbar_{\mu\nu}.\] We consider the following equations for the functions $y^\mu_\infty$:

\[\left\lbrace
\begin{aligned}
\Deltabar y^\mu_\infty & = 0 \quad\text{on $B_{\frac{1}{\lambda}}$},\\
y^\mu_\infty & = z^\mu \quad\text{on $\partial B_{\frac{1}{\lambda}}$},
\end{aligned}
\right.\]
where $B_{\frac{1}{\lambda}}$ denotes the ball \[B_{\frac{1}{\lambda}} = \left\{\phat \in M(\infty) \vert \sum (x^\mu(\phat))^2 < \lambda^{-2}\right\} = \left\{\phat \in M(\infty) \vert \sum (z^\mu(\phat))^2 < 1\right\}.\] Remark that since $\gbar$ and $\gbar'$ are equal up to multiplication by a constant $\lambda^2$, solving $\Deltabar y^\mu_\infty = 0$ is equivalent to solving $\Delta_{\gbar'} y^\mu_\infty = 0$. Next, from the fact that the metric $\gbar$ is $\alpha$-H\"older continuous, for any two points $z_1 = \lambda x_1, z_2 = \lambda x_2 \in B_{\frac{1}{\lambda}}$ we have
\[\begin{aligned}
\left|\gbar'_{\mu'\nu'}(z_1) - \gbar'_{\mu'\nu'}(z_2)\right|
  & = \left|\gbar_{\mu\nu}(x_1) - \gbar_{\mu\nu}(x_2)\right|\\
  & \leq C \left|x_1 - x_2\right|^\alpha\\
  & \leq C \lambda^{-\alpha} \left|z_1 - z_2\right|^\alpha.
\end{aligned}\]
Hence the metric $\gbar'$ is uniformly controlled in $C^{1, \alpha}$-norm on $B_{\frac{1}{\lambda}}$. We write $y^\mu_\infty = z^\mu + u^\mu$. The functions $u^\mu$ have to satisfy

\[\left\lbrace
\begin{aligned}
\Delta_{\gbar'} u^\mu & = - \Delta_{\gbar'} z^\mu \quad\text{on $B_{\frac{1}{\lambda}}$},\\
u^\mu & = 0 \quad\text{on $\partial B_{\frac{1}{\lambda}}$}.
\end{aligned}
\right.\]

The first equation can be rewritten in the following simple form:
\[\begin{aligned}
\partial_{\mu'} \left(\gbar'^{\mu'\nu'} \sqrt{\det \gbar'}~ \partial_{\nu'} u^\mu\right)
  & = - \partial_{\mu'} \left(\gbar'^{\mu'\nu'} \sqrt{\det \gbar'}\right)\\
  & = - \partial_{\mu'} \left(\gbar'^{\mu'\nu'} \sqrt{\det \gbar'} - \gbar'^{\mu'\nu'}(0) \sqrt{\det \gbar'}(0)\right).   
\end{aligned}\]

From \cite[Theorem 8.34]{GilbargTrudinger}, the functions $u^\mu$ always exits and belong to $C^{1, \alpha}(B_{\frac{1}{\lambda}})$. An easy exercise from \cite[Theorem 8.33]{GilbargTrudinger} shows that there exist constants $C$ and $C'$ independent of $\lambda$ such that \[\left\|u^\mu\right\|_{C^{1, \alpha}(B_{\frac{1}{\lambda}})} \leq C \left\|\gbar'^{\mu'\nu'} \sqrt{\det \gbar'} - \gbar'^{\mu'\nu'}(0) \sqrt{\det \gbar'}(0) \right\|_{C^{0, \alpha}(B_{\frac{1}{\lambda}})} \leq C' \lambda^{-\alpha}.\] So if $\lambda$ is large enough, the functions $y^\mu_\infty = z^\mu + u^\mu$ form a coordinate system on $B_{\frac{1}{\lambda}}$.\\

Summarizing what we found so far, we get the following lemma:

\begin{lemma}\label{lmCoordHarmInfty}
For any $\phat_0 \in M(\infty)$, there exist coordinate charts $y^\mu_\infty$ centered at $\phat_0$ such that $y^\mu_\infty \in C^{1, \alpha}(B_{\frac{1}{\lambda}}(\phat_0))$ for some $\lambda > 0$ (i.e. the functions $y^\mu_\infty$ are $C^{1, \alpha}$ functions for the structure induced by the coordinates $x^\mu$), such that $y^\mu_\infty$ are harmonic for the metric $\gbar_\infty$ and such that $y^\mu_\infty = x^\mu$ on $\partial B_{\frac{1}{\lambda}}$.
\end{lemma}

\subsection{Harmonic charts on $M$}\label{secHarmonicInterior}
In what follows, we look for a function $y^\mu$ defined in a neighborhood of $\phat_0$ of the form $D_{\frac{1}{\lambda}}$ such that $y^\mu$ corresponds to the previously defined function $y^\mu_\infty$ on $M(\infty)$ and $y^\mu$ is harmonic with respect to the metric $g$. To simplify the notations, we choose once and for all an index $\mu \in \{1, \ldots, n\}$ and denote $\phi = y^\mu$, $\phi_\infty = y^\mu_\infty$, etc.\\

We define the warped product metric $\gcheck = dw^2 + e^{2 w} \gbar_\infty$. We first define $\phi_0: D_{\frac{1}{\lambda}} \to \bR$ to be constant along the $w$ coordinate and such that $\phi_0 = \phi_\infty$ on $M(\infty)$. Remark that $\phi_0$ is harmonic with respect to the metric $\gcheck$. With this remark at hand we can estimate $\Delta \phi_0$. Let $\xi \in C^1_c\left(D_{\frac{1}{\lambda}}\right)$ be an arbitrary test function, then

\begin{eqnarray*}
\int_{D_{\frac{1}{\lambda}}} \left\langle d \xi, d \phi_0\right\rangle_g d\mu_g
  & = & \int_{D_{\frac{1}{\lambda}}}\left\langle d \xi, d \phi_0\right\rangle_g d\mu_g - \int_{D_{\frac{1}{\lambda}}}\left\langle d \xi, d \phi_0\right\rangle_{\gcheck} d\mu_{\gcheck}\\
  & = & \int_{D_{\frac{1}{\lambda}}} \left\langle d \xi, d \phi_0\right\rangle_g \left(1 - \sqrt{\frac{\det \gcheck}{\det g}}\right) d\mu_g\\
  & & + \int_{D_{\frac{1}{\lambda}}} \left(\left\langle d \xi, d \phi_0\right\rangle_g - \left\langle d \xi, d \phi_0\right\rangle_{\gcheck}\right) \sqrt{\frac{\det \gcheck}{\det g}} d\mu_g.\\
\end{eqnarray*}
Hence, in the weak sense, $\Delta \phi_0 = \nabla^i \psi_i$ where \[\psi = -\left(\sqrt{\frac{\det \gcheck}{\det g}} - 1\right) d \phi_0 + \sqrt{\frac{\det \gcheck}{\det g}} \left[\left(g^{-1} - \gcheck^{-1}\right) (d \phi_0)\right]^{\flat}.\] From the fact that $\phi_0 \in C^{1, \alpha}(\Mbar, \gbar)$, it follows that $d\phi_0 \in C^{0, \alpha}(\Mbar, \gbar) \subset C^{0, \alpha}_1(M, g)$ (see \cite[Lemma 3.7]{LeeFredholm}). From Lemmas \ref{lmLinftyEstimateZeroShift}, \ref{lmEstimateNormalDerivativeGbar1} and \ref{lmEstimatesLapse1}, we also have that \[\sqrt{\frac{\det \gcheck}{\det g}} = \frac{1}{N} \sqrt{\frac{\det \gbar_w}{\det \gbar_\infty}} = 1+O(e^{-aw})\] and $\left|g^{-1} - \gcheck^{-1}\right|_g = O(e^{-a w})$. Hence, we proved that \[\Delta \phi_0 = \nabla^i \psi_i\] where $\psi_i$ satisfies $\left|\psi\right|_g = O(e^{-(1+a)w})$. The equation for $\phi_1$ reads \[-\Delta \phi_1 = \nabla^i \psi_i.\]

To prove the existence of $\phi_1$, we study first the model case of the hyperbolic half-space. Let $(\bB, g_{\bB})$ denote the ball model of the hyperbolic space: $\bB = B_1(0) \subset \bR^{n+1}$, $g_{\bB} = \left(\frac{1 - |x|^2}{2}\right)^{-2} \delta$, where $\delta$ denotes the Euclidean metric, $\bB_+ = \{x \in \bB | x^0 > 0\}$, $\bB_- = \{x \in \bB | x^0 < 0\}$ and $H = \{x \in \bB | x^0 = 0\}$. For any $x \in \bB$, $x = (x^0, x^1, \ldots, x^n)$ we denote $\xtil = (-x^0, x^1, \ldots, x^n)$.

\begin{lemma}\label{lmIsomHyp} Let $p \in (n+1; \infty)$ and $\delta \in (0; n)$. Assume given $v \in X^{0, p}_\delta (\bB_-, T^*\bB)$,  and $w \in X^{0, q}_\delta(\bB_-, \bR)$, where $\frac{1}{q} \leq \frac{1}{p} + \frac{1}{n+1}$. There exists a unique function $u \in X^{1, p}_\delta(\bB_-, \bR)$ such that
\[\left\lbrace
\begin{aligned}
\Delta_{g_{\bB}} u & = \nabla_{g_{\bB}}^i v_i + w \text{ on $\bB_-$}\\
u & = 0\text{ on $H$}.
\end{aligned}
\right.\]
Further, for some constant $C = C(p, q, \delta)$, \[\left\|u\right\|_{X^{1, p}_\delta(\bB, \bR)} \leq C \left(\left\|v\right\|_{X^{0, p}_\delta(\bB, T^*\bB)} + \left\|w\right\|_{X^{0, q}_\delta(\bB, \bR)}\right).\]
\end{lemma}

\begin{proof} Uniqueness is simple to prove. Indeed, assume that $u \in X^{1, p}_\delta(\bB_-, \bR)$ is such that $\Delta_{g_{\bB}} u = 0$, $u = 0$ on $H$, then by elliptic regularity in geodesic balls, it is easy to get that $u \in C^{2, \alpha}_\delta(\bB_-, \bR)$. If $u \neq 0$, then since $\delta > 0$, $u$ reaches a non-zero extremum in the interior of $\bB_-$ which is absurd.\\

To prove existence, we solve the equation with either $v = 0$ or $w = 0$. In the case $v = 0$, we extend $w$ to $\bB$ by setting $w(x) = -w(\xtil)$ for any $x \in \bB_+$. This extended $w$ belongs to $X^{0, q}_\delta(\bB, \bR)$ and $\left\|w\right\|_{X^{0, q}_\delta(\bB, \bR)} = \left\|w\right\|_{X^{0, q}_\delta(\bB_-, \bR)}$. By the isomorphism theorem \cite[Theorem A.4]{GicquaudSakovich}, there exists a unique function $u_1 \in X^{2, q}_\delta(\bB, \bR)$ such that $\Delta_{g_{\bB}} u_1 = w$ on $\bB$. Since $-u_1(\xtil)$ also satisfies the equation, we deduce that $u_1(x) = -u_1(\xtil)$ for any $x \in \bB$. In particular, $u_1 = 0$ on $H$. By the Sobolev embedding theorem, we infer $u_1 \in X^{1, p}_\delta(\bB_-, \bR)$.\\

If $w = 0$ we also extend $v$ to the whole of $\bB$ by setting $v_0(x) = v_0(\xtil)$ and $v_\alpha(x) = - v_\alpha(\xtil)$ for any $x \in \bB_+$. Let $K(x, y)$ denote the Green kernel of the Laplacian. Then arguing as in \cite[Section 8.11]{GilbargTrudinger}, we see that \[u_2(x) = - \int_{\bB} \nabla_i^y K(x, y) g_{\bB}^{ij}(y) v_j(y) d\mu_{g_\bB}(y)\] solves $\Delta_{g_{\bB}} u_2 = \nabla^i v_i$ on $\bB$. By a calculation similar to \cite[Theorem A.4]{GicquaudSakovich}, this integral makes sense and $u_2 \in X^{0, p}_\delta(\bB, \bR)$. From elliptic theory, we deduce that $u_2 \in X^{1, p}_\delta(\bB, \bR)$ and by an oddness argument similar to the previous one, we get that $u_2 = 0$ on $H$.
\end{proof}

We use the previous lemma to prove a similar statement regarding the domain $D_{\frac{1}{\lambda}}$:

\begin{lemma}\label{lmRegEllip} Let $p \in (n+1; \infty)$, $\delta \in (0; n)$. Assume given $v \in X^{0, p}_\delta (D_{\frac{1}{\lambda}}, T^*M)$,  and $w \in X^{0, q}_\delta(D_{\frac{1}{\lambda}}, \bR)$, where $\frac{1}{q} \leq \frac{1}{n+1}+\frac{1}{p}$. Then, if $\lambda$ is large enough, there exists a unique function $u \in X^{1, p}_\delta(D_{\frac{1}{\lambda}}, \bR)$ such that
\begin{equation}\label{eqPbDirichlet}\left\lbrace
\begin{aligned}
\Delta u & = \nabla^i v_i + w \text{ on $D_{\frac{1}{\lambda}}$}\\
u & = 0\text{ on $\partial D_{\frac{1}{\lambda}} \cap M$}.
\end{aligned}
\right.\end{equation}
Further, for some constant $C = C(p, q, \delta)$, \[\left\|u\right\|_{X^{1, p}_\delta(D_{\frac{1}{\lambda}}, \bR)} \leq C \left(\left\|v\right\|_{X^{0, p}_\delta(D_{\frac{1}{\lambda}}, T^*M)} + \left\|w\right\|_{X^{0, q}_\delta(D_{\frac{1}{\lambda}}, \bR)}\right).\]
\end{lemma}

\begin{proof} The proof of this lemma consists in several steps which we now describe.\\

\noindent $\bullet$ {\sc Claim 1:} If $\lambda > 0$ is large enough, there exists a constant $C > 0$ such that if $u \in X^{1, p}_\delta(M, \bR)$ is such that $\Delta u = \nabla^i v_i + w$ on $D_{\frac{1}{\lambda}}$ for some $v \in X^{0, p}_\delta(D_{\frac{1}{\lambda}}, T^*M)$ and $w \in X^{0, q}_\delta(M, \bR)$, then \[\left\|u\right\|_{X^{1, p}_\delta(D_{\frac{1}{\lambda}}, \bR)} \leq C \left(\left\|v\right\|_{X^{0, p}_\delta(D_{\frac{1}{\lambda}}, T^*M)} + \left\|w\right\|_{X^{0, q}_\delta(D_{\frac{1}{\lambda}}, \bR)}\right).\]

Indeed, we denote $g_{\bB}$ the hyperbolic metric $g_{\bB} = \frac{1}{\rho^2} \left( d\rho^2 + \sum_{\mu=1}^n (dx^\mu)^2\right)$. We compute the difference $\Delta u - \Delta_{g_{\bB}} u$:
\begin{eqnarray*}
\Delta u - \Delta_{g_{\bB}} u
  & = & \frac{1}{\sqrt{\det g}} \partial_i \left(\sqrt{\det g} g^{ij} \partial_j u\right) - \frac{1}{\sqrt{\det g_{\bB}}} \partial_i \left(\sqrt{\det g_{\bB}} g_{\bB}^{ij} \partial_j u\right)\\
  & = & \frac{1}{\sqrt{\det g_{\bB}}} \left[\sqrt{\frac{\det g_{\bB}}{\det g}} \partial_i \left(\sqrt{\det g} g^{ij} \partial_j u\right) - \partial_i \left(\sqrt{\det g_{\bB}} g_{\bB}^{ij} \partial_j u\right)\right]\\
  & = & \frac{1}{\sqrt{\det g_{\bB}}} \left[\partial_i \left(\sqrt{\det g_{\bB}} g^{ij} \partial_j u\right) - \sqrt{\det g} g^{ij} \partial_i \sqrt{\frac{\det g_{\bB}}{\det g}} \partial_j u\right.\\
  & & \left. - \partial_i \left(\sqrt{\det g_{\bB}} g_{\bB}^{ij} \partial_j u\right)\right]\\
  & = & \frac{1}{\sqrt{\det g_{\bB}}} \left[\partial_i \left(\sqrt{\det g_{\bB}} (g^{ij} - g_{\bB}^{ij}) \partial_j u\right) - \sqrt{\det g} g^{ij} \partial_i \sqrt{\frac{\det g_{\bB}}{\det g}} \partial_j u\right]\\
  & = & \frac{1}{\sqrt{\det g_{\bB}}} \left[\partial_i \left(\sqrt{\det g_{\bB}} g_{\bB}^{ik}(g_{\bB kl} g^{lj} - \delta_l^j) \partial_j u\right) - \sqrt{\det g} g^{ij} \partial_i \sqrt{\frac{\det g_{\bB}}{\det g}} \partial_j u\right]\\
  & = & \nabla^k_{g_{\bB}} \left((g_{\bB kl} g^{lj} - \delta_l^j) \partial_j u\right) - \sqrt{\frac{\det g}{\det g_{\bB}}} g^{ij} \partial_i \sqrt{\frac{\det g_{\bB}}{\det g}} \partial_j u.
\end{eqnarray*}
Hence $u$ satisfies
\begin{eqnarray*}
\Delta_{g_{\bB}} u
	& = & \nabla^k_{g_{\bB}} \left((\delta_k^j - g_{\bB kl} g^{lj}) \partial_j u\right) + \nabla^k v_k + \sqrt{\frac{\det g}{\det g_{\bB}}} g^{ij} \partial_i \sqrt{\frac{\det g_{\bB}}{\det g}} \partial_j u + w\\
	& = & \nabla^k_{g_{\bB}} \left((\delta_k^j - g_{\bB kl} g^{lj}) (\partial_j u - v_j)\right) + \nabla^k_{g_{\bB}} v_k +  \sqrt{\frac{\det g}{\det \gcheck}} g^{ij} \partial_i \sqrt{\frac{\det g_{\bB}}{\det g}} \partial_j u + w,
\end{eqnarray*}
where we have done the same calculation for $\nabla^k v_k$ as we did for the Laplacian of $u$. Remark now that $(D_{\frac{1}{\lambda}}, g_{\bB})$ is isometric to $(\bB_-, g_{\bB})$, so we can apply Lemma \ref{lmIsomHyp} to get \[\left\|u\right\|_{X^{1, p}_\delta} \leq C \left( \left\|\vtil\right\|_{X^{0,p}_\delta} + \left\|\wtil\right\|_{X^{0, q}_\delta}\right),\] where
\[
\begin{aligned}
\vtil_k & = v_k + \left((\delta_l^k - g_{\bB kl} g^{lj}) (\partial_j u - v_j)\right),\\
\wtil   & = \sqrt{\frac{\det g}{\det g_{\bB}}} g^{ij} \partial_i \sqrt{\frac{\det g_{\bB}}{\det g}} \partial_j u + w.
\end{aligned}
\]
Remark that for some constant $C > 0$, $\left\|\vtil\right\|_{X^{0, p}_\delta(D_{\frac{1}{\lambda}}, T^*M)} \leq \left\|v\right\|_{X^{0, p}_\delta(D_{\frac{1}{\lambda}}, T^*M)} + \frac{C}{\lambda^\alpha} \left\|du - v\right\|_{X^{0, p}_\delta(D_{\frac{1}{\lambda}}, T^*M)}$ since we assumed that $\gbar_{ij}(\phat) = \delta_{ij}$. Similarly, $\wtil$ can be estimated as follows:
\[\begin{aligned}
\left\|\wtil\right\|_{X^{0, q}_\delta}
  & \leq \left\|w\right\|_{X^{0, q}_\delta} + \left\|\sqrt{\frac{\det g}{\det g_{\bB}}} g^{ij} \partial_i \sqrt{\frac{\det g_{\bB}}{\det g}} \partial_j u\right\|_{X^{0, q}_\delta}\\
  & \leq \left\|w\right\|_{X^{0, q}_\delta} + \frac{1}{\lambda^{n+1}} \left\|g^{ij} \partial_i \sqrt{\frac{1}{\det \gbar}} \partial_j u\right\|_{X^{0, q}_\delta},   
\end{aligned}\]
where we used the fact that $\sqrt{\det g_{\bB}} = \rho^{-(n+1)}$ and $\sqrt{\det g} = \rho^{-(n+1)} \sqrt{\det \gbar}$ in the coordinate system $(\rho, x^1, \ldots, x^n)$. From Section \ref{secAlt1}, we know that \[d \frac{1}{\sqrt{\det \gbar}} = - \frac{1}{2 (\det \gbar)^{\frac{3}{2}}} \gbar^{ij} d\gbar_{ij} \in X^{0, p}_a(D_{\frac{1}{\lambda}}, \bR).\] Hence, there exists a constant $C > 0$ such that \[\left\|\sqrt{\frac{\det g}{\det \gcheck}} g^{ij} \partial_i \sqrt{\frac{\det g_{\bB}}{\det g}} \partial_j u\right\|_{X^{0, q}_\delta} \leq \frac{C}{\lambda^{\alpha}} \left\|u\right\|_{X^{1, p}_\delta}.\] As a consequence, we have proved that for a certain constant $C > 0$, \[\left\|u\right\|_{X^{1, p}_\delta} \leq C \left( \left\|v\right\|_{X^{0,p}_\delta} + \left\|w\right\|_{X^{0, q}_\delta}\right) + \frac{C}{\lambda^\alpha} \left\|u\right\|_{X^{1, p}_\delta}.\] So if $\frac{C}{\lambda^\alpha} \leq \frac{1}{2}$, \[\left\|u\right\|_{X^{1, p}_\delta} \leq 2C \left( \left\|v\right\|_{X^{0,p}_\delta} + \left\|w\right\|_{X^{0, q}_\delta}\right).\] This concludes the first part of the lemma.\\

\noindent $\bullet$ {\sc Claim 2:} Given $\delta' \in (0; \delta)$, and $\epsilon > 0$ there exist $v^\epsilon \in C^{1,0}_\delta(D_{\frac{1}{\lambda}}, \bR)$, $w^\epsilon \in C^{0,0}_\delta(D_{\frac{1}{\lambda}}, \bR)$ such that $\left\|v^\epsilon\right\|_{X^{0, p}_\delta} \leq 2\left\|v\right\|_{X^{0, p}_\delta}$, $\left\|w^\epsilon\right\|_{X^{0, q}_\delta} \leq 2\left\|w\right\|_{X^{0, q}_\delta}$ and such that $v^\epsilon \to v$ in $X^{0, p}_{\delta'}$, $w^\epsilon \to w$ in $X^{0, q}_{\delta'}$ when $\epsilon \to 0$ (see \cite[Proposition 2.2]{GicquaudSakovich}).\\

The proof is similar for both $v$ and $w$, hence we only give it for $v$. Choose an arbitrary $\epsilon > 0$. Let $\chi: \bR_+ \to \bR$ be a smooth cut-off function such that $0 \leq \chi \leq 1$, $\chi(x) = 1$ for any $x \in [0; 1]$, $\chi(x) = 0$ for any $x \geq 2$. Remark that the 1-forms $\left(1 - \chi\left(\frac{\rho}{2^i}\right)\right) v$ are such that \[\left\lbrace\begin{aligned}
\left\|\left(1 - \chi\left(\frac{\rho}{2^i}\right)\right) v\right\|_{X^{0, p}_\delta} & \leq \left\|v\right\|_{X^{0, p}_\delta},\\
v - \left(1 - \chi\left(\frac{\rho}{2^i}\right)\right) v & = \chi\left(\frac{\rho}{2^i}\right) v \to 0 \text{ in } X^{0, p}_{\delta'}\text{ when } i \to \infty.                                                                                                                                                                                                                                                                                                                                                                                         
\end{aligned}\right.\]

Selecting $i$ large enough, we can assume that $\left\| v - \left(1 - \chi\left(\frac{\rho}{2^i}\right)\right) v \right\|_{X^{0, p}_{\delta'}} \leq \frac{\epsilon}{2}$. We now use the fact that $\chi\left(\frac{\rho}{2^i}\right) v$ has compact support to find $v^\epsilon \in C^{1, 0}_c(D_{\frac{1}{\lambda}}, T^*M)$ such that $\left\|\chi\left(\frac{\rho}{2^i}\right) v - v^\epsilon\right\|_{X^{0, p}_{\delta'}} < \frac{\epsilon}{2}$. From the triangle inequality, we obtain that $\left\|v - v^\epsilon\right\|_{X^{0, p}_{\delta'}} < \epsilon$. $v^\epsilon$ can be constructed by a mollification argument, see e.g. \cite[Chapter 7]{GilbargTrudinger}, thus we can select $v^\epsilon$ such that $\left\|v^\epsilon\right\|_{X^{0, p}_\delta} \leq 2\left\|\left(1 - \chi\left(\frac{\rho}{2^i}\right)\right)v\right\|_{X^{0, p}_\delta} \leq 2\left\|v\right\|_{X^{0, p}_\delta}$.\\

\noindent $\bullet$ {\sc Claim 3:} For any $\epsilon > 0$, there exists a solution $u^\epsilon$ to the equation 
\[\left\lbrace
\begin{aligned}
\Delta u^\epsilon & = \nabla^i v^\epsilon_i + w^\epsilon \text{ on $D_{\frac{1}{\lambda}}$}\\
u^\epsilon & = 0\text{ on $\partial D_{\frac{1}{\lambda}} \cap M$}.
\end{aligned}
\right.\]

The proof is similar to the proof of Lemma \ref{lmConstructionT} and is omitted.\\

\noindent $\bullet$ {\sc Claim 4:} $u^\epsilon$ converge to $u \in X^{1, p}_\delta$ which solves \eqref{eqPbDirichlet}.\\

We finally prove that when $\epsilon \to 0$, the functions $u^\epsilon$ converge to the solution $u$ of the initial PDE. We first apply Step $1$ for the $X^{1, p}_{\delta'}$-space and get that there exists a constant $C > 0$ such that for any $\epsilon, \epsilon' > 0$, \[\left\|u^\epsilon - u^{\epsilon'}\right\|_{X^{1, p}_{\delta'}(D_{\frac{1}{\lambda}}, \bR)} \leq C \left(\left\|v^\epsilon - v^{\epsilon'}\right\|_{X^{0, p}_{\delta'}(D_{\frac{1}{\lambda}}, T^*M)} + \left\|w^\epsilon - w^{\epsilon'}\right\|_{X^{0, q}_{\delta'}(D_{\frac{1}{\lambda}}, \bR)}\right).\] Thus $u^\epsilon$ is of Cauchy type in $X^{1, p}_{\delta'}(D_{\frac{1}{\lambda}}, \bR)$ when $\epsilon \to 0$ so there exists a function $u \in X^{1, p}_{\delta'}(D_{\frac{1}{\lambda}}, \bR)$ such that
\[
\left\lbrace
\begin{aligned}
\Delta u & = \nabla^i v_i + w \text{ on $D_{\frac{1}{\lambda}}$}\\
u & = 0\text{ on $\partial D_{\frac{1}{\lambda}} \cap M$}.
\end{aligned}
\right.
\]
Further, the functions $u^\epsilon$ are uniformly bounded in $X^{1, p}_\delta(D_{\frac{1}{\lambda}}, \bR)$. Indeed, from Step 1,
\[\begin{aligned}
\left\|u^\epsilon\right\|_{X^{1, p}_\delta(D_{\frac{1}{\lambda}}, \bR)}
  & \leq C \left(\left\|v^\epsilon\right\|_{X^{0, p}_\delta(D_{\frac{1}{\lambda}}, T^*M)} + \left\|w^\epsilon\right\|_{X^{0, q}_\delta(D_{\frac{1}{\lambda}}, \bR)}\right)\\
  & \leq 2 C \left(\left\|v\right\|_{X^{0, p}_\delta(D_{\frac{1}{\lambda}}, T^*M)} + \left\|w\right\|_{X^{0, q}_\delta(D_{\frac{1}{\lambda}}, \bR)}\right).   
\end{aligned}\]
This proves that $u \in X^{1, p}_\delta(D_{\frac{1}{\lambda}}, \bR)$.\\

Uniqueness of the function $u$ is an easy exercise.
\end{proof}

As a direct application of Lemma \ref{lmRegEllip}, we get that if $\lambda > 0$ is large enough, there exists a unique function $\phi_1 \in X^{1, p}_{1+a}(D_{\frac{1}{\lambda}}, \bR)$ such that $-\Delta \phi_1 = \nabla^i \psi_i$. The function $\phi = \phi_0 + \phi_1$ we constructed on $D_{\frac{1}{\lambda}}$ is harmonic and is such that $\left\langle dw, d\phi\right\rangle = \left\langle dw, d\phi_1\right\rangle \in X^{0, p}_{a+1}(D_{\frac{1}{\lambda}}, \bR)$.\\

Our next task is to prove that the function $\phi$ we constructed belongs to $C^{1, \alpha}(\overline{D_{\frac{1}{4\lambda}}})$. To this end, we first give an equation for $\left\langle dt, d\phi\right\rangle$:

\begin{lemma}\label{lmDtDphi} Let $\phi: \Omega \to \bR$ be an harmonic function, the following formula holds:
\begin{equation}
\label{eqDtDphi}
-\Delta \left\langle dt, d\phi\right\rangle - (n-1) \left\langle dt, d\phi\right\rangle = -2 \left( \ric + n g\right) (dt, d\phi) - 2\left\langle \mathring{\hess}(t), \hess(\phi)\right\rangle.
\end{equation}
\end{lemma}

\begin{proof} The proof is a simple calculation:
\begin{eqnarray*}
\Delta \left\langle dt, d\phi\right\rangle
	& = & \left\langle \Delta dt, d\phi\right\rangle + \left\langle dt, \Delta d\phi\right\rangle + 2 \left\langle \hess(t), \hess(\phi)\right\rangle\\
	& = & (n+1) \left\langle dt, d\phi\right\rangle + 2 \ric(dt, d\phi) + 2 \left\langle \hess(t), \hess(\phi)\right\rangle
\end{eqnarray*}
where we used the fact that
\begin{eqnarray*}
\Delta \nabla_i t & = & g^{ab} \nabla_a \nabla_b \nabla_i t\\
  & = & g^{ab} \nabla_a \nabla_i \nabla_b t\\
  & = & g^{ab} \left(\nabla_i \nabla_a \nabla_b t - \riemuddd{j}{b}{a}{i} \nabla_j t\right)\\
  & = & \nabla_i \Delta t + \ricud{j}{i} \nabla_j t\\
  & = & (n+1) \nabla_i t + \ricud{j}{i} \nabla_j t,
\end{eqnarray*}
hence \[\Delta dt = (n+1) dt + \ric(dt),\] and similarly \[\Delta d\phi = \ric(d\phi).\] Formula \eqref{eqDtDphi} follows.
\end{proof}

Remark that $\left|d\phi_0\right|_{\gbar}$ is bounded, hence $d\phi_0 \in X^{0, p}_{1}(D_{\frac{1}{\lambda}}, T^*M)$. From the fact that $\phi_1 \in X^{1, p}_{a+1}$, we also have that $d\phi_1 \in X^{0, p}_{1}(D_{\frac{1}{\lambda}}, T^*M)$. So $d\phi = d\phi_0 + d\phi_1 \in X^{0, p}_0(M, T^*M)$. Remark that $D_{\frac{5}{6\lambda}}$ lies at some positive distance (for the metric $g$) from $\partial D_{\frac{1}{\lambda}}$, hence, since $d\phi$ satisfies the elliptic equation $\Delta d\phi = \ric(d\phi)$ on $D_{\frac{1}{\lambda}}$, we conclude from standard elliptic theory in harmonic charts that $d\phi \in X^{2, p}_{1}(D_{\frac{3}{4\lambda}}, T^*M)$. In particular, $\hess(\phi) \in X^{0, p}_{1}$. From Equation \eqref{eqDtDphi}, we get:

\begin{multline*}
-\Delta \left\langle dt, d\phi\right\rangle - (n-1) \left\langle dt, d\phi\right\rangle\\ = -2 \left( \ric + n g\right) (dt, d\phi) - 2\left\langle \mathring{\hess}(t), \hess(\phi)\right\rangle \in X^{0, p}_a(D_{\frac{3}{4\lambda}}, \bR).
\end{multline*}

This leads to the following estimate:

\begin{equation}\label{lmEstDtDphi}
\forall p \in (n+1; \infty),~ \left\langle dt, d\phi\right\rangle \in X^{2, p}_a(D_{\frac{2}{3\lambda}}, \bR).
\end{equation}

Remark that $\phi_0 \in C^{1, \alpha}\left(\overline{D_{\frac{1}{\lambda}}}\right)$ while $\phi_1 \in C^{1, \alpha}_{a+1}(D_{\frac{1}{\lambda}}) \subset C^{1, \alpha} \left(\overline{D_{\frac{1}{\lambda}}}\right)$ (see \cite[Lemma 3.7]{LeeFredholm}). This ends the proof of Proposition \ref{propCoordsHarm}.

\section{Estimates for harmonic coordinates}\label{secEstimates}
% Estimates for harmonic coordinates
In this section, we denote by $g_{ij}$ the components of the metric tensor with respect to the coordinate system $(w, y^1, \ldots, y^n)$. We shall use repeatedly the following formula: For any function $\phi : M \to \bR$,
\begin{equation}\label{eqTransConfHessian}
\hessbardd{i}{j} \phi = \hessdd{i}{j} \phi + \nabla_i w \nabla_j \phi + \nabla_i \phi \nabla_j w - \left\langle dw, d\phi\right\rangle g_{ij}.
\end{equation}
We proved in Lemma \ref{lmConstructionT} that \[\hess(t) = t g + T\] where $T$ is a traceless symmetric 2-tensor such that \[T \in X^{1, p}_{a-1}\] for any $p \in (1; \infty)$.  We first give the order of magnitude of each term of the metric tensor $g$:

\begin{lemma}[$L^\infty$-estimates for the components of the metric tensor]\label{lmEstimatesHarm}
 The components of the metric tensor $g$ and of its inverse satisfy the following estimates:
\[\begin{array}{cc}
\begin{aligned}
g^{00} & = 1 + O(e^{-aw}),\\
g^{0\mu} & = O(e^{-(a+1)w}),
\end{aligned} &
\begin{aligned}
g_{00} & = 1 + O(e^{-aw})\\
g_{0\mu} & = O(e^{(1-a)w})
\end{aligned}\\
C^{-1} e^{-2w} \delta^{\mu\nu} \leq g^{\mu\nu} \leq C e^{-2w} \delta^{\mu\nu}, &
C^{-1} e^{2w} \delta_{\mu\nu} \leq g_{\mu\nu} \leq C e^{2w} \delta_{\mu\nu}
\end{array}\]
for some constant $C > 0$, where the estimates on the last line are inequalities between quadratic forms.
\end{lemma}

\begin{proof}
We first remark that $g^{00} = \left\|dw\right\|^2 = N^{-2}$ where $N$ is the function appearing in Equation \eqref{eqMetricZeroShift}. Hence, the estimate for $g^{00}$ follows from Lemma \ref{lmLinftyEstimateZeroShift}. From the estimate \eqref{eqDtDphi}, we have \[g^{0\mu} = \left\langle dw, dy^\mu\right\rangle = O(e^{-(a+1)w}).\] Finally from Proposition \ref{propCoordsHarm}, we know that the 1-forms $dy^\mu$ are linearly independent at any point of $\overline{D_{\frac{1}{\lambda}}}$. By compactness of $\overline{D_{\frac{1}{\lambda}}}$, there exists a constant $C > 0$ such that \[ C^{-1} \delta^{\mu\nu} \leq \gbar^{\mu\nu} = \left\langle dy^\mu, dy^\nu \right\rangle_{\gbar} \leq C \delta^{\mu\nu}.\] This immediately gives the estimate for $g^{\mu\nu} = e^{-2w} \gbar^{\mu\nu}$.\\

Denote $h^{\mu\nu} = g^{\mu\nu}$ and $(h_{\mu\nu})_{\mu,\nu}$ the inverse matrix of $(h^{\mu\nu})_{\mu,\nu}$. Note that $(h_{\mu\nu})_{\mu,\nu}$ is an $n \times n$ matrix containing only tangential components. From our previous estimate, \[C^{-1} e^{2w} \delta_{\mu\nu} \leq h_{\mu\nu} \leq C e^{2w} \delta_{\mu\nu}.\] By Gaussian elimination, it is easy to prove that
\[\begin{aligned}
 g_{00}     & = \left(g^{00} - h_{\mu\nu} g^{0\mu} g^{0\nu}\right)^{-1},\\
 g_{0\mu}   & = - g_00 g^{0\nu} h_{\nu\mu},\\
 g_{\mu\nu} & = h_{\mu\nu} + h_{\mu\sigma} h_{\nu\tau} g_00 g^{\sigma 0} g^{\tau 0}.
\end{aligned}\]
Estimates for $g_{00}$, $g_{0\mu}$ and $g_{\mu\nu}$ follow, replacing the constant $C$ by some larger constant.
\end{proof}

\subsection{$W^{1, p}$-estimates for $\hessbar(\rho)$}

Our next step is to study the regularity of $\rho = \frac{1}{t}$.

\begin{lemma}[$W^{1, p}$-estimates for $\hessbar(\rho)$]\label{lmEstimatesHessRho} If $a \in (0; 2)$, the following estimates hold for the derivatives of $\rho$:
\[\left\lbrace
\begin{aligned}
\left|d\rho\right|_{\gbar}^2 & = 1 + O(e^{-aw}),\\
\hessbar(\rho) & = O(e^{-(1+a)w}),\\
\nablabar~\hessbar(\rho) & \in X^{0, p}_{a+1} \quad\text{for any $p \in (n+1; \infty)$}.
\end{aligned}
\right.\]
\end{lemma}

\begin{proof}
First remark that $d\rho = -e^{-w} dw$, hence \[\left|d\rho\right|_{\gbar}^2 = e^{2w} \left|e^{-w} dw\right|_{\gbar}^2 = \left|dw\right|_g^2 = g^{00} = 1 + O(e^{-aw}),\] see Lemma \ref{lmEstimatesHarm}. The Hessian of $\rho$ is given by: \[\hess(\rho) = -\frac{\hess(t)}{t^2} + 2\frac{dt \otimes dt}{t^3},\] thus

\begin{eqnarray*}
\hessbar(\rho) & = & \hess(\rho) + dw \otimes d\rho + d\rho \otimes dw - g(dw, d\rho) g\\
  & = & -\frac{g}{t} - \frac{T}{t^2} + 2 \frac{dt \otimes dt}{t^3} - 2 \frac{dt \otimes dt}{t^3} + g\left(\frac{dt}{t^2}, \frac{dt}{t}\right) g\\
  & = & \frac{1}{t} \left(|dw|^2_g - 1\right) g - \frac{T}{t^2}\\
  & = & O(e^{-(1+a)w}).
\end{eqnarray*}

We finally compute $\nablabar~\hessbar(\rho)$. Remark that \[\hessbardd{i}{j} \rho = \frac{1}{\rho} \left(|d\rho|^2_{\gbar} - 1\right) \gbar_{ij} - \rho^2 T_{ij}.\] Hence,

\begin{eqnarray*}
\nablabar_i \hessbardd{j}{k} \rho
  & = & -\frac{\nablabar_i \rho}{\rho^2} \left(|d\rho|^2_{\gbar} - 1\right) \gbar_{jk} + \frac{2}{\rho} \gbar^{ab} \hessbardd{i}{a} \rho \nablabar_b \rho \gbar_{jk} - 2 \rho \nablabar_i \rho T_{jk} - \rho^2 \nablabar_i T_{jk}\\
  & = & \frac{\nablabar_i \rho}{\rho^2} \left(|d\rho|^2_{\gbar} - 1\right) \gbar_{jk} + 2 \rho \gbar^{ab} T_{ia} \nablabar_b \rho \gbar_{jk} - 2 \rho \nablabar_i \rho T_{jk} - \rho^2 \nablabar_i T_{jk}\\
  & = & \left(|d\rho|^2_{\gbar} - 1\right) \nablabar_i \rho g_{jk} + 2\rho g^{ab} T_{ia} \nabla_b \rho g_{jk} - 2 \rho \nabla_i \rho T_{jk} - \rho^2 \nablabar_i T_{jk}\\
\end{eqnarray*}

Remark that, from the conformal transformation law for the Levi-Civita connection, we have $\nablabar T = \nabla T + dw * T$, hence, from Lemma \ref{lmConstructionT}, we get $\nablabar T \in X^{0, p}_{a-1}$. From the previous equality, we infer $\nablabar \hessbar(\rho) \in X^{0, p}_{a+1}$.
\end{proof}

\subsection{$L^p$-estimate for $\hessbar (y^\mu)$}
From the previous section, we know that \[\left|dy^\mu\right|_g = O(e^{-w}).\] We shall now concentrate on getting higher order estimates for $y^\mu$. Since $dy^\mu$ satisfies the elliptic equation \[\Delta dy^\mu = \ric(dy^\mu),\] by standard elliptic regularity, we get first a very rough estimate:
\begin{equation}
\label{eqRoughEstimateHessianXp}
\left\|\hess(y^\mu)\right\|_{X^{1, p}_1} \leq \left\|d y^\mu\right\|_{X^{2, p}_1} < \infty
\end{equation}
valid for any $p \in (1; \infty)$. In particular if we select $p > n+1$, we get:
\begin{equation}
\label{eqRoughEstimateHessian}
\left|\hess(y^\mu)\right|_g \leq C e^{-w}.
\end{equation}

We now get an improved estimate for the Hessian of $y^\mu$ with respect to the metric $\gbar$. We concentrate on each type (purely tangential, purely radial and mixed terms) separately. Remark that $\hessbar (y^\mu)$ also satisfies the rough estimate 
\begin{equation}
\label{eqRoughEstimateHessianBar}
\left|\hessbar(y^\mu)\right|_g = O(e^{-w}).
\end{equation}

\noindent $\bullet$ \textbf{Estimation of $\hessbardd{0}{0} y^\mu$:} We compute first
\begin{eqnarray*}
\left\langle dw, d\left\langle dw, dy^\mu\right\rangle\right\rangle
  & = & \nabla^i w \nabla_i \left(\nabla^j w \nabla_j y^\mu\right)\\
  & = & \nabla^i w \nabla^j w \hessdd{i}{j} y^\mu + \hessdd{i}{j} w \nabla^i w \nabla^j y^\mu\\ 
  & = & \left\langle dw \otimes dw, \hess (y^\mu)\right\rangle + \left(g_{ij} - \nabla_i w \nabla_j w\right) \nabla^i w \nabla^j y^\mu\\
  & = & \left\langle dw \otimes dw, \hess (y^\mu)\right\rangle + \left(1 - \left|dw\right|_g^2\right) \left\langle dw, dy^\mu\right\rangle + \left\langle\frac{T}{t}, dw \otimes dy^\mu \right\rangle.
\end{eqnarray*}

From Proposition \ref{propCoordsHarm}, we have $\left\langle dw, dy^\mu\right\rangle \in X^{2, p}_{a+1}$. So, by the Sobolev embedding theorem, \[\left\langle dw, d\left\langle dw, dy^\mu\right\rangle\right\rangle = O(e^{-(a+1)w}).\] Similarly, from Lemma \ref{lmLinftyEstimateZeroShift}, $\left(1 - \left|dw\right|_g^2\right) \left\langle dw, dy^\mu\right\rangle = O(e^{-(2a+1)w})$ and from Lemma \ref{lmConstructionT}, $\left\langle\frac{T}{t}, dw \otimes dy^\mu \right\rangle = O(e^{-(a+1)w})$. This proves that \[\left\langle dw \otimes dw, \hess (y^\mu)\right\rangle = O(e^{-(a+1)w}).\] From Formula \eqref{eqTransConfHessian}, we obtain:
\begin{equation}\label{eqEstimateChristoffel}
\begin{aligned}
\left\langle dw \otimes dw, \hessbar (y^\mu)\right\rangle
 & = \left\langle dw \otimes dw, \hess (y^\mu) + dw \otimes dy^\mu + dy^\mu \otimes dw - \left\langle dw, dy^\mu\right\rangle g \right\rangle\\
 & = O(e^{-(a+1)w}). 
\end{aligned}
\end{equation}
We now remark that:
\begin{equation} \label{eqDwDwHess}
\begin{aligned}
\left\langle dw \otimes dw, \hessbar (y^\mu)\right\rangle
  & = g^{0i} g^{0j} \hessbardd{i}{j} y^\mu\\
  & = (g^{00})^2 \hessbardd{0}{0} y^\mu + 2 g^{00} g^{0\nu} \hessbardd{0}{\nu} y^\mu + g^{0\nu} g^{0\sigma} \hessbardd{\nu}{\sigma} y^\mu.
\end{aligned}
\end{equation}

Estimate \eqref{eqRoughEstimateHessian} gives

\begin{eqnarray*}
\left|\hessbardd{0}{\nu} y^\mu \right| & \leq & |\partial_w|_g |\partial_\nu|_g |\hessbar(y^\mu)|_g\\
  & = & O(1)\\
\left|\hessbardd{\nu}{\sigma} y^\mu \right| & \leq & |\partial_\nu|_g |\partial_\sigma|_g |\hessbar(y^\mu)|_g\\
  & = & O(e^w).
\end{eqnarray*}

Inserting these estimates in Equation \eqref{eqDwDwHess} and using Lemma \ref{lmEstimatesHarm}, we obtain \[\hessbardd{0}{0} y^\mu = O(e^{-(a+1)w}).\]

\noindent $\bullet$ \textbf{Estimation of $\hessbardd{0}{\alpha} y^\mu$:} The estimate follows the same line. We decompose
\begin{eqnarray*}
\left\langle dy^\nu, d \left\langle dw, dy^\nu\right\rangle\right\rangle
  & = & \nabla^i y^\nu \nabla_i \left(\nabla^j w \nabla_j y^\mu\right)\\
  & = & \nabla^i y^\nu \nabla^j w \hessdd{i}{j} y^\mu + \nabla^i y^\nu \hessdd{i}{j} w \nabla_j y^\mu\\
  & = & \left\langle dy^\nu \otimes dw, \hess(y^\nu)\right\rangle + \nabla^i y^\nu \nabla^j y^\mu \left( g_{ij} - \nabla_i w \nabla_j w + \frac{T_{ij}}{t}\right)\\
  & = & \left\langle dy^\nu \otimes dw, \hess(y^\nu)\right\rangle + \left\langle dy^\nu, dy^\mu\right\rangle - \left\langle dw, dy^\nu\right\rangle \left\langle dw, dy^\mu\right\rangle\\
  & & \qquad + \left\langle dy^\nu\otimes dy^\mu, \frac{T}{t}\right\rangle.
\end{eqnarray*}

Hence, \[\left\langle dy^\nu \otimes dw, \hess(y^\mu)\right\rangle + \left\langle dy^\nu, dy^\mu\right\rangle = O(e^{-(a+2)w}).\] From Formula \eqref{eqTransConfHessian}, we get:
\begin{eqnarray*}
\left\langle dy^\nu \otimes dw, \hessbar(y^\mu)\right\rangle
  & = & \left\langle dy^\nu \otimes dw, \hess(y^\mu) + dw \otimes dy^\mu + dy^\mu \otimes dw - \left\langle dy^\mu, dw\right\rangle g \right\rangle\\
  & = & \left\langle dy^\nu \otimes dw, \hess(y^\mu) \right\rangle + \left\langle dy^\nu, dy^\mu\right\rangle + O(e^{-(2a+2)w})\\
  & = & O(e^{-(a+2)w}).
\end{eqnarray*}

We decompose the scalar product and get:
\begin{eqnarray*}
\left\langle dy^\nu \otimes dw, \hessbar(y^\mu)\right\rangle
  & = & g^{\nu i} g^{0j} \hessbardd{i}{j} y^\mu\\
  & = & g^{\nu\sigma} g^{00} \hessbardd{\sigma}{0} y^\nu + g^{\nu 0} g^{00} \hessbardd{0}{0} y^\nu + g^{\nu\sigma} g^{0\alpha} \hessbardd{\sigma}{\alpha} y^\mu + g^{\nu 0} g^{0\alpha} \hessbardd{0}{\alpha} y^\mu\\
  & = & g^{\nu\sigma} g^{00} \hessbardd{\sigma}{0} y^\nu + O(e^{-2(a+1)w}).
\end{eqnarray*}

We multiply this estimate by $g_{\alpha\nu}$ and sum over $\nu$. From the fact that $g_{\alpha\nu} g^{\nu\sigma} = \delta_\alpha^\sigma - g_{\alpha 0} g^{0\sigma} = \delta_\alpha^\sigma + O(e^{-2aw})$, we conclude that \[\hessbardd{0}{\alpha} y^\mu = O(e^{-aw}).\]

\noindent $\bullet$ \textbf{Estimation of $\hessbardd{\alpha}{\beta} y^\mu$:} This estimate is the only one that relies on Lemma \ref{lmMovingWindow1}. We compute:
\begin{equation}\label{eqHessTT}
\hessdd{\alpha}{\beta} \left\langle dw, d y^\mu\right\rangle
  = g^{ij} \left(\nabla_\alpha\nabla_\beta \nabla_i w \nabla_j y^\mu + \nabla_i w \nabla_\alpha\nabla_\beta\nabla_j y^\mu
  + \hessdd{\alpha}{i} w \hessdd{\beta}{j} y^\mu + \hessdd{\beta}{i} w \hessdd{\alpha}{j} y^\mu\right).
\end{equation}

We compute the last two terms first. Note that they only differ by the exchange of $\alpha$ and $\beta$.

\begin{eqnarray*}
g^{ij} \hessdd{\alpha}{i} w \hessdd{\beta}{j} y^\mu
  & = & g^{ij} \left(g_{\alpha i} - \nabla_\alpha w \nabla_i w + \frac{T_{\alpha i}}{t}\right) \hessdd{\beta}{j} y^\mu\\
  & = & \hessdd{\beta}{\alpha} y^\mu + g^{ij} \frac{T_{\alpha i}}{t} \hessdd{\beta}{j} y^\mu.
\end{eqnarray*}
Remark that we used the fact that $\nabla_\alpha w = 0$. Hence,
\[g^{ij} \left(\hessdd{\alpha}{i} w \hessdd{\beta}{j} y^\mu + \hessdd{\beta}{i} w \hessdd{\alpha}{j} y^\mu\right)
  = 2 \hessdd{\alpha}{\beta} y^\mu + g^{ij} \left(\frac{T_{\alpha i}}{t} \hessdd{\beta}{j} y^\mu + \frac{T_{\beta i}}{t} \hessdd{\alpha}{j} y^\mu\right).\]

Similarly, we compute the first term of Equation \eqref{eqHessTT}:
\begin{eqnarray*}
g^{ij} \nabla_\alpha \nabla_\beta \nabla_i w \nabla_j y^\mu
  & = & g^{ij} \nabla_{\alpha} \left(g_{\beta i} - \nabla_\beta w \nabla_i w + \frac{T_{\beta i}}{t}\right) \nabla_j y^\mu\\
  & = & - g^{ij} \left(\hessdd{\alpha}{\beta} w \nabla_i w + \nabla_\beta w \hessdd{\alpha}{i} w - \nabla_\alpha \frac{T_{\beta i}}{t}\right) \nabla_j y^\mu\\
  & = & -\left(g_{\alpha\beta} + \frac{T_{\alpha\beta}}{t}\right) \left\langle dw, d y^\mu\right\rangle + \nabla_\alpha \left(\frac{T_{\beta i}}{t}\right) \nabla^i y^\mu.
\end{eqnarray*}

We finally compute the second term of Equation \eqref{eqHessTT}:
\begin{eqnarray*}
g^{ij} \nabla_i w \nabla_\alpha \nabla_\beta \nabla_j y^\mu
  & = & g^{ij} \nabla_i w \nabla_j \hessdd{\alpha}{\beta} y^\mu - g^{ij} \riemuddd{k}{\beta}{\alpha}{j} \nabla_i w \nabla_k y^\mu\\
  & = & g^{ij} \nabla_i w \nabla_j \hessdd{\alpha}{\beta} y^\mu + g^{ij} \left(\delta^k_\alpha g_{\beta j} - \delta^k_j g_{\beta\alpha} - \tuddd{E}{k}{\beta}{\alpha}{j} \right) \nabla_i w \nabla_k y^\mu\\
  & = & g^{ij} \nabla_i w \nabla_j \hessdd{\alpha}{\beta} y^\mu - g_{\alpha\beta} \left\langle dw, d y^\mu\right\rangle - \tuddd{E}{k}{\beta}{\alpha}{j} \nabla_i w \nabla_k y^\mu.
\end{eqnarray*}

Combining all three calculations, we get that:
\begin{equation}\label{eqEstimateHessTT}
\begin{aligned}
g^{ij} \nabla_i w \nabla_j \hessdd{\alpha}{\beta} y^\mu + 2 \hessdd{\alpha}{\beta} y^\mu =
  & \hessdd{\alpha}{\beta} \left\langle dw, d y^\mu\right\rangle + \left( 2 g_{\alpha\beta} + \frac{T_{\alpha\beta}}{t}\right) \left\langle dw, d y^\mu\right\rangle\\
  & - g^{ij} \left(\frac{T_{\alpha i}}{t} \hessdd{\beta}{j} y^\mu + \frac{T_{\beta i}}{t} \hessdd{\alpha}{j} y^\mu\right)\\
  & - \nabla_\alpha \left(\frac{T_{\beta i}}{t}\right) \nabla^i y^\mu + \tuddd{E}{k}{\beta}{\alpha}{j} \nabla_i w \nabla_k y^\mu.
\end{aligned}
\end{equation}

We decompose further the left hand side of Equation \ref{eqEstimateHessTT}:
\[g^{ij} \nabla_i w \nabla_j \hessdd{\alpha}{\beta} y^\mu = g^{00} \nabla_0 \hessdd{\alpha}{\beta} y^\mu + g^{0\nu} \nabla_\nu \hessdd{\alpha}{\beta} y^\mu,\]
$\nabla_0 \hessdd{\alpha}{\beta} y^\mu$ can be rewritten:
\begin{eqnarray*}
\nabla_0 \hessdd{\alpha}{\beta} y^\mu
  & = & \partial_0 \hessdd{\alpha}{\beta} y^\mu - \Gamma^i_{0\alpha} \hessdd{i}{\beta} y^\mu - \Gamma^i_{0\beta} \hessdd{\alpha}{i} y^\mu\\
  & = & \partial_0 \hessdd{\alpha}{\beta} y^\mu - \Gamma^0_{0\alpha} \hessdd{0}{\beta} y^\mu - \Gamma^0_{0\beta} \hessdd{\alpha}{0} y^\mu - \Gamma^\nu_{0\alpha} \hessdd{\nu}{\beta} y^\mu - \Gamma^\nu_{0\beta} \hessdd{\alpha}{\nu} y^\mu.
\end{eqnarray*}
Now remark that $\Gamma^0_{0\alpha} = -\hessdd{0}{\alpha} w$ and $\Gamma^\nu_{0\alpha} = -\hessdd{0}{\alpha} y^\nu$. Thus,
\begin{eqnarray*}
\hessdd{0}{\alpha} w & = & g_{0\alpha} - \nabla_0 w \nabla_\alpha w + \frac{T_{0\alpha}}{t} = O(e^{(1-a)w}),\\
\Gamma^0_{0\alpha} \hessdd{0}{\beta} y^\mu & = & O(e^{(1-2a)w}).\\
\end{eqnarray*}
To estimate $\Gamma^\nu_{0\alpha}$ we use \eqref{eqEstimateChristoffel}. Straightforward calculations yield that \[\Gamma^\nu_{0\alpha} = \delta^\nu_\alpha + O(e^{-aw}).\] As a consequence, we proved that \[\nabla_0 \hessdd{\alpha}{\beta} y^\mu = \partial_0 \hessdd{\alpha}{\beta} y^\mu - 2 \hessdd{\alpha}{\beta} y^\mu + \left[O(e^{-aw}) * \hess (y^\mu)\right]_{\alpha\beta} + O(e^{(1-2a)w}),\] where $*$ only involves tangential components.\\

We are now in a position to apply Lemma \ref{lmMovingWindow1}. We first remark that we can rescale the coordinates $\rho$ and $y^\mu$ by some (large) factor so that $D_{\frac{1}{\lambda}}$ contains a neighborhood $\Omega''$ of $p_0$ of the form $\Omega'' = \{w \geq w_0-1, \sum_\mu (y^\mu)^2 < 2\}$. We recall that the metric $g$ on $D_{\frac{1}{\lambda}}$ is uniformly equivalent to the hyperbolic metric $h = dw^2 + e^{2w} \sum_\mu (dy^\mu)^2$. We set
\begin{equation}\label{eqDefOmega3}
\Omega^{(3)} = \{w \geq w_0, \sum_\mu (y^\mu)^2 < 1\}. 
\end{equation}
Hence, for any n-tuple $(x^1, \ldots, x^n)$ such that $\sum_\mu (y^\mu)^2 < 1$ and any $w_1 \geq w_0$, Lemma \ref{lmMovingWindow1} yields:
\begin{align*}
\left\|\hess(y^\mu)\right\|_{L^p_t(\Omega_{w_1, x}(1), h)}
  \leq & \int_{w_0}^{w_1} e^{\left(\frac{n}{p}-2\right)(w_1-w)} \left\|\partial_0 \hess(y^\mu)\right\|_{L^p_t(\Omega_{w, x}(1), h)} dw\\
  & + e^{\left(\frac{n}{p}-2\right)(w_1-w_0)} \left\|\hess(y^\mu)\right\|_{L^p_t(\Omega_{w_0, x}(1), h)}.
\end{align*}

Taking the tangential norm of Equation \eqref{eqEstimateHessTT}, we get that there exists a constant $C > 0$ independent of $(x^1, \ldots, x^n)$ such that \[\left\|g^{ij} \nabla_i w \nabla_j \hess (y^\mu) + 2 \hess (y^\mu)\right\|_{L^p_t(\Omega_{w, x}(1), h)} \leq C e^{-(1+a)w}.\]

From our previous calculations, we infer:
\begin{eqnarray*}
\left\|\partial_0 \hess(y^\mu)\right\|_{L^p_t(\Omega_{w, x}(1), h)}
  & \leq & \left\|\nabla_0 \hess(y^\mu) + 2 \hess (y^\mu)\right\|_{L^p_t(\Omega_{w, x}(1), h)}\\
  &	 & + C \left(e^{-aw} \left\| \hess(y^\mu)\right\|_{L^p_t(\Omega_{w, x}(1), h)} + e^{-(2a+1)w}\right)\\
  & \leq & \left\|\frac{1}{N^2} \left(g^{0i} \nabla_i \hess(y^\mu) - g^{0\nu} \nabla_\nu \hess(y^\mu)\right) + 2 \hessdd{\alpha}{\beta} y^\mu\right\|_{L^p_t(\Omega_{w, x}(1), h)}\\
  &	 & + C \left(e^{-aw} \left\| \hess(y^\mu)\right\|_{L^p_t(\Omega_{w, x}(1), h)} + e^{-(2a+1)w}\right)\\
  & \leq & \left\|\frac{1}{N^2} \left(g^{ij} \nabla_i w \nabla_j \hessdd{\alpha}{\beta} y^\mu + 2 \hessdd{\alpha}{\beta} y^\mu - g^{0\nu} \nabla_\nu \hess(y^\mu)\right) \right.\\
  &	 & \left. + \left(2 - \frac{2}{N^2}\right) \hessdd{\alpha}{\beta} y^\mu\right\|_{L^p_t(\Omega_{w, x}(1), h)}\\
  &      & + C \left(e^{-aw} \left\| \hess(y^\mu)\right\|_{L^p_t(\Omega_{w, x}(1), h)} + e^{-(2a+1)w}\right)\\
  & \leq & \left\|\frac{1}{N^2} g^{0\nu} \nabla_\nu \hess(y^\mu)\right\|_{L^p_t(\Omega_{w, x}(1), h)}\\
  &      & + C \left(e^{-aw} \left\| \hess(y^\mu)\right\|_{L^p_t(\Omega_{w, x}(1), h)} + e^{-(a+1)w}\right)\\
\end{eqnarray*}
There only remains to estimate $g^{0\nu} \nabla_\nu \hess(y^\mu)$. From the naive estimate \eqref{eqRoughEstimateHessianXp}, we obtain 
\begin{eqnarray*}
\left\|\frac{1}{N^2} g^{0\nu} \nabla_\nu \hess(y^\mu)\right\|_{L^p_t(\Omega_{w, x}(1), h)}
  & \leq & C \left\|g^{0\nu} \nabla_\nu \hess(y^\mu)\right\|_{L^p_t(\Omega_{w, x}(1), h)}\\
  & \leq & C \sum_\nu \left\|g^{0\nu} \nabla_\nu \hess(y^\mu)\right\|_{L^p_t(\Omega_{w, x}(1), h)}\\
  & \leq & C e^{-(a+1)w} \left\|\partial_\nu\right\|_{L^\infty_t(\Omega_{w, x}(1), h)} \left\|\nabla\hess(y^\mu)\right\|_{L^p_t(\Omega_{w, x}(1), h)}\\
  & \leq & C e^{-(a+1)w}.
\end{eqnarray*}

Thus, for any $w_1 \geq w_0$,
\begin{equation}\label{eqPartial0Hess}
\left\|\partial_0 \hess(y^\mu)\right\|_{L^p_t(\Omega_{w, x}(1), h)} \leq C \left(e^{-aw} \left\| \hess(y^\mu)\right\|_{L^p_t(\Omega_{w, x}(1), h)} + e^{-(a+1)w}\right).
\end{equation}

From Lemma \ref{lmGronwall}, we obtain:

\[\left\|\hess(y^\mu)\right\|_{L^p_t(\Omega_{w, x}(1), h)} \leq\left\lbrace \begin{array}{ll}
  C e^{-(a+1)w} & \text{if } a < 1 - \frac{n}{p}\\
  C e^{\left(2 - \frac{n}{p}\right)w} & \text{if } a > 1 - \frac{n}{p}.
\end{array}\right.\]

Remark that this formula gives the expected decay rate for if $0 < a < 1$ and $p > \frac{n}{1-a}$ while if $a > 1$, we only get that $\left|\hessbar(y^\mu)\right|_g = O(e^{-(2-\epsilon)w})$ for any $\epsilon > 0$ (this is not yet true since we only estimated the $L^p$-norm of $\hessbar(y^\mu)$ in balls). To remedy this, we have to be more careful in the end of the proof of Lemma \ref{lmMovingWindow1} (see the remark after the proof of Lemma \ref{lmMovingWindow1}). The problem is not due to the term in the integral but to the boundary term \[e^{\left(\frac{n}{p}-2\right)(w_1-w_0)} \left\|\hess(y^\mu)\right\|_{L^p_t(\Omega_{w_0, x}(1), h)}.\] We return to Estimate \eqref{eqMovingWindow1} and replace $w + B_{w_1, x}$ by $\Omega_{w, x}(1)$ only in the integral:
\[\begin{aligned}
\left\|\hess(y^\mu)\right\|_{L^p_t(\Omega_{w_1, x}(1), h)}
  \leq & \int_{w_0}^{w_1} e^{\left(\frac{n}{p}-2\right)(w_1-w)} \left\|\partial_0 \hess(y^\mu)\right\|_{L^p_t(\Omega_{w, x}(1), h)} dw\\
       & + e^{\left(\frac{n}{p}-2\right)(w_1-w_0)} \left\|\hess(y^\mu)\right\|_{L^p_t(w_0-w_1+B_{w_1, x}, h)}.
\end{aligned}\]

We now remark that $\left|\hess(y^\mu)\right|_g$ is uniformly bounded on $\Omega''$, so
\begin{eqnarray*}
\left\|\hess(y^\mu)\right\|_{L^p_t(w_0+B_{w_1, x}, h)}
  & \leq & \left\|\hess(y^\mu)\right\|_{L^\infty(\Omega'',h)} \left[\vol_h \left(w_0-w_1+B_{w_1, x}\right)\right]^\frac{1}{p}\\
  & \leq & C \left\|\hess(y^\mu)\right\|_{L^\infty(\Omega'',g)} e^{\frac{n}{p} (w_0-w_1)}.
\end{eqnarray*}
This leads to the following improved inequality:
\[\begin{aligned}
\left\|\hess(y^\mu)\right\|_{L^p_t(\Omega_{w_1, x}(1), h)}
  \leq & \int_{w_0}^{w_1} e^{\left(\frac{n}{p}-2\right)(w_1-w)} \left\|\partial_0 \hess(y^\mu)\right\|_{L^p_t(\Omega_{w, x}(1), h)} dw\\
       & + e^{-2(w_1-w_0)} \left\|\hess(y^\mu)\right\|_{L^\infty(\Omega'', g)}.
\end{aligned}\]
This permits us to strengthen our previous estimate when $a > 1$ and remove assumptions on $p > n+1$ when $0 < a < 1$:
\[\left\|\hess(y^\mu)\right\|_{L^p_t(\Omega_{w, x}(1), h)} \leq\left\lbrace \begin{array}{ll}
  C e^{-(a+1)w} & \text{if } a < 1\\
  C e^{-2w} & \text{if } a > 1.
\end{array}\right.\]
We finally note that \[\hessbardd{\alpha}{\beta} y^\mu = \hessdd{\alpha}{\beta} y^\mu - \left\langle dw, dy^\mu\right\rangle g_{\alpha\beta},\] and in particular \[\left\|\hessbar(y^\mu)\right\|_{L^p_t(\Omega_{w, x}(1), h)} = \left\|\hess(y^\mu)\right\|_{L^p_t(\Omega_{w, x}(1), h)} + O(e^{-(a+1)w}).\]

We summarize all the results of this subsection in the following lemma:
\begin{lemma}[$L^p$-estimates for $\hessbar(y^\mu)$]\label{lmEstimatesHess1}
The following estimates hold:
\[\begin{aligned}
\hessbardd{0}{0} y^\mu & = O(e^{-(a+1)w})\\
\hessbardd{0}{\alpha} y^\mu & = O(e^{-aw})\\
\left\|\hessbar(y^\mu)\right\|_{L^p_t(\Omega_{w, x}(1), h)} & \leq \left\lbrace \begin{array}{ll}
  C e^{-(a+1)w} & \text{if } a < 1\\
  C e^{-2w} & \text{if } a > 1
\end{array}\right.
\end{aligned}\]
for any $p \in (n+1; \infty)$ and any $(w, x) \in \Omega^{(3)}$. In particular
\[\hessbar(y^\mu) \in \left\lbrace \begin{array}{ll}
  X^{0, p}_{a+1}(\Omega^{(3)}(1)) & \text{if } a < 1,\\
  X^{0, p}_2 (\Omega^{(3)}(1)) & \text{if } a > 1.
\end{array}\right.\] (See Equation \eqref{eqDefOmegaR} for the definition of $\Omega^{(3)}(1)$).
\end{lemma}

\subsection{$W^{1, p}$-estimates for $\hessbar(y^\mu)$}\label{secW1pHessbar}
In this subsection, we prove the following result:
\begin{prop}\label{propEstimatesHess2}
The following estimates hold for $\nablabar~\hessbar (y^\mu)$:
\begin{itemize}
 \item If $0 < a < 1$, for any $p \in (n+1; \infty)$, there exists $r > 0$ such that \[\nablabar~\hessbar (y^\mu) \in X^{0, p}_{a+1}(\Omega^{(3)}(r)).\]
 \item If $1 < a < 2$, for any $p \in (n+1; \infty)$ and any $\mu \in (0; a-1)$, there exists $r > 0$ such that \[\nablabar~\hessbar (y^\mu) \in X^{0, p}_{\mu+2}(\Omega^{(3)}(r)).\]
\end{itemize}
\end{prop}

We start from the following formula:
\begin{equation}\label{eqCodazziRiembar}
\nablabar_i \hessbardd{j}{k} y^\mu - \nablabar_j \hessbardd{i}{k} y^\mu = -\riembaruddd{l}{k}{i}{j} \nablabar_l y^\mu,
\end{equation}
where $\riembar$ is the Riemann tensor of the metric $\gbar$ (see \cite[Theorem 1.159]{Besse}):
\begin{eqnarray*}
\riembardddd{i}{j}{k}{l}
  & = & t^{-2} \left(\riemdddd{i}{j}{k}{l} + \left( \frac{\hess(t)}{t} \kulk g\right)_{ijkl} - \frac{1}{2} \left|\frac{dt}{t}\right|_g^2 \left(g \kulk g\right)_{ijkl}\right)\\
  & = & t^{-2} \left(\tdddd{(R+K)}{i}{j}{k}{l} + \left( \frac{T}{t} \kulk g\right)_{ijkl} - \frac{1}{2} \left(\left|dw\right|_g^2 - 1\right) \left(g \kulk g\right)_{ijkl}\right)\\
\riembaruddd{i}{j}{k}{l}
  & = & \gbar^{ia} t^{-2} \left(\tdddd{E}{a}{j}{k}{l} + \left( \frac{T}{t} \kulk g\right)_{ajkl} - \frac{1}{2} \left(\left|dw\right|_g^2 - 1\right) \left(g \kulk g\right)_{ajkl}\right)\\
  & = & g^{ia} \left(\tdddd{E}{a}{j}{k}{l} + \left( \frac{T}{t} \kulk g\right)_{ajkl} - \frac{1}{2} \left(\left|dw\right|_g^2 - 1\right) \left(g \kulk g\right)_{ajkl}\right).
\end{eqnarray*}

So we see that the $(3, 1)$-Riemann tensor satisfies the following estimate \[\left|\riembar\right|_g = O(e^{-a w}).\] In particular, the norm of the right hand side of Equation \eqref{eqCodazziRiembar} with respect to the metric $g$ is $O(e^{-(a+1)w})$. From the conformal transformation law of the Laplacian, the trace of $\hessbar(y^\mu)$ is given by \[\gbar^{ij} \hessbardd{i}{j} y^\mu = \Deltabar y^\mu = -(n-1) t^2 g(dw, dy^\mu).\] So the traceless part of $\hessbar(y^\mu)$ is given by
\[\begin{aligned}
\tlhessbar(y^\mu) & = \hessbar(y^\mu) + \frac{n-1}{n+1} t^2 g(dw, dy^\mu) \gbar\\
 & = \hessbar(y^\mu) + \frac{n-1}{n+1} g(dw, dy^\mu) g.
\end{aligned}\]
Equation \eqref{eqCodazziRiembar} can be written in the following form:
\begin{equation}\label{eqCodazziRiembar2}
\begin{aligned}
\nablabar_i \tlhessbardd{j}{k} y^\mu - \nablabar_j \tlhessbardd{i}{k} y^\mu =
  & -\riembaruddd{l}{k}{i}{j} \nablabar_l y^\mu + 2\frac{n-1}{n+1} g(dw, dy^\mu)\left(g_{jk}\nabla_i w - g_{ik} \nabla_j w\right)\\
  & + \frac{n-1}{n+1} \left(g_{jk}\nabla_i g(dw, dy^\mu) - g_{ik} \nabla_j g(dw, dy^\mu)\right),
\end{aligned}
\end{equation}
From Estimate \eqref{lmEstDtDphi}, $g(dw, dy^\mu)$ belongs to $X^{2, p}_{a+1}$. We remark that Equation \eqref{eqCodazziRiembar} is unchanged if we multiply the metric $\gbar$ by some constant. Hence, if $p_0$ is any point in $\Omega^{(3)}$, the metric $t^2(p_0) \gbar = \frac{t^2(p_0)}{t^2} g$ is uniformly equivalent to the metric $g$ on $B_{r_H}(p_0)$ and is $W^{2, p}$-controlled in some harmonic charts for the metric $g$ (as given by Theorem \ref{thmHarmRadCtrl}).  Thus, applying Proposition \ref{propEllipticRegCodazzi} to Equation \eqref{eqCodazziRiembar}, we get that for any $p \in (n+1; \infty)$:

\begin{equation}\label{eqEstimateHess0}
\hessbar(y^\mu) \in \left\lbrace \begin{array}{ll} X^{1, p}_{a+1}(\Omega^{(3)}(1)) & \text{ if $0 < a < 1$,}\\ X^{1, p}_2(\Omega^{(3)}(1)) & \text{ if $1 < a < 2$.}\end{array}\right. 
\end{equation}

\begin{rk}\label{rkLinftyEstimateHess}
Applying the Sobolev embedding theorem, we get that $\left|\hessbar(y^\mu)\right|_g = O(e^{-(a+1)w})$ if $0 < a < 1$ and $\left|\hessbar(y^\mu)\right|_g = O(e^{-2w})$ if $1 < a < 2$.
\end{rk}

Estimate \ref{eqEstimateHess0} proves the proposition if $0 < a < 1$ while if $1 < a < 2$ this naive application of elliptic regularity is not enough. To remedy this, we would be tempted to apply Lemma \ref{lmMovingWindow1} to $\nablabar~\hessbar(y^\mu)$. However this strategy would require controling the $X^{3, p}_{a+1}$-norm of $\left\langle dw, dy^\mu\right\rangle$ so a certain assumption on the covariant derivative of the Ricci tensor would be needed (see Formula \eqref{eqDtDphi}). We shall instead apply Lemma \ref{lmMovingWindow2}: subtract the average value of $\hessbar (y^\mu)$ over some $\Omega_{w, y}(r)$. The proof then consists in applying iteratively the following lemma:

\begin{lemma}\label{lmIteration}
Assume that $\nablabar \hessbar (y^\mu) \in X^{0, q}_{2+b}(\Omega^{(3)}(r)$ for some $r > 0$, some $b \in [0; a-1)$ and some $q \in (1+\frac{n}{p}; \infty)$ then $\nablabar \hessbar (y^\mu) \in X^{0, q}_{2+b'}(\Omega^{(3)}\left(\frac{r}{2}\right)$ where \[b' = \frac{a-1-\frac{n}{q}}{a-b}.\]
\end{lemma}

Before giving the proof of Lemma \ref{lmIteration}, we indicate how it implies Proposition \ref{propEstimatesHess2}. Remark that it suffice to prove that $\nablabar \hessbar (y^\mu) \in X^{0, q}_{2+\mu'}$ for $q \geq p$ large enough and $\mu' \geq \mu$ since $X^{0, q}_{2+\mu'}((\Omega^{(3)}(r)) \subset X^{0, p}_{2+\mu}((\Omega^{(3)}(r))$. From Lemma \ref{lmEstimatesHess1}, we know that the assumption of the lemma is satisfied for $b = b_0 = 0$ and $r=1$. By induction, we get that $\nablabar \hessbar (y^\mu) \in X^{0, q}_{2+b_k}\left(\Omega^{(3)}(\frac{1}{2^k})\right)$ where $b_k$ is defined by \[b_k = \frac{a-1-\frac{n}{q}}{a-b_{k-1}}.\] If $q$ is large enough, we can assume that $a > 1 + \frac{n}{q}$. It is then a standard exercise to prove that $(b_k)_k$ is an increasing sequence converging to $b_\infty$ where \[b_\infty = \frac{a - \sqrt{(2-a)^2 + \frac{n}{q}}}{2} = a-1 - O\left(\frac{1}{q}\right).\] Hence, choosing $k$ and $q$ large enough such that $b_k \geq \mu$, we get that $\nablabar \hessbar (y^\mu) \in X^{0, q}_{2+b_k} \subset X^{0, p}_{2+\mu}$. This proves Proposition \ref{propEstimatesHess2}.

\begin{proof}[Proof of Lemma \ref{lmIteration}] Let $\alpha$ and $\beta$ be two tangential indices and $(w_1, y)$ be such that $\Omega_{w_1, y}(r) \subset \Omega^{(3)}(r)$. Let $F = \hessbardd{\alpha}{\beta} y^\mu$. Using the notations of Lemma \ref{lmMovingWindow2}, we estimate $\left\|F - \Ftil\right\|_{L^p(\Omega_{w_1, y}(r))}$. From Estimate \eqref{eqPartial0Hess}, we infer \[\left\|\partial_0 F\right\|_{L^p(\Omega_{w, y}(r))} = O(e^{-(a-1)(w-w_0)}).\] We estimate next $\left\|d F\right\|_{L^p(\Omega_{w_0, y}(1))}$: \[\nablabar_i \hessbardd{\alpha}{\beta} y^\mu = \partial_i \hessbardd{\alpha}{\beta} y^\mu - \Gambar^k_{i\alpha} \hessbardd{k}{\beta} y^\mu - \Gambar^k_{i\beta} \hessbardd{\alpha}{k} y^\mu.\] Remark that $\Gambar^\nu_{ij} = \hessbardd{i}{j} y^\nu$ and $\Gambar^0_{ij} = \hessbardd{i}{j} w$, hence, from Proposition \ref{lmEstimatesHessRho} and Remark \ref{rkLinftyEstimateHess}, it is straightforward to prove that
\[\begin{aligned}
\left\|d F\right\|_{L^p(\Omega_{w, y}(r))}
  & \leq \left\|\nablabar \hessbardd{\alpha}{\beta} y^\mu\right\|_{L^p(\Omega_{w, y}(r))} + O(e^{-w})\\
  & \leq \left\|\left|\nablabar \hessbar (y^\mu)\right| \left|\partial_\alpha\right| \left|\partial_\beta\right|\right\|_{L^p(\Omega_{w, y}(r))} + O(e^{-w})\\
  & = O(e^{-b (w-w_0)}).   
\end{aligned}\]
For any $w_1 \geq w_0$, we set \[w = w_0 + \frac{w_1-w_0}{a-b}.\] Since $a > b+1$, $w_0 \leq w \leq w_1$. We apply Lemma \ref{lmMovingWindow2} between $w$ and $w_1$:
\[\begin{aligned}
\left\|F - \Ftil(w_1, x)\right\|_{L^p(\Omega_{w_1, y}(r), h)} \leq
  & c_1 e^{\frac{n}{p}(w_1 - w)} \left\|\partial_0 F \right\|_{L^p(\Omega_{w, y}(r), h)}\\
  & + c_2 e^{-\left(1-\frac{n}{p}\right)(w_1 - w)}\left\|d F \right\|_{L^p(\Omega_{w, y}(r), h)}\\
  & + 2 \int_{w}^{w_1} e^{\frac{n}{p}(w_1-w')} \left\|\partial_0 F\right\|_{L^p(\Omega_{w', y}(r), h)} dw.
\end{aligned}\]
Using the estimates for the norm of the terms appearing in this inequality, we conclude that there exists a constant $C > 0$ such that
\begin{equation}\label{eqEstimateFav}
\left\|F - \Ftil(w_1, y)\right\|_{L^p(\Omega_{w_1,x}(r), h)} \leq C e^{\left(\frac{n}{q} \frac{a-b-1}{a-b} - \frac{a-1}{a-b}\right)(w_1-w_0)} \leq C e^{- b'(w_1-w_0)}. 
\end{equation}

The idea is now to replace $\hess(y^\mu)$ in Equation \eqref{eqCodazziRiembar} by $\hess(y^\mu)-\widetilde{\hess(y^\mu)}$ which has the right decay according to the previous estimate and apply Proposition \ref{propEllipticRegCodazzi} to prove the lemma. We set $\rho_1 = e^{-w_1}$ and we introduce coordinates $z^0 = \rho_1^{-1} \rho$, $z^\mu = \rho_1^{-1} y^\mu$ and metric $g' = \rho_1^{-2} \gbar$. We remark that the Levi-Civita connection of $g'$ equals that of $g$ and that in this coordinate system \[\Omega_{w_1, y}(r) = \left\{((z')^0, (z')^\mu) | \rho_1^{-1} e^{-r} < (z')^0 < \rho_1^{-1} e^r,~ \sqrt{\sum_\mu \left[(z')^\mu-z^\mu\right]^2} < r \rho_1^{-1} (z')^0\right\}.\] We denote with a prime the components of tensors with respect to coordinates $z'$. Let \[H_{i'j'} = \hessbardd{i'}{j'}y^\mu - \widetilde{\hessbardd{i'}{j'}y^\mu}.\] We rewrite Equation \eqref{eqCodazziRiembar} as follows:
\begin{equation}\label{eqCodazzi3}
\nablabar_{i'} H_{j'k'} - \nablabar_{j'} H_{i'k'} = \Gambar^{l'}_{i'k'} \widetilde{\hessbardd{l'}{j'} y^\mu} - \Gambar^{l'}_{j'k'} \widetilde{\hessbardd{i'}{l'} y^\mu} - \riembaruddd{l'}{k'}{i'}{j'} \nablabar_{l'} y^\mu. 
\end{equation}
From Lemmas \ref{lmEstimatesHessRho}, \ref{lmEstimatesHess1}, Remark \ref{rkLinftyEstimateHess} and the fact that $\Gambar_{i'j'}^{k'} = -\hessbardd{i'}{j'} z^{k'}$, it is easy to argue that the $L^p$-norm of the right-hand side of Equation \eqref{eqCodazzi3} over $\Omega_{w_1, y}(r)$ is $O(e^{-(a+1)w_1})$. To apply Proposition \ref{propEllipticRegCodazzi} to Equation \eqref{eqCodazzi3}, we only need to show that the trace of $H$ is well controled in $W^{1,q}$-norm over $\Omega_{w_1, y}(r)$. We remark that from Estimate \eqref{eqEstimateFav} and Lemma \ref{lmEstimatesHess1}, \[\left\|(g')^{i' j'} H_{i'j'}\right\|_{L^p(\Omega_{w_1, r}(r), g')} \leq C e^{-b' w_1}.\] Further,
\[\partial_{k'} \left[(g')^{i' j'} H_{i'j'}\right] = \rho_1^2 \partial_{k'} \Deltabar y^\mu + \left(\partial_{k'} (g')^{i' j'}\right) \widetilde{\hessbardd{i'}{j'} y^\mu}.\]
Remark that the derivatives of $g'$ are up to a constant given by the Hessians of $\rho$ and $y^\mu$ which have been estimated in Lemma \ref{lmEstimatesHessRho} and Remark \ref{rkLinftyEstimateHess}. It is then straightforward to see that \[\left\|(g')^{i' j'} H_{i'j'}\right\|_{W^{1,p}(\Omega_{w_1, r}(r), g')} \leq C e^{-b' w_1}.\] We are now in a position to apply elliptic regularity (Proposition \ref{propEllipticRegCodazzi}) to Equation \eqref{eqCodazzi3}: \[\left\|H\right\|_{W^{1, q}(\Omega_{w_1, y}(\frac{r}{2}))} \leq C \left(\left\|H\right\|_{L^q(\Omega_{w_1, y}(r))} + e^{-(a+1)w_1}\right),\] where we used the fact that the right hand side of Equation \eqref{eqCodazzi3} is $O(e^{-(a+1)w_1})$. Estimate \ref{eqEstimateFav} and Lemma \ref{lmEstimatesHess1} yield the following estimate for $H$: \[\left\|H\right\|_{L^q(\Omega_{w_1, y}(r))} \leq C e^{-b' w_1}\] proving that \[\left\|H\right\|_{W^{1, q}(\Omega_{w_1, y}(\frac{r}{2}))} \leq C e^{-b' w_1}.\] We finally remark that, as we computed previously in this proof, $\nablabar H$ equals $\nablabar \hessbar(y^\mu)$ up to some terms decaying fast with respect to $w_1$. Hence, \[\left\|\nablabar \hessbar(y^\mu)\right\|_{L^q(\Omega_{w_1, y}(\frac{r}{2}))} \leq C e^{-b' w_1}.\]
\end{proof}

\section{Boundary regularity}\label{secBoundaryReg}
% Boundary regularity

In this section, we end the proof of Theorem \ref{thmMain}. We first give another proof for the case $0 < a < 1$. From Propositions \ref{lmEstimatesHessRho}, \ref{propEstimatesHess2} and the Sobolev embedding theorem, we get that on $\Omega^{(3)}(r)$ \[\left|\hessbar(\rho)\right|_g, \left|\hessbar(y^\mu)\right|_g = O(e^{-(1+a)w}).\] In particular, denoting $\bar{i}, \bar{j}, ...$ the compactified coordinates: $y^{\bar{0}} = \rho, y^{\bar{\mu}} = y^\mu$ for any $\mu = 1, \ldots, n$, this means that the components of the Hessians of the coordinate functions satisfy \[\hessbar(y^{\bar{i}}) = O(\rho^{a-1}).\] From the fact that the metric $\gbar$ is uniformly equivalent to the flat metric on $\overline{\Omega^{(3)}(r)} \subset \Mbar$, the derivatives of the metric satisfy \[\partial_{\bar{k}} \gbar_{\bar{i}\bar{j}} = O(\rho^{a-1}).\] Applying Lemma \ref{lmBridge} (or \cite[Lemma 3.8]{BahuaudGicquaud}), we conclude that $\gbar \in C^{0, a}(\overline{\Omega^{(3)}})$. A similar reasoning proves that if $1 < a < 2$, the metric $\gbar$ is Lipschitz continuous on $\overline{\Omega^{(3)}(r)}$.\\

We finally turn our attention to H\"older regularity for the Christoffel symbols of $\gbar$, which is equivalent to regularity for the partial derivatives of $\gbar$. We remark that \[ \nablabar_{\bar{l}} \hessbardd{\bar{i}}{\bar{j}} y^{\bar{k}} = - \partial_{\bar{l}} \Gambar_{\bar{i}\bar{j}}^{\bar{k}} + \Gambar_{\bar{l}\bar{i}}^{\bar{m}} \Gambar_{\bar{m}\bar{j}}^{\bar{k}} + \Gambar_{\bar{l}\bar{j}}^{\bar{m}} \Gambar_{\bar{i}\bar{m}}^{\bar{k}}.\] From the fact that the Christoffel symbols of $\gbar$ are uniformly bounded on $\overline{\Omega^{(3)}(r)}$, we infer that
\[\begin{aligned}
  \left\|\Gambar_{\cdot\bar{i}}^{\bar{m}} \Gambar_{\bar{m}\bar{j}}^{\bar{k}} \right\|_{L^p(\Omega_{w, y}(r), h)}
    & = \left(\int_{\Omega_{w, y}(r)} \left[h^{\bar{a}\bar{b}} \left(\Gambar_{\bar{a}\bar{i}}^{\bar{m}} \Gambar_{\bar{m}\bar{j}}^{\bar{k}}\right) \left(\Gambar_{\bar{b}\bar{i}}^{\bar{m}} \Gambar_{\bar{m}\bar{j}}^{\bar{k}}\right)\right]^{\frac{p}{2}} d\mu_h\right)^{\frac{1}{p}}\\
    & = e^{-w} \left(\int_{\Omega_{w, y}(r)} \left[\delta^{\bar{a}\bar{b}} \left(\Gambar_{\bar{a}\bar{i}}^{\bar{m}} \Gambar_{\bar{m}\bar{j}}^{\bar{k}}\right) \left(\Gambar_{\bar{b}\bar{i}}^{\bar{m}} \Gambar_{\bar{m}\bar{j}}^{\bar{k}}\right)\right]^{\frac{p}{2}} d\mu_h\right)^{\frac{1}{p}}\\
    & \leq C e^{-w},
\end{aligned}\]
where $\delta^{\bar{a}\bar{b}} = 1$ iff $\bar{a} = \bar{b}$ and is zero otherwise. Thus,
\[\begin{aligned}
 \left\|\partial \Gambar_{\bar{i}\bar{j}}^{\bar{k}} \right\|_{L^p(\Omega_{w, y}(r), h)}
   & = \left\|\nablabar \hessbardd{\bar{i}}{\bar{j}} y^{\bar{k}}\right\|_{L^p(\Omega_{w, y}(r), h)} + O(e^{-w})\\
   & \leq \left\|\nablabar \hessbar y^{\bar{k}}\right\|_{L^p(\Omega_{w, y}(r), h)} \left\|\partial_{\bar{i}}\right\|_{L^\infty(\Omega_{w, y}(r), h)} \left\|\partial_{\bar{j}}\right\|_{L^\infty(\Omega_{w, y}(r), h)}+ O(e^{-w})\\
   & \leq C e^{-\mu w},
\end{aligned}\]
where we used the fact that $\left\|\partial_{\bar{i}}\right\|_{L^\infty(\Omega_{w, y}(r), h)} = O(e^w)$, Lemma \ref{lmEstimatesHessRho} and Proposition \ref{propEstimatesHess2}. This shows that $\partial \Gambar_{\bar{i}\bar{j}}^{\bar{k}} \in X^{0, p}_\mu$. Selecting $p$ large enough so that $\mu \leq 1 - \frac{n+1}{p}$ and applying Lemma \ref{lmBridge}, we immediately get that the Christoffel symbols $\Gambar_{\bar{i}\bar{j}}^{\bar{k}}$ of $\gbar$ in the coordinate system $(\rho, y^1, \ldots, y^n)$ belong to $C^{0, \mu}(\overline{\Omega^{(3)}})$. This ends the proof of Theorem \ref{thmMain}.

\bibliographystyle{amsalpha}
\bibliography{biblio}
\end{document}